\documentclass[10pt, a4paper, leqno]{amsart}
\usepackage[utf8]{inputenc} 
\usepackage[english]{babel} 
\usepackage[usenames, dvipsnames]{xcolor}
\usepackage[a4paper, tmargin=1.0in,bmargin=1.0in, lmargin=1in, rmargin=1in]{geometry} 
\usepackage{bm} 
\usepackage{amsmath} 
\usepackage{amsthm}
\usepackage{amssymb}
\usepackage{dsfont}
\usepackage{cases} 
\usepackage{tikz} 
\usetikzlibrary{arrows.meta, decorations.markings, calc}
\usepackage{stmaryrd} 
\usepackage{multirow} 
\usepackage{lipsum} 
\usepackage{wrapfig} 

\usepackage{breakcites} 
\usepackage{mathtools}   
\usepackage{cases}
\usepackage{float}
\usepackage{esint}

\usepackage{pgfplots}
\pgfplotsset{compat=1.15}

\usepackage[colorlinks=true,linkcolor=BrickRed,citecolor=RoyalBlue]{hyperref} 
\usepackage{cleveref} 

\theoremstyle{plain}
\newtheorem{proposition}{Proposition}%
\newtheorem{lemma}{Lemma}%
\newtheorem*{theorem*}{Theorem}
\newtheorem{informalresult}{Property}%
\theoremstyle{definition}
\newtheorem{assumption}{Assumption}%
\theoremstyle{remark}
\newtheorem{remark}{Remark}%
\newtheorem{example}{Example}%

\newcommand{\strong}[1]{\emph{#1}}
\newcommand{\confer}{\emph{cf.}}
\newcommand{\idEst}{\emph{i.e.}}

\newcommand{\advectionVelocity}{a}
\newcommand{\solutionCauchyProblem}{u}
\newcommand{\definitionEquality}{:=}
\newcommand{\indexSpace}{j}
\newcommand{\timeVariable}{t}
\newcommand{\spaceVariable}{x}
\newcommand{\initial}{\circ}

\newcommand{\timeGridPoint}[1]{\timeVariable^{#1}}
\newcommand{\indexTime}{n}
\newcommand{\spaceStep}{\Delta \spaceVariable}
\newcommand{\timeStep}{\Delta \timeVariable}
\newcommand{\relatives}{\mathbb{Z}}
\newcommand{\reals}{\mathbb{R}}
\newcommand{\complex}{\mathbb{C}}

\newcommand{\naturals}{\mathbb{N}}

\newcommand{\finalTime}{T}
\newcommand{\conservedMoment}{u}
\usepackage{mathrsfs}  
\newcommand{\courantNumber}{\mathscr{C}}

\newcommand{\bigO}[1]{\mathcal{O}(#1)}
\newcommand{\globalTruncationErrorNotNormalized}{e}
\newcommand{\globalTruncationError}{\epsilon}
\newcommand{\fourierShift}{\kappa}
\newcommand{\timeShiftOperator}{z}

\newcommand{\differential}[1]{\text{d}#1}
\newcommand{\zTransformed}[1]{\tilde{#1}}
\newcommand{\stableRoot}{\fourierShift_{\textnormal{s}}}
\newcommand{\unitCircle}{\mathbb{S}}
\newcommand{\unitDisk}{\mathbb{D}}
\newcommand{\closedUnitDisk}{\overline{\mathbb{D}}}
\newcommand{\neighborhoodInfinity}{\mathbb{U}}
\newcommand{\closedNeighborhoodInfinity}{\overline{\mathbb{U}}}
\newcommand{\coefficientStableSolution}{C_{\textnormal{s}}}

\newcommand{\doubleZTransformed}[1]{\check{#1}}
\newcommand{\firstGreenFunction}{\mathscr{F}}
\newcommand{\secondGreenFunction}{\mathscr{S}}
\newcommand{\fourierTransformed}[1]{\hat{#1}}
\newcommand{\frequency}{\xi}

\newcommand{\coefficientCornerScheme}{c}
\newcommand{\coefficientInitialBulkScheme}{s}
\newcommand{\coefficientEventualBoundaryScheme}{b}
\newcommand{\coefficientEventualBoundarySchemeOld}{\tilde{b}}

\newcommand{\branchPointAngle}{\vartheta_{\textnormal{BP}}}
\newcommand{\saddlePoint}{\timeShiftOperator_{\textnormal{SP}}}
\newcommand{\saddlePointAngle}{\vartheta_{\textnormal{SP}}}
\newcommand{\ksAtsaddlePointAngle}{\xi_{\textnormal{SP}}}
\newcommand{\ksAtsaddlePointAngleplain}{\xi}

\newcommand{\relaxationParameter}{\omega}

\newcounter{review}

\title[Long-time behavior of multi-step FD schemes with boundary]{Long-time behavior of multi-step Finite Difference schemes with boundary via steepest descent and analytic combinatorics}


\author{Thomas Bellotti}
\address{Université Paris-Saclay, CNRS, CentraleSupélec, Laboratoire EM2C \& Fédération de Mathématiques de CentraleSupélec, 91190, Gif-sur-Yvette, France}
\email{thomas.bellotti@centralesupelec.fr}

\author{Tommaso Tenna}
\address{Laboratoire J. A. Dieudonné, Université Côte d’Azur, CNRS, F-06108 Nice, France \& Dipartimento di Matematica ``Guido Castelnuovo'', Sapienza Università di Roma, 00185 Rome, Italy}
\email{tommaso.tenna@uniroma1.it}

\allowdisplaybreaks
\begin{document}

\maketitle

\begin{abstract}
    We demonstrate how steepest descent arguments and singularity analysis from analytic combinatorics allow for an accurate description of the behavior of linear numerical schemes---including the notorious leap-frog scheme---in presence of stable and unstable boundary conditions in the long-time limit.
\end{abstract}


\section{Introduction}

\strong{Steepest descent} techniques are widely used to precisely estimate the leading behavior of complex integrals over a contour $C$ of the form
\begin{equation}\label{eq:saddlePointIntegralAbstract}
    \int_C g(\timeShiftOperator)e^{\indexTime f(\timeShiftOperator)}\differential{\timeShiftOperator} 
    \qquad 
    \text{in the limit of \strong{large} }\indexTime,
\end{equation}
by invoking contour deformation to pass through---or very close to---critical points of $f$ along steepest descent directions.
In the present contribution, we discuss how these tools can be used to understand the behavior of \strong{linear multi-step Finite Difference approximations} in presence of a \strong{boundary}.

\subsection{Aims of the paper and motivation}

More precisely, the first aim of the paper, which has originally stimulated our study, is to understand the \strong{empirical order of convergence} and \strong{structure of the solution} of a non-dissipative bulk numerical scheme (\strong{e.g.}, the leap-frog scheme) once used together with an inconsistent scheme on the spatio-temporal corner.
Although this setting may appear unconventional at first sight, it naturally corresponds to the case of a simple lattice Boltzmann scheme made second-order accurate---a recent trend in the literature \cite{guillon2024stability, bellotti2025fourth, wissocq2025positive}---initialized ``at equilibrium'' \cite{bellotti2024initialisation}, and endowed with a first-order extrapolation of the lacking information at the boundary \cite{bellotti2025consistency}. 
For concreteness, consider
\begin{equation}
\label{eq:LinearAdvectionIntroduction}
    \begin{cases}
        \partial_{\timeVariable}\solutionCauchyProblem(\timeVariable, \spaceVariable) + \advectionVelocity \partial_{\spaceVariable}\solutionCauchyProblem(\timeVariable, \spaceVariable) = 0, \qquad \timeVariable > 0, \quad &\spaceVariable>0, \\
        \solutionCauchyProblem(0, \spaceVariable) = \solutionCauchyProblem^{\initial}(\spaceVariable), \qquad &\spaceVariable>0, 
    \end{cases}
\end{equation}
with $\advectionVelocity<0$, exact solution $\solutionCauchyProblem(\timeVariable, \spaceVariable) = \solutionCauchyProblem^{\initial}(\spaceVariable-\advectionVelocity\timeVariable)$, and smooth initial data $\solutionCauchyProblem^{\initial}$.
The approximation on a uniform time-space mesh of steps $\timeStep$-$\spaceStep$ of fixed ratio is taken as 
\begin{align}
    &\conservedMoment_{\indexSpace}^0 = \solutionCauchyProblem^{\initial}(\indexSpace\spaceStep), \qquad \indexSpace\in\naturals, \\
    &\conservedMoment_0^1 = \sum_{k\in\naturals}\coefficientCornerScheme_k \conservedMoment_k^0, \qquad \conservedMoment_{\indexSpace}^1 = \sum_{k \geq -1}\coefficientInitialBulkScheme_k \conservedMoment_{\indexSpace+k}^0, \quad \indexSpace \geq 1, \\
    \indexTime\geq 1 \qquad & \conservedMoment_0^{\indexTime + 1} = \sum_{k\in\naturals}\coefficientEventualBoundaryScheme_k \conservedMoment_k^{\indexTime} + \sum_{k\in\naturals}\coefficientEventualBoundarySchemeOld_k \conservedMoment_k^{\indexTime-1}, \qquad \conservedMoment_{\indexSpace}^{\indexTime + 1} = \conservedMoment_{\indexSpace}^{\indexTime - 1} +\courantNumber (\conservedMoment_{\indexSpace-1}^{\indexTime} - \conservedMoment_{\indexSpace+1}^{\indexTime}), \quad \indexSpace\geq 1,\label{eq:leapFrogEventually}
\end{align}
where the real sequences $(\coefficientCornerScheme_k)_k, (\coefficientInitialBulkScheme_k)_k, (\coefficientEventualBoundaryScheme_k)_k, (\coefficientEventualBoundarySchemeOld_k)_k$ are compactly supported.
We now assume that these coefficients fulfill the following order-constraints:
\begin{multline}\label{eq:orderConstraints}
    \underbrace{\sum_{k\in\naturals}\coefficientCornerScheme_k = 1 \quad 
        \text{and}\quad 
    \sum_{k\in\naturals}k \coefficientCornerScheme_k \neq -\courantNumber,}_{\text{0th-order time-space corner scheme}}
    \qquad 
    \underbrace{\sum_{k\geq -1}\coefficientInitialBulkScheme_k = 1 \quad 
        \text{and}\quad 
    \sum_{k\geq -1}k \coefficientInitialBulkScheme_k = -\courantNumber,}_{\text{at least 1st-order initial-time/space-bulk scheme}}
    \\ 
    \underbrace{\sum_{k\in\naturals} (\coefficientEventualBoundaryScheme_k+\coefficientEventualBoundarySchemeOld_k) = 1
        \quad\text{and}\quad
        \sum_{k\in\naturals} k(\coefficientEventualBoundaryScheme_k+\coefficientEventualBoundarySchemeOld_k) = -\courantNumber \Bigl(1+\sum_{k\in\naturals} \coefficientEventualBoundarySchemeOld_k\Bigr ).}_{\text{at least 1st-order eventual-time/boundary scheme}}
\end{multline}
The \strong{Courant number} is defined by $\courantNumber\definitionEquality \advectionVelocity \timeStep/\spaceStep$ and $\conservedMoment_{\indexSpace}^{\indexTime}$ has to be interpreted as an approximation of $\solutionCauchyProblem(\indexTime\timeStep, \indexSpace\spaceStep)$, where $\solutionCauchyProblem$ is the solution of \eqref{eq:LinearAdvectionIntroduction}.
The linearity of the problem entails that the \strong{global truncation error} $\globalTruncationErrorNotNormalized_{\indexSpace}^{\indexTime}\definitionEquality \solutionCauchyProblem(\indexTime\timeStep, \indexSpace\spaceStep) - \conservedMoment_{\indexSpace}^{\indexTime}$ fulfills
\begin{align*}
    &\globalTruncationErrorNotNormalized_{\indexSpace}^0 = 0, \qquad \indexSpace\in\naturals, \\
    &\globalTruncationErrorNotNormalized_0^1 = -\spaceStep \Bigl ( \courantNumber + \sum_{k\in\naturals} k \coefficientCornerScheme_k \Bigr )\frac{\differential{\solutionCauchyProblem^{\initial}(0)}}{\differential{\spaceVariable}}+ \bigO{\spaceStep^2}, \qquad \globalTruncationErrorNotNormalized_{\indexSpace}^1 = \bigO{\spaceStep^2}, \quad \indexSpace \geq 1, \\
    \indexTime\geq 1 \qquad & \globalTruncationErrorNotNormalized_0^{\indexTime + 1} = \sum_{k\in\naturals}\coefficientEventualBoundaryScheme_k \globalTruncationErrorNotNormalized_k^{\indexTime} + \sum_{k\in\naturals}\coefficientEventualBoundarySchemeOld_k \globalTruncationErrorNotNormalized_k^{\indexTime-1} + \bigO{\spaceStep^2}, \qquad \globalTruncationErrorNotNormalized_{\indexSpace}^{\indexTime + 1} = \globalTruncationErrorNotNormalized_{\indexSpace}^{\indexTime - 1} +\courantNumber (\globalTruncationErrorNotNormalized_{\indexSpace-1}^{\indexTime} - \globalTruncationErrorNotNormalized_{\indexSpace+1}^{\indexTime}) + \bigO{\spaceStep^3} , \quad \indexSpace\geq 1,
\end{align*}
thanks to \eqref{eq:orderConstraints}, following Taylor expansions\footnote{In the previous equations, we sloppily employ the notation $\bigO{\spaceStep^s}$ without precisely specifying any uniform character of these reminders. However, as the initial datum is supposed smooth, these terms can be made explicit and uniform, depending on the $W^{s, \infty}$ semi-norm of the initial datum.}.
The issue with the scheme is that, whenever $\frac{\differential{\solutionCauchyProblem^{\initial}}(0)}{\differential{\spaceVariable}} \neq 0$, the $L^2$ error at final time is empirically of order $\bigO{\spaceStep^{3/2}}$. 
This phenomenon was not observed in \cite[Chapter 6.3]{helie2023schema}, since the initial datum considered there has a vanishing derivative at $\spaceVariable = 0$.
We assume that $\frac{\differential{\solutionCauchyProblem^{\initial}(0)}}{\differential{\spaceVariable}}\neq 0$, and normalize the leading source of the global truncation error to one, neglecting all the other terms, which are at least $\bigO{\spaceStep^2}$.
We thus eventually consider
\begin{align}
    &\globalTruncationError_{\indexSpace}^0 = 0, \qquad \indexSpace\in\naturals, \label{eq:zeroInitialError} \\
    &\globalTruncationError_0^1 = 1, \qquad \globalTruncationError_{\indexSpace}^1 = 0, \quad \indexSpace \geq 1, \label{eq:schemeErrorInitial} \\
    \indexTime\geq 1 \qquad & \globalTruncationError_0^{\indexTime + 1} = \sum_{k\in\naturals}\coefficientEventualBoundaryScheme_k \globalTruncationError_k^{\indexTime} + \sum_{k\in\naturals}\coefficientEventualBoundarySchemeOld_k \globalTruncationError_k^{\indexTime-1}, \qquad \globalTruncationError_{\indexSpace}^{\indexTime + 1} = \globalTruncationError_{\indexSpace}^{\indexTime - 1} +\courantNumber (\globalTruncationError_{\indexSpace-1}^{\indexTime} - \globalTruncationError_{\indexSpace+1}^{\indexTime}), \quad \indexSpace\geq 1,\label{eq:schemeErrorEventual}
\end{align}
as a \strong{reliable model} for the global truncation error at leading order in $\spaceStep$.
Indeed, if stability holds, we expect $\globalTruncationErrorNotNormalized_{\indexSpace}^{\indexTime} = -\spaceStep  ( \courantNumber + \sum_{k\in\naturals} k \coefficientCornerScheme_k  )\frac{\differential{\solutionCauchyProblem^{\initial}(0)}}{\differential{\spaceVariable}} \globalTruncationError_{\indexSpace}^{\indexTime}+ \bigO{\spaceStep^2} $.
Accordingly, the remainder of this work is devoted to the analysis of $\globalTruncationError_{\indexSpace}^{\indexTime}$.
\begin{remark}[Stability]
    In the following, we do not claim anything on the \strong{stability} of the numerical scheme. 
    A full and rigorous stability analysis of the scheme would require accounting for the additional source terms that have been neglected to accurately control their impact on the norms numerical solution.
    Thus, the order of convergence with respect to the $L^p$ (with $1\leq p \leq\infty$) norms must be addressed from an \strong{empirical perspective}, given that the leap-frog scheme is known to be unstable on an infinite--periodic domain ---as it is the case for many dispersive schemes---for every $p\neq 2$, see \cite{trefethen1984lp}.
    Nevertheless, the work of \cite{estep1990boundedness} has shown that stability in $L^p$ norms with $p\neq 2$ can be expected for initial data of bounded variation, rendering such instabilities difficult to observe unless ad hoc-designed initial conditions are employed. 
\end{remark}

The second aim of this paper is to finely describe the \strong{structure of solutions} to the time-space recurrence relation \eqref{eq:zeroInitialError}--\eqref{eq:schemeErrorInitial}--\eqref{eq:schemeErrorEventual} (or analogous ones), independently on the consistency of the underlying schemes described by \eqref{eq:orderConstraints}, where the leading-order error originates from the time-space corner.
Quite the opposite, \eqref{eq:zeroInitialError}--\eqref{eq:schemeErrorInitial}--\eqref{eq:schemeErrorEventual} model the setting considered in \cite{bellotti2025stability}, where strong stability-instability (frequently known as GKS, for Gustafsson, Kreiss, and Sundstr{\"o}m \cite{gustafsson1972stability}) of boundary conditions was numerically showcased by considering boundary data set to one at initial time, and zero thereafter.

Regardless of the meaning that we assign to the solution of \eqref{eq:zeroInitialError}--\eqref{eq:schemeErrorInitial}--\eqref{eq:schemeErrorEventual}, we bring its study back to an integral of the form \eqref{eq:saddlePointIntegralAbstract}, where $g$ encodes the boundary condition depending on $\coefficientEventualBoundaryScheme_k$ and $\coefficientEventualBoundarySchemeOld_k$, whereas $f$ is determined solely by the bulk numerical scheme.
Two significant classes of points in the complex plane provide significant contributions to \eqref{eq:saddlePointIntegralAbstract}:
\begin{itemize}
    \item \strong{Saddle points} of $f$, namely $\saddlePoint\in\complex$ such that $f'(\saddlePoint) = 0$.
    Once these points are found, they are further classified as 
    \begin{itemize}
        \item \strong{non-degenerate}, when $f''(\saddlePoint)\neq 0$, and a standard steepest descent procedure can be used, see \cite{arfken2011mathematical} for example;
        \item \strong{degenerate}, when $f''(\saddlePoint)= 0$, and the more subtle procedure in \cite{chester1957extension} needs to be employed.
    \end{itemize}
    \item \strong{Poles} of $g$ of modulus larger or equal to one, which are symptomatic of a GKS-unstable boundary condition. 
\end{itemize}
It is worth noting that saddle points of $f(\timeShiftOperator)$ and poles of $g(\timeShiftOperator)$ can coincide, giving rise to interesting behaviors that must be addressed.
When no saddle point is present, typically in the setting where we look at $\indexTime\mapsto \globalTruncationError_{\indexSpace}^{\indexTime}$ with $\indexSpace\in\naturals$ fixed for $\indexTime\to+\infty$, or with the study of $\indexTime\mapsto \sum_{\indexSpace\geq 0}\globalTruncationError_{\indexSpace}^{\indexTime}$ and $\indexTime\mapsto \sum_{\indexSpace\geq 0}(\globalTruncationError_{\indexSpace}^{\indexTime})^2$, we rely on techniques germane to \strong{monovariate analytic combinatorics}, see \cite{flajolet2009analytic}.
Back to the case where saddle points are present, it is interesting to observe that asymptotics could be investigated, despite involved mathematics, in the context of \strong{multivariate analytic combinatorics} \cite{melczer2021invitation}. 
We do not pursue this path in the present study, which would, however, yield analogous results, since asymptotics in multivariate analytic combinatorics are also based on steepest descent approximations.

\subsection{State of the art}

Let us now review existing literature concerning the analysis of numerical schemes through \strong{steepest descent}/\strong{stationary phase} arguments or---more broadly---\strong{contour deformation}.

Contributions started in the 1970s with \cite{hedstrom1975models}, who studied the numerical solution of one-step dissipative schemes on $\relatives$\footnote{Thus with Fourier transform being a valid tool.} with the initial condition being a step function.
Roughly at the same time, \cite{serdyukova1971oscillations} considers Hedstrom's framework, with extensions allowing implicit schemes and several (finite number) contact points of the symbol with the unit circle, where the scheme is dissipative of some order. Schemes are thus not dissipative in the strictest possible sense.
In this paper, the presence of a pole, due to the initial step function, coinciding with a saddle point, and related difficulties, are clearly made explicit.
At the end of this work, a brief account of the behavior of non-dissipative schemes is given (\confer{}, the late work of \cite{estep1990boundedness}).
Nearly three decades later, \cite{bouche2003comparison} proposed experimental verification of the trends highlighted by Hedstrom and Serdyukova. Bouche also recently published a textbook \cite{bouche2024analyse} where---\emph{inter alia}---the behavior of the sum of the two Green functions (\idEst{}, the solutions with Dirac delta as initial data, see \Cref{sec:GreenFunction}) of the leap-frog scheme is studied in the zones where saddle points are non-degenerate. 
Authors of \cite{coulombel2022generalized, coulombel2023sharp} consider Serdyukova's dissipative setting and study the Green functions of the schemes strongly relying on contour deformation. Similar results in this framework with explicit schemes are obtained in \cite{coeuret2025local}, whose work generalizes part of that in \cite{randles2015convolution} by providing terms of arbitrary order in the asymptotic expansions. However, it must be noted that the work of Randles and Saloff-Coste also deals with schemes where the symbol does not ``dissipate'' at the contact points with the unit circle.
We also mention works by Trefethen \cite{trefethen1984instability, trefethen1985stability} where stationary phase arguments are utilized.

Concerning problems involving boundaries, the work of \cite{chin1975dispersion} provides a key contribution to the topic. 
The author analyzes a semi-discretized staggered approximation for the two-way wave equation on a segment with zero initial and right-boundary data, and a step function as left-boundary datum. The procedure is based on the Laplace transform and on rewriting the discrete solution using Bessel functions, which are eventually studied by the steepest descent method that faces saddle points coinciding with a pole, as it was the case in Serdyukova's work.

\subsection{Plan of the paper}
The rest of the paper is structured as follows.
In \Cref{sec:mainResults}, we outline the main results concerning the leap-frog scheme and a dissipative first-order two-steps scheme endowed with stable and unstable boundary conditions.
These results are described in a coarse-grained and qualitative fashion that provides insight on the ``physics''  of such numerical algorithms and its causes---rather than precise quantitative statements.
This latter level of detail is the aim of \Cref{sec:detailedResults}, which rigorously states results and provides their proofs.
In \Cref{sec:GreenFunction}, which retrospectively complements the discussion with boundary, we apply the same techniques to a simpler problem, that is the analysis of the Green functions of the boundary-less leap-frog scheme.
General conclusions are drawn in \Cref{sec:Conclusions}.

\section{Outline of the main results}\label{sec:mainResults}
In what follows, we need to consider functions of complex variables instead of real ones.
We thus introduce the following notations
\begin{align*}
	\unitCircle\definitionEquality\{\timeShiftOperator\in\complex\quad\text{s.t.}\quad |\timeShiftOperator|=1\}, \qquad
	\unitDisk &\definitionEquality \{\timeShiftOperator\in\complex\quad\text{s.t.}\quad |\timeShiftOperator|<1\}, \qquad \closedUnitDisk\definitionEquality \unitDisk\cup\unitCircle, \\
	\neighborhoodInfinity &\definitionEquality \{\timeShiftOperator\in\complex\quad\text{s.t.}\quad |\timeShiftOperator|>1\}, \qquad \closedNeighborhoodInfinity\definitionEquality \neighborhoodInfinity\cup\unitCircle.
\end{align*}
Moreover, for a sequence $\naturals\ni\indexSpace\mapsto\globalTruncationError_{\indexSpace}$, we consider the norms ($1\leq p < \infty$)
\begin{equation*}
    \lVert \globalTruncationError\rVert_p
    \definitionEquality
    \Bigl ( 
    \sum_{\indexSpace\in\naturals}
    |\globalTruncationError_{\indexSpace}|^p
    \Bigr )^{1/p}, 
    \quad 
    \lVert \globalTruncationError\rVert_{\infty}
    \definitionEquality
    \sup_{\indexSpace\in\naturals}
    |\globalTruncationError_{\indexSpace}|, \qquad
    \lVert \globalTruncationError\rVert_{\ell^p(\spaceStep\naturals)}
    \definitionEquality
    \Bigl ( 
    \sum_{\indexSpace\in\naturals}
    \spaceStep|\globalTruncationError_{\indexSpace}|^p
    \Bigr )^{1/p}, 
    \quad 
    \lVert \globalTruncationError\rVert_{\ell^{\infty}(\spaceStep\naturals)}
    \definitionEquality
    \lVert \globalTruncationError\rVert_{\infty}.
\end{equation*}

\subsection{Leap-frog bulk scheme}\label{sec:leapFrog}
Before analyzing the long-time behavior of solutions in presence of boundary conditions, let us first establish the following assumption on the inherent stability of the leap-frog bulk scheme. In particular, we restrict the Courant number $\courantNumber$ to ensure $L^2$ stability and consider the boundary to be an outflow. 

\begin{assumption}[Stable bulk scheme and outflow]\label{ass:stableBulk}
    Assume that $-1<\courantNumber<0$, so that the bulk scheme \eqref{eq:leapFrogEventually} without boundary ($\indexSpace\in\relatives$) is $L^2$ stable, and that the considered boundary is an outflow.
\end{assumption}
Under the stability condition by \Cref{ass:stableBulk}, the two amplification factors (or symbols) associated with the boundary-less leap-frog scheme belong to $\unitCircle$ for every harmonics, see \cite[Chapter 4]{strikwerda2004finite} and \Cref{sec:GreenFunction}. This entails that, loosely speaking, \strong{no frequency is damped} in time, whence initial disturbances---such as a time-space corner error---generate rather involved solutions at later times.
This also implies some sorts of ``asymptotic'' \strong{conservation of energy}, \idEst{} the $L^2$ norm in space---see \Cref{prop:asymptoticL2Stable} and \ref{prop:asymptoticL2StableGreen}, for the scheme acts as a friction-less medium.

\subsubsection{Stable boundary conditions}\label{sec:stableBC}

\begin{assumption}[Stable boundary conditions]\label{ass:stableBC}
    Let $\stableRoot(\timeShiftOperator)$ be the root of $\timeShiftOperator^2 - \courantNumber\timeShiftOperator(\stableRoot(\timeShiftOperator)^{-1}-\stableRoot(\timeShiftOperator)) - 1=0$ such that $\stableRoot(\timeShiftOperator)\in\unitDisk$ for $\timeShiftOperator\in\neighborhoodInfinity$.
    Assume that the function
    \begin{equation*}
        \timeShiftOperator \mapsto \timeShiftOperator^2 - \timeShiftOperator \sum_{k\in\naturals}\coefficientEventualBoundaryScheme_k \stableRoot(\timeShiftOperator)^{k} - \sum_{k\in\naturals}\coefficientEventualBoundarySchemeOld_k \stableRoot(\timeShiftOperator)^k
        \quad \text{\strong{does not have any zero} in }\closedNeighborhoodInfinity.
    \end{equation*}
\end{assumption}

\begin{example}[Upwind boundary scheme]\label{ex:upwind}
    This important example is based on $\coefficientEventualBoundaryScheme_0 = 1+\courantNumber$ and $\coefficientEventualBoundaryScheme_1 = -\courantNumber$ and all other coefficients equal to zero, and corresponds to \cite[Equation (12.2.2d)]{strikwerda2004finite}, where it is proven that the scheme fulfills \Cref{ass:stableBC}.
    This boundary scheme also arises within lattice Boltzmann schemes in \cite{bellotti2025consistency}, and verifies \eqref{eq:orderConstraints}.
\end{example}

We avail of the previous example to ponder something: the recurrent definition of $\globalTruncationError_{\indexSpace}^{\indexTime}$ by \eqref{eq:zeroInitialError}--\eqref{eq:schemeErrorInitial}--\eqref{eq:schemeErrorEventual}, although simple and easily computer-implementable, does not give much insight into the behavior of $\globalTruncationError_{\indexSpace}^{\indexTime}$ as $\indexTime$ grows.
On the other hand, although an explicit expression of $\globalTruncationError_{\indexSpace}^{\indexTime}$ for \Cref{ex:upwind} is available---see Appendix \ref{app:explicitExpressionUpwind}---it is neither simple, nor numerically-stable, nor sheds any light on the ``physics'' of the scheme as $\indexTime\to+\infty$.
\begin{example}\label{ex:upwindWithDiffusion}
    Adding a three-point Laplacian at time $\timeGridPoint{\indexTime-1}$ to the boundary scheme in \Cref{ex:upwind}, we get $\coefficientEventualBoundaryScheme_0 = 1+\courantNumber$ and $\coefficientEventualBoundaryScheme_1 = -\courantNumber$, plus $\coefficientEventualBoundarySchemeOld_{0} = \tfrac{1}{2}\delta \courantNumber$, $\coefficientEventualBoundarySchemeOld_1 = -\delta \courantNumber$, and $\coefficientEventualBoundarySchemeOld_2 = \tfrac{1}{2}\delta \courantNumber$, with $\delta \in\reals$.
    With this choice, \eqref{eq:orderConstraints} is fulfilled regardless of the value of $\delta$.
    Usually, when $\delta$ is close to zero, \Cref{ass:stableBC} is fulfilled.
\end{example}
\begin{example}[Dirichlet boundary conditions]\label{ex:Dirichlet}
    If we take $\coefficientEventualBoundaryScheme_k = \coefficientEventualBoundarySchemeOld_k = 0$ for all $k\in\naturals$, this condition does not fulfill \eqref{eq:orderConstraints} but verifies \Cref{ass:stableBC} by direct inspection, see also the Goldberg-Tadmor lemma \cite{goldbergtadmor1981, coulombel2011semigroup}.
\end{example}

\begin{informalresult}[Structure of $\globalTruncationError_{\indexSpace}^{\indexTime}$ for $\indexTime\gg 1$: stable boundary conditions]\label{property:stable}
    Let $\globalTruncationError_{\indexSpace}^{\indexTime}$ be the solution of \eqref{eq:zeroInitialError}--\eqref{eq:schemeErrorInitial}--\eqref{eq:schemeErrorEventual}, and \Cref{ass:stableBulk} and \ref{ass:stableBC} be fulfilled.
    For $\indexTime\gg 1$, $\globalTruncationError_{\indexSpace}^{\indexTime}$ features four zones according the value of $\indexSpace$, given as follows.
    \begin{enumerate}
        \item A \strong{near-wall} zone at $\indexSpace/\indexTime \sim 0$, where $\globalTruncationError_{\indexSpace}^{\indexTime} = \bigO{\indexTime^{-3/2}}$ and features a grid-scale--oscillating profile depending on the boundary conditions modulated by slowly oscillating profile from the bulk scheme.
        The detailed claim is \Cref{prop:nearWallStable}.
        \item A \strong{transition} zone for $\indexSpace/\indexTime\sim \nu$ with $\nu \in (0, |\courantNumber|)$, where $\globalTruncationError_{\indexSpace}^{\indexTime} = \bigO{\indexTime^{-1/2}}$ and features a grid-scale--oscillating profile depending on the boundary conditions modulated by slowly oscillating profile from the bulk scheme.
        The detailed claim is \Cref{prop:transitionZoneStable}.
        \item A \strong{front} zone for $\indexSpace/\indexTime\sim |\courantNumber|$, where $\globalTruncationError_{\indexSpace}^{\indexTime} = \bigO{\indexTime^{-1/3}}$ and features a grid-scale--oscillating profile depending on the boundary conditions modulated by an \strong{Airy function} from the bulk scheme.
        The detailed claim is \Cref{prop:frontZoneStable}.
        \item A zone \strong{ahead-of-the-front} for $\indexSpace/\indexTime\sim \nu$ with $\nu \in (|\courantNumber|, 1]$, where $\globalTruncationError_{\indexSpace}^{\indexTime}$ exponentially goes to zero with $\indexTime$.
    \end{enumerate}
    
    Moreover, we have that for $\indexTime\gg 1$, $\lVert
        \globalTruncationError^{\indexTime}
        \rVert_2
        =
        C
        +\bigO{\indexTime^{-3/2}}$,
    where the constant $C$ depends on the boundary scheme and is precisely given in \Cref{prop:asymptoticL2Stable}.
\end{informalresult}

\begin{remark}[Link with saddle points]
    The previous behaviors come from saddle points, as follows (see \Cref{lemma:saddlePoints} for more information).
    \begin{enumerate}
        \item \strong{Near-wall} zone. Saddle points do not exist. The behavior is dominated by \strong{branch point singularities} (hence the scaling in $\indexTime$, see \cite{flajolet2009analytic, besse2021discrete}) on $\unitCircle\smallsetminus \reals$---occurring in complex-conjugate pairs, thus the oscillating behavior.
        \item \strong{Transition} zone. Saddle points on $\unitCircle\smallsetminus \reals$ are present and dominate as complex-conjugate pairs, thus providing the oscillating behavior\footnote{Not simply Gaussians as in the non-degenerate steepest descent theory for one saddle point.}. The diffusive scaling in $\indexTime$ comes from the fact that these saddle points are non-degenerate.
        \item \strong{Front} zone. Saddle points are $\pm 1$. 
        The dispersive scaling in $\indexTime$ and the modulation by the Airy function come from the dispersive character of the bulk scheme around these saddle points.
        Otherwise said, the saddle points are degenerate.
        \item \strong{Ahead-of-the-front} zone. The saddle points which can be crossed by contour deformation are in $\neighborhoodInfinity$. 
        However, the function $\textnormal{Re}(f)$ at these saddle points belongs to $\unitDisk$, yielding geometrical damping in $\indexTime$.
    \end{enumerate}
\end{remark}

\begin{remark}[Compatibility of our asymptotics with strong stability]
    Note that the asymptotics $\lVert
        \globalTruncationError^{\indexTime}
        \rVert_2
        =
        C
        +\bigO{\indexTime^{-3/2}}$ for the norm of the solution in the bulk, and $\globalTruncationError_0^{\indexTime}=\bigO{\indexTime^{-3/2}}$ for the trace terms in the limit $\indexTime\gg 1$ are compatible with strong (GKS) stability---and even semi-group stability, see \cite[Theorem 1]{coulombel2015leray}.
\end{remark}

\begin{figure}[h]
    \centering
    \includegraphics[width=1\textwidth]{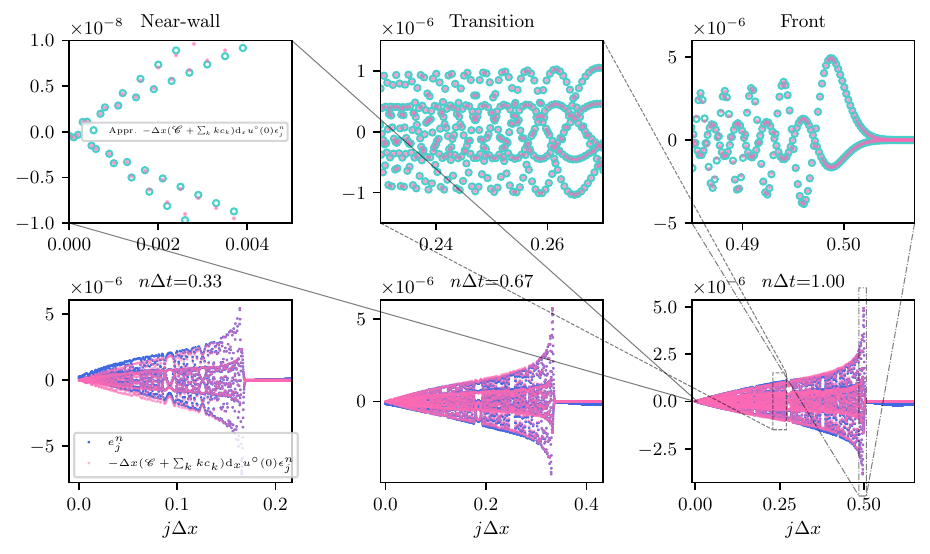}
    \caption{Results for the scheme in \Cref{ex:upwind} and \eqref{eq:laxFriedrichsInitialization}. Bottom: comparison between $\globalTruncationErrorNotNormalized_{\indexSpace}^{\indexTime}$ and $\globalTruncationError_{\indexSpace}^{\indexTime}$ (renormalized).
    Top: zoom on certain areas for $\globalTruncationError_{\indexSpace}^{\indexTime}$, compared to the approximations given by the truncated right-hand sides of \eqref{eq:nearWallStable} (near-wall zone), \eqref{eq:expansionBeyndShock_upwindScheme} (transition zone), and \eqref{eq:approximationFrontShockByAiry} (front zone).}
    \label{fig:my_label}
\end{figure}

We continue by qualitatively comparing $\globalTruncationErrorNotNormalized_{\indexSpace}^{\indexTime}$ and $\globalTruncationError_{\indexSpace}^{\indexTime}$, and check the approximations by \Cref{prop:nearWallStable}, \ref{prop:transitionZoneStable}, and \ref{prop:frontZoneStable}---recapitulated in \Cref{property:stable}---for the latter.
As an illustration, consider the boundary scheme in \Cref{ex:upwind}.
In order to match with an actual time-space corner scheme and an initial-time/space-bulk scheme fulfilling \eqref{eq:orderConstraints}, we consider
\begin{equation}\label{eq:laxFriedrichsInitialization}
    \coefficientCornerScheme_0 = \tfrac{1}{2}(1+\courantNumber), \quad 
    \coefficientCornerScheme_1 = \tfrac{1}{2}(1-\courantNumber) \qquad
    \text{and}\qquad
    \coefficientInitialBulkScheme_{-1} = \tfrac{1}{2}(1+\courantNumber), \quad 
    \coefficientInitialBulkScheme_{1} = \tfrac{1}{2}(1-\courantNumber),
\end{equation}
thus the initial-time/space-bulk scheme is the Lax-Friedrichs scheme.
With this, we perform simulations on the bounded domain $[0, 1]$ paved with $10^4$ discrete points.
The Courant number is $\courantNumber = -\tfrac{1}{2}$ and the initial datum reads $\solutionCauchyProblem^{\initial}(\spaceVariable) = \textnormal{exp}(-50 (\spaceVariable-\tfrac{1}{10})^2)$, so that $\frac{\differential{\solutionCauchyProblem^{\initial}}}{\differential{\spaceVariable}}(0) = 10 e^{-1/2}\neq 0$.
The global truncation error $\globalTruncationErrorNotNormalized_{\indexSpace}^{\indexTime}$ and the renormalized approximation $\globalTruncationError_{\indexSpace}^{\indexTime}$ within this setting are presented on the bottom row of \Cref{fig:my_label}.
The top row compares, in three different zones, the exact expression of $\globalTruncationError_{\indexSpace}^{\indexTime}$ and the obtained approximations.
On the bottom, we observe that the renormalized $\globalTruncationError_{\indexSpace}^{\indexTime}$ is a good representation of $\globalTruncationErrorNotNormalized_{\indexSpace}^{\indexTime}$, up to the addition of a smooth profile of amplitude $\bigO{\spaceStep^2}$ coming from non-zero smooth initial datum.  
The top row reveals that  the truncated right-hand sides of \eqref{eq:nearWallStable} (near-wall zone), \eqref{eq:expansionBeyndShock_upwindScheme} (transition zone), and \eqref{eq:approximationFrontShockByAiry} (front zone) are effective in describing each zone.

\begin{figure}
    \centering
    \includegraphics[width=1\textwidth]{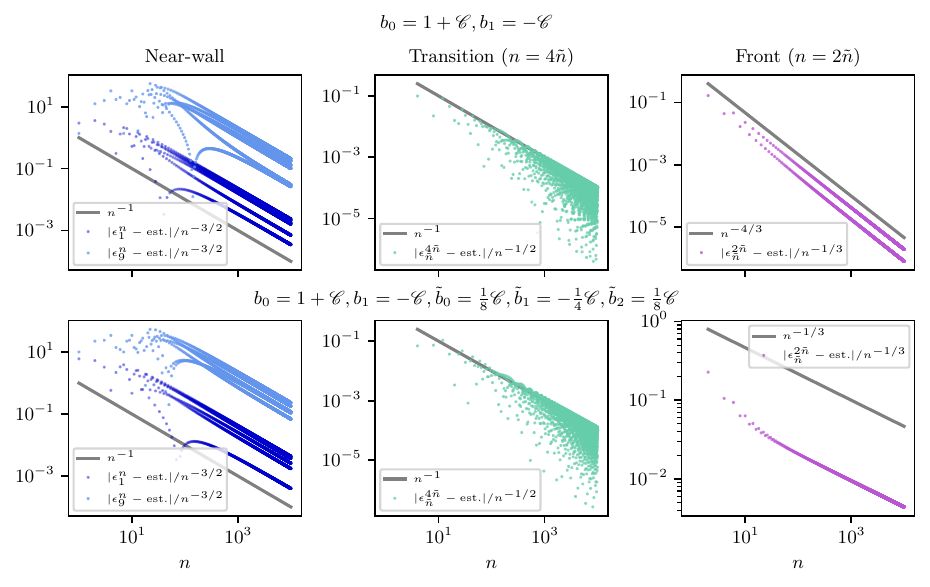}
    \caption{Errors (normalized by the leading asymptotics in $\indexTime$) between the actual values of $\globalTruncationError_{\indexSpace}^{\indexTime}$ and the estimations given by the truncated right-hand sides of \eqref{eq:nearWallUpwind} (near-wall zone, on the second and tenth cells), \eqref{eq:expansionBeyndShock} (transition zone), and \eqref{eq:approximationFrontShockByAiry} (front zone), as time $\indexTime$ advances.
    We considered \Cref{ex:upwind} (top) and \Cref{ex:upwindWithDiffusion} with $\delta = \tfrac{1}{4}$ (bottom), and $\courantNumber = -\tfrac{1}{2}$.}
    \label{fig:errorsStable}
\end{figure}

Quantitative comparisons between the three approximations against the true values of $\globalTruncationError_{\indexSpace}^{\indexTime}$ are conducted in the same setting.
To this end, the near-wall expression is evaluated at the second ($\indexSpace = 1$) and tenth ($\indexSpace = 9$) cells of the domain.
The transition approximation is probed at the point moving at (group) velocity $\tfrac{1}{2}|\courantNumber|= \tfrac{1}{4}$, which belongs to the spatial mesh every four time iterations.
Finally, the front approximation is evaluated at the point moving at (group) velocity $|\courantNumber| = \tfrac{1}{2}$, which belongs to the mesh every two iterations.
Results are shown in \Cref{fig:errorsStable}, where errors are renormalized by the found asymptotic in $\indexTime$.
Near-wall and transition zones yield linear convergence to the leading-order terms, as expected.
For the front zone, the rate of convergence with a multi-step boundary scheme from \Cref{ex:upwindWithDiffusion} (second row) is $\bigO{\indexTime^{-1/3}}$, which is in accordance with \cite[Chapter VII, Equation (4.21)]{wong2001asymptotic}. The supra-convergence in the case of \Cref{ex:upwind} (first row) is likely due to the vanishing of the neglected terms---proportional to $\bigO{\indexTime^{-2/3}}$ times the derivative of the Airy function.

\begin{figure}
    \centering
    \includegraphics[width=1\textwidth]{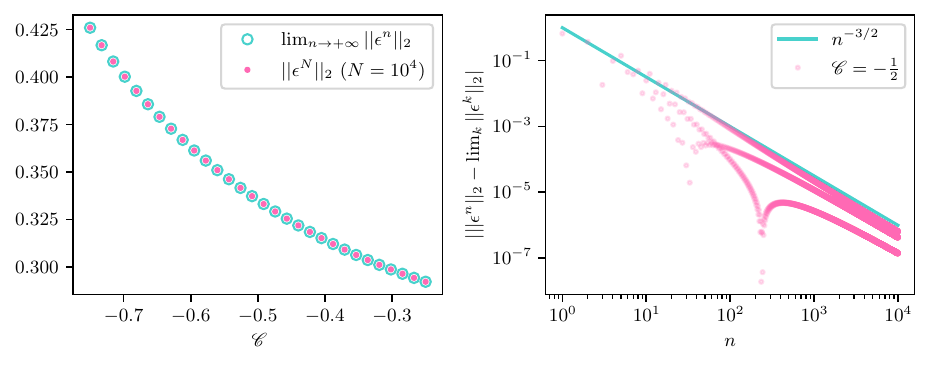}
    \caption{Asymptotic for the $L^2$ norm of $\indexSpace\mapsto\globalTruncationError_{\indexSpace}^{\indexTime}$ with upwind scheme, \confer{} \Cref{ex:upwind}, as boundary scheme. Left: comparison of the final value versus the theoretical estimate \eqref{eq:asymptoticL2Norm} (computed through numerical quadrature) at different Courant numbers. Right: convergence rate for a given Courant number.}
    \label{fig:L2normAsymptotic}
\end{figure}

\begin{figure}
    \centering
    \includegraphics[width=.5\textwidth]{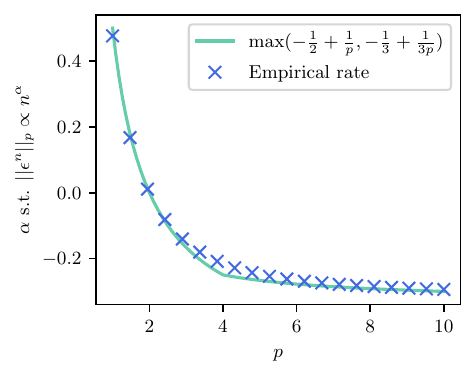}
    \caption{Empirical trend with $\indexTime$ of $\lVert\globalTruncationError^{\indexTime}\rVert_p$ as function of $p$ for \Cref{ex:upwind}.}
    \label{fig:normLP}
\end{figure}

\begin{figure}
    \centering
    \includegraphics[width=1\textwidth]{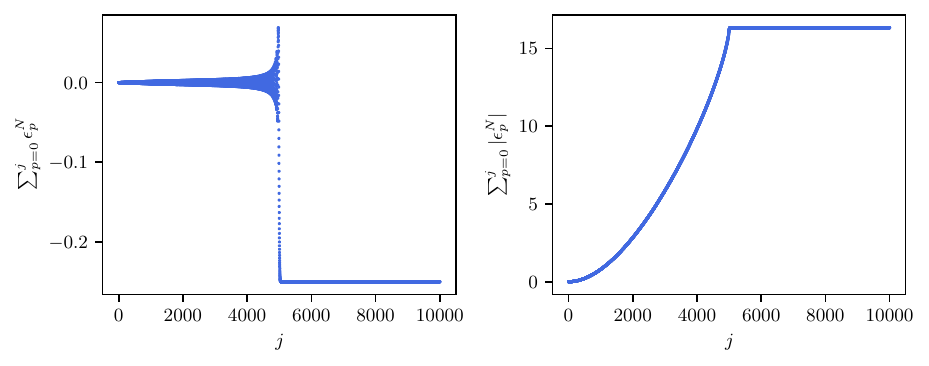}
    \caption{Partial sum of $\indexSpace\mapsto\globalTruncationError_{\indexSpace}^{N}$ (left) and of $\indexSpace\mapsto|\globalTruncationError_{\indexSpace}^{N}|$ (right) for $N= 10^4$ for \Cref{ex:upwind}.}
    \label{fig:cumulativeSum}
\end{figure}

Let us try to conclude on the trend of
\begin{equation*}
    \lVert \globalTruncationErrorNotNormalized^{\finalTime/\timeStep}\rVert_{\ell^p(\spaceStep\naturals)}
    \qquad \textnormal{and} \qquad 
    \sup_{0\leq \indexTime\leq \finalTime/\timeStep}\lVert \globalTruncationErrorNotNormalized^{\indexTime}\rVert_{\ell^p(\spaceStep\naturals)}
\end{equation*}
in $\spaceStep$, where $\finalTime>0$ is a fixed final time for the simulation, at which we would like to approximate the exact solution.
This was one of the questions originally stimulating this work.
Notice that we have, if $L^p$ stability holds true:
\begin{equation*}
    \lVert \globalTruncationErrorNotNormalized^{\indexTime}\rVert_{\ell^p(\spaceStep\naturals)} 
    \simeq
    \spaceStep^{1+\frac{1}{p}}
    \lVert \globalTruncationError^{\indexTime}\rVert_{p} 
    \qquad 
    \textnormal{at leading order (in }\spaceStep).
\end{equation*}
We try to extract the behavior of $\lVert \globalTruncationError^{\indexTime}\rVert_{p}$ for large $\indexTime$ from the previous discussion.
For $p = 2$, we have seen that $\lVert \globalTruncationError^{\indexTime}\rVert_{p}$ tends to a constant at a rate $\bigO{\indexTime^{-3/2}}$.
This fact and the found value of the constant are empirically confirmed, see \Cref{fig:L2normAsymptotic}.
Going on, we obtain
\begin{equation*}
    \lVert \globalTruncationErrorNotNormalized^{\finalTime/\timeStep}\rVert_{\ell^2(\spaceStep\naturals)}
    \simeq\spaceStep^{3/2}
    \qquad \textnormal{and} \qquad 
    \sup_{0\leq \indexTime\leq \finalTime/\timeStep}\lVert \globalTruncationErrorNotNormalized^{\indexTime}\rVert_{\ell^2(\spaceStep\naturals)}
    \simeq\spaceStep^{3/2},
\end{equation*}
which agree with the observations in \cite{bellotti2025consistency}.
For $p=\infty$, we have $ \lVert \globalTruncationError^{\indexTime}\rVert_{\infty} \sim \indexTime^{-1/3}$for large $\indexTime$, hence
\begin{equation*}
    \lVert \globalTruncationErrorNotNormalized^{\finalTime/\timeStep}\rVert_{\ell^{\infty}(\spaceStep\naturals)}
    \simeq\spaceStep^{4/3}
    \qquad \textnormal{and} \qquad 
    \sup_{0\leq \indexTime\leq \finalTime/\timeStep}\lVert \globalTruncationErrorNotNormalized^{\indexTime}\rVert_{\ell^{\infty}(\spaceStep\naturals)}
    \simeq\spaceStep.
\end{equation*}
In order to study other values of $p$, we assume that the behavior of $\globalTruncationError^{\indexTime}_{\indexSpace}$ is essentially that of a piecewise constant function of order $\bigO{\indexTime^{-1/2}}$ for $\spaceVariable\in(0, |\courantNumber|\indexTime - \beta \indexTime^{1/3})$ and $\bigO{\indexTime^{-1/3}}$ for $\spaceVariable\in(|\courantNumber|\indexTime - \beta \indexTime^{1/3}, |\courantNumber|\indexTime + \beta \indexTime^{1/3})$ for some $\beta>0$.
This guess is inspired by \cite{serdyukova1971oscillations} and \cite[Theorem 1.2]{randles2015convolution}.
This suggests that, forgetting any constants, we have 
\begin{equation*}
    \lVert \globalTruncationError^{\indexTime}\rVert_{p} \propto
    \Bigl ( (\indexTime - \indexTime^{1/3})\indexTime^{-p/2} + \indexTime^{1/3}\indexTime^{-p/3}\Bigr )^{1/p} \propto \indexTime^{\max(\frac{1}{p}-\frac{1}{2}, \frac{1}{3p}-\frac{1}{3})}.
\end{equation*}
In particular, this yields $\lVert \globalTruncationError^{\indexTime}\rVert_{1} \propto \indexTime^{1/2}$.
The two arguments in the maximum at the exponent are equal when $p = 4$, thus we have two regimes:
\begin{itemize}
    \item $p<4$, the dominant term is the first one, linked to the \strong{transition zone}.
    \item $p>4$, the dominant term is the second one, linked to the \strong{front zone}.
\end{itemize}
The estimate gains accuracy as $p$ departs from $4$, while close to this value the magnitude of each term must be carefully studied.
This can be observed in \Cref{fig:normLP}, where simulations are conducted under the same conditions as before.
In the case $p=1$, the results of \Cref{fig:cumulativeSum} show that the transition zone plays a crucial role in the asymptotics of the norm, which is far from being the case for $\indexTime\mapsto \sum_{\indexSpace} \globalTruncationError_{\indexSpace}^{\indexTime}$ (\confer{}, \Cref{prop:momentsInSpaceStable}).
This is due to the fact that $\globalTruncationError_{\indexSpace}^{\indexTime}$ strongly oscillates and changes sign in $\indexSpace$, as it is far, even after normalization, from being the probability mass function of some discrete random variable.
Back to the general setting, we obtain
\begin{equation*}
    \lVert \globalTruncationErrorNotNormalized^{\finalTime/\timeStep}\rVert_{\ell^p(\spaceStep\naturals)}
    \propto
    C_p(\finalTime)
    \spaceStep^{1+\frac{1}{p} - \max(\frac{1}{p}-\frac{1}{2}, \frac{1}{3p}-\frac{1}{3})},
\end{equation*}
On the other hand, we have
\begin{equation*}
    \sup_{0\leq \indexTime\leq \finalTime/\timeStep}\lVert \globalTruncationErrorNotNormalized^{\indexTime}\rVert_{\ell^p(\spaceStep\naturals)} 
    \simeq
    \spaceStep^{1+\frac{1}{p}}
    \sup_{0\leq \indexTime\leq \finalTime/\timeStep}
    \lVert \globalTruncationError^{\indexTime}\rVert_{p} 
    \propto 
    \begin{cases}
        \spaceStep^{\frac{3}{2}}, \qquad &p\leq 2, \\
        \spaceStep^{1+\frac{1}{p}}, \qquad &p>2.
    \end{cases}
\end{equation*}

\subsubsection{Unstable boundary conditions}

The following assumption describes boundary conditions with instabilities of very mild severity.
\begin{assumption}[Unstable boundary conditions]\label{ass:unstableBC}
    Let $\stableRoot(\timeShiftOperator)$ be the root of $\timeShiftOperator^2 - \courantNumber\timeShiftOperator(\stableRoot(\timeShiftOperator)^{-1}-\stableRoot(\timeShiftOperator)) - 1=0$ such that $\stableRoot(\timeShiftOperator)\in\unitDisk$ for $\timeShiftOperator\in\neighborhoodInfinity$.
    Assume that the function
    \begin{equation*}
        \timeShiftOperator \mapsto \timeShiftOperator^2 - \timeShiftOperator \sum_{k\in\naturals}\coefficientEventualBoundaryScheme_k \stableRoot(\timeShiftOperator)^{k} - \sum_{k\in\naturals}\coefficientEventualBoundarySchemeOld_k \stableRoot(\timeShiftOperator)^k
        \quad \text{\strong{has} }\timeShiftOperator=-1\text{ \strong{as only zero} in }\closedNeighborhoodInfinity\text{ and its \strong{multiplicity is one}}.
    \end{equation*}
\end{assumption}
\begin{remark}[On \Cref{ass:unstableBC}]
    We can provide analogous results when the zero is (also) $\timeShiftOperator = 1$. 
    Complex zeros on $\unitCircle$ could also be considered, in which case they must appear in complex conjugate pairs, see \Cref{rem:anyGroupVelocity}. 
    
    Considering zeros on $\unitCircle$ of higher multiplicity, e.g. \cite{bellotti2025consistency} or \cite[Equation (11.2.c)]{strikwerda2004finite}, encompasses more serious (but still polynomial) instabilities and can be treated by similar tools.
    If one considers the possibility of zeros in $\neighborhoodInfinity$, this leads to severe (exponential) boundary-localized instabilities known as of \strong{Godunov-Ryabenkii} type, see \cite{trefethen1984instability}.
    
    The assumption of a simple zero in $\timeShiftOperator=-1$ automatically implies that
    \begin{equation}\label{eq:simpleZeroAtMinusOne}
        1+\sum_{k\geq 0}\coefficientEventualBoundaryScheme_k - \sum_{k\geq 0}\coefficientEventualBoundarySchemeOld_k = 0
        \qquad \text{and}\qquad 
        2 + \sum_{k\geq 0}\coefficientEventualBoundaryScheme_k + \frac{1}{\courantNumber}\sum_{k\geq 1}k(\coefficientEventualBoundaryScheme_k-\coefficientEventualBoundarySchemeOld_k)
        \neq 0.
    \end{equation}
\end{remark}

\begin{example}[An upwind leap-frog scheme]\label{ex:unstable}
    One example that we consider, fulfilling \Cref{ass:unstableBC}, is $\coefficientEventualBoundaryScheme_0 = - \coefficientEventualBoundaryScheme_1 = 1+2\courantNumber$ and $\coefficientEventualBoundarySchemeOld_1 = 1$ and all other coefficients equal to zero.
    This scheme satisfies the consistency conditions \eqref{eq:orderConstraints} (and is also second-order accurate).
    Although this scheme can be obtained from a lattice Boltzmann scheme as detailed in \cite{bellotti2025consistency}, it is already well-known in the literature as bulk scheme, see \cite{iserles1986generalized, thomas1993development}.
    This scheme shares similarities with the standard leap-frog scheme.
    Indeed, it is second-order accurate and non-dissipative, meaning that both symbols of the scheme have constant modulus equal to one regardless of the frequency.
\end{example}

\begin{example}[Inconsistent scheme coming from lattice Boltzmann]
    Another example meeting \Cref{ass:unstableBC} comes from lattice Boltzmann schemes, see \cite[Chapter 12]{bellotti:tel-04266822}, with the ``anti-bounce-back'' condition, yielding $\coefficientEventualBoundaryScheme_0 = \courantNumber$, $\coefficientEventualBoundaryScheme_1 = -\courantNumber$, and $\coefficientEventualBoundarySchemeOld_0 = 1$.
    This scheme does not satisfy the consistency conditions \eqref{eq:orderConstraints}: it is consistent with $\partial_{\timeVariable}\solutionCauchyProblem + \tfrac{1}{2}\advectionVelocity\partial_{\spaceVariable}\solutionCauchyProblem = 0$.
\end{example}

\begin{informalresult}[Structure of $\globalTruncationError_{\indexSpace}^{\indexTime}$ for $\indexTime\gg 1$: unstable boundary conditions]\label{property:unstable}
Let $\globalTruncationError_{\indexSpace}^{\indexTime}$ be the solution of \eqref{eq:zeroInitialError}--\eqref{eq:schemeErrorInitial}--\eqref{eq:schemeErrorEventual}, and \Cref{ass:stableBulk} and \ref{ass:unstableBC} be fulfilled. As for the stable boundary conditions case detailed in \Cref{property:stable}, for $\indexTime\gg 1$, $\globalTruncationError_{\indexSpace}^{\indexTime}$ features four zones according the value of $\indexSpace$, given as follows.
    \begin{enumerate}
        \item A \strong{near-wall} zone at $\indexSpace/\indexTime \sim 0$, where $\globalTruncationError_{\indexSpace}^{\indexTime} = (-1)^\indexTime R + \bigO{\indexTime^{-3/2}}$, where the $\bigO{\indexTime^{-3/2}}$-profile is as in \Cref{property:stable}. The detailed claim is \Cref{prop:nearWallUnstable}.
        \item A \strong{transition} zone for $\indexSpace/\indexTime\sim \nu$ with $\nu \in (0, |\courantNumber|)$, where $\globalTruncationError_{\indexSpace}^{\indexTime} = (-1)^\indexTime R + \bigO{\indexTime^{-1/2}}$,  where the $\bigO{\indexTime^{-1/2}}$-profile is as in \Cref{property:stable}.
        The detailed claim is \Cref{prop:transitionUnstable}.
        \item A \strong{front} zone for $\indexSpace/\indexTime\sim |\courantNumber|$, where $\globalTruncationError_{\indexSpace}^{\indexTime} = (-1)^\indexTime R \mathcal{M}_{\indexSpace}^{\indexTime} + \bigO{\indexTime^{-1/3}}$, where the profile $ \mathcal{M}_{\indexSpace}^{\indexTime}$ features the \strong{primitive of the Airy function}.
        The detailed claim is \Cref{prop:FrontUnstable}.
        \item A zone \strong{ahead-of-the-front} for $\indexSpace/\indexTime\sim \nu$ with $\nu \in (|\courantNumber|, 1]$, where $\globalTruncationError_{\indexSpace}^{\indexTime}$ exponentially goes to zero with $\indexTime$.
    \end{enumerate}
    The value $R$ depends on the coefficients of the boundary scheme and it is given by
    \begin{equation*}
        R = -
        \Bigl (2
            + \sum_{k\geq 0}
            \coefficientEventualBoundaryScheme_k
            + \frac{1}{\courantNumber} \sum_{k\geq 1}
            k (\coefficientEventualBoundaryScheme_k-\coefficientEventualBoundarySchemeOld_k)\Bigr )^{-1},
            \qquad \text{which is well-defined by virtue of \eqref{eq:simpleZeroAtMinusOne}.}
    \end{equation*}
\end{informalresult}

\begin{figure}[h]
    \centering
    \includegraphics[width=1\textwidth]{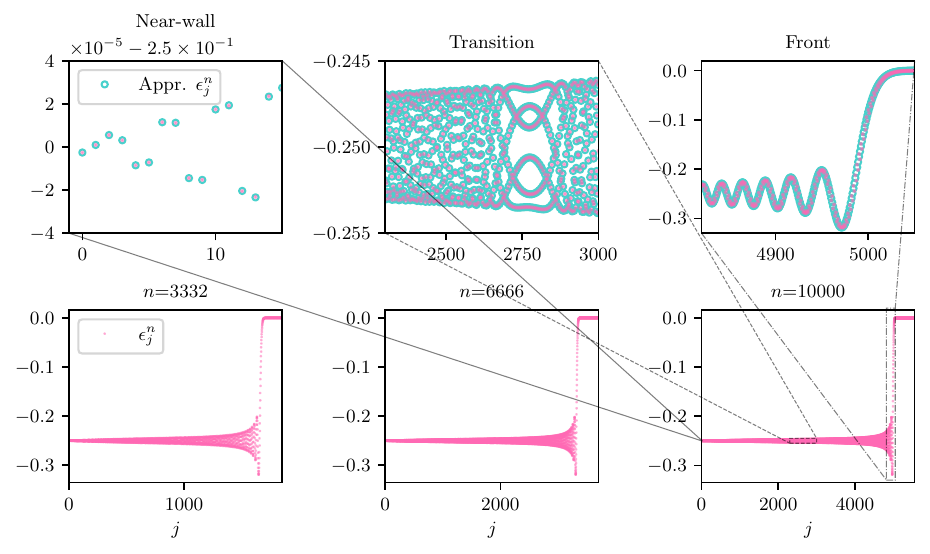}
    \caption{Results for the scheme in \Cref{ex:unstable}. 
    Bottom: values of $\globalTruncationError_{\indexSpace}^{\indexTime}$ at different time steps.
    Top: zoom on certain areas for $\globalTruncationError_{\indexSpace}^{\indexTime}$, compared to the approximations given by the truncated right-hand sides of \eqref{eq:nearWallUnstable} (near-wall zone), \eqref{eq:unstable_error_transition} (transition zone), and \eqref{eq:frontUnstable} (front zone).}
    \label{fig:unstableComparison}
\end{figure}

Under the same conditions as \Cref{sec:stableBC}, numerical results compared to the approximations are displayed in \Cref{fig:unstableComparison}, and show remarkable agreement.

\subsection{A dissipative bulk scheme with stable boundary conditions}\label{sec:dissipative}

Let now, contrarily to \Cref{sec:leapFrog}, $\courantNumber$ be positive, so that the boundary is an \strong{inflow}, and consider the two-indices sequence problem
\begin{align}
    &\globalTruncationError_{\indexSpace}^0 = 0, \qquad \indexSpace\in\naturals, \label{eq:zeroInitialErrorDissipative} \\
    &\globalTruncationError_0^1 = 1, \qquad \globalTruncationError_{\indexSpace}^1 = 0, \quad \indexSpace \geq 1, \label{eq:schemeErrorInitialDissipative} \\
    \indexTime\geq 1 \qquad & \globalTruncationError_0^{\indexTime + 1} = \sum_{k\in\naturals}\coefficientEventualBoundaryScheme_k \globalTruncationError_k^{\indexTime} + \sum_{k\in\naturals}\coefficientEventualBoundarySchemeOld_k \globalTruncationError_k^{\indexTime-1},\label{eq:schemeErrorEventualBoundaryDissipative}\\
    &\globalTruncationError_{\indexSpace}^{\indexTime + 1} = 
    \tfrac{1}{2}(2-\relaxationParameter)(\globalTruncationError_{\indexSpace-1}^{\indexTime} + \globalTruncationError_{\indexSpace+1}^{\indexTime})
    +
    (\relaxationParameter-1)\globalTruncationError_{\indexSpace}^{\indexTime - 1} +\tfrac{1}{2}\relaxationParameter\courantNumber (\globalTruncationError_{\indexSpace-1}^{\indexTime} - \globalTruncationError_{\indexSpace+1}^{\indexTime}), \quad \indexSpace\geq 1,\label{eq:schemeErrorEventualDissipative}
\end{align}
with the parameter $\relaxationParameter\in (0, 2)$.
Note that the bulk scheme \eqref{eq:schemeErrorEventualDissipative} can be see as a combination between a Lax-Friedrichs scheme ($\relaxationParameter = 1$) and a leap-frog scheme ($\relaxationParameter = 2$).

This scheme is---in its boundary-less version---``locally'' dissipative (of order two), in the sense that its symbols lay in $\unitDisk$, except at a finite number of frequencies where they belong to $\unitCircle$, and at these points, the symbols meet the assumptions of \cite[Theorem 1]{thomee1965stability}.
We will return to this point in \Cref{rem:linkwithThomee}.

\begin{assumption}[Stable boundary conditions]\label{ass:stableBCDissipative}
    Let $\stableRoot(\timeShiftOperator)$ be the root of $\timeShiftOperator^2 + \tfrac{1}{2}(\relaxationParameter-2)\timeShiftOperator (\stableRoot(\timeShiftOperator)^{-1}+\stableRoot(\timeShiftOperator))- \tfrac{1}{2}\relaxationParameter\courantNumber\timeShiftOperator(\stableRoot(\timeShiftOperator)^{-1}-\stableRoot(\timeShiftOperator)) + (1-\relaxationParameter)=0$ such that $\stableRoot(\timeShiftOperator)\in\unitDisk$ for $\timeShiftOperator\in\neighborhoodInfinity$.
    Assume that the function
    \begin{equation*}
        \timeShiftOperator \mapsto \timeShiftOperator^2 - \timeShiftOperator \sum_{k\in\naturals}\coefficientEventualBoundaryScheme_k \stableRoot(\timeShiftOperator)^{k} - \sum_{k\in\naturals}\coefficientEventualBoundarySchemeOld_k \stableRoot(\timeShiftOperator)^k
        \quad \text{\strong{does not have any zero} in }\closedNeighborhoodInfinity.
    \end{equation*}
\end{assumption}

\begin{informalresult}[Structure of $\globalTruncationError_{\indexSpace}^{\indexTime}$ for $\indexTime\gg 1$: stable boundary conditions]\label{property:Gaussian}
    Let $\globalTruncationError_{\indexSpace}^{\indexTime}$ be the solution of \eqref{eq:zeroInitialErrorDissipative}--\eqref{eq:schemeErrorInitialDissipative}--\eqref{eq:schemeErrorEventualBoundaryDissipative}--\eqref{eq:schemeErrorEventualDissipative}, $\relaxationParameter\in (0, 2)$ and $0<\courantNumber\leq 1$, and \Cref{ass:stableBCDissipative} be fulfilled.
    For $\indexTime\gg 1$, $\globalTruncationError_{\indexSpace}^{\indexTime}$ is solely significant for $\indexSpace/\indexTime\sim \courantNumber$, where it features a grid-scale oscillating profile times a Gaussian profile.
    The magnitude of the Gaussian profile is of order $(\indexTime ( \frac{1}{\relaxationParameter}-\frac{1}{2} )(1-\courantNumber^2))^{-1/2}$ and its standard deviation equal to $(\indexTime ( \frac{1}{\relaxationParameter}-\frac{1}{2} )(1-\courantNumber^2))^{1/2}$.
    The detailed claim is \Cref{prop:GaussianPeak}.
\end{informalresult}

\begin{remark}[Unstable boundary conditions or excited saddle points]
    Consider the case where \Cref{ass:stableBCDissipative} is replaced by the assumption that the boundary condition is unstable, with the only zero of the associated function in $\closedNeighborhoodInfinity$ being a simple zero at $\timeShiftOperator = 1$ (or/and at $\timeShiftOperator=-1$), or alternatively where \eqref{eq:schemeErrorEventualBoundaryDissipative} is replaced by $\globalTruncationError_{0}^{\indexTime + 1}= 1$ for $\indexTime\geq 1$.
    In this setting, the Gaussian profile in \Cref{property:Gaussian} is supplanted by the \strong{complementary error function} (roughly speaking, the integral of the Gaussian), see for example \cite[Chapter VII, Section 2]{wong2001asymptotic}.
    This is analogous to what occurs when passing from \Cref{property:stable} to \Cref{property:unstable}, where the Airy function in the front zone approximation is replaced by the \strong{primitive of the Airy function}, and shall not be discussed further.
\end{remark}

\begin{figure}
    \centering
    \includegraphics[width=1\textwidth]{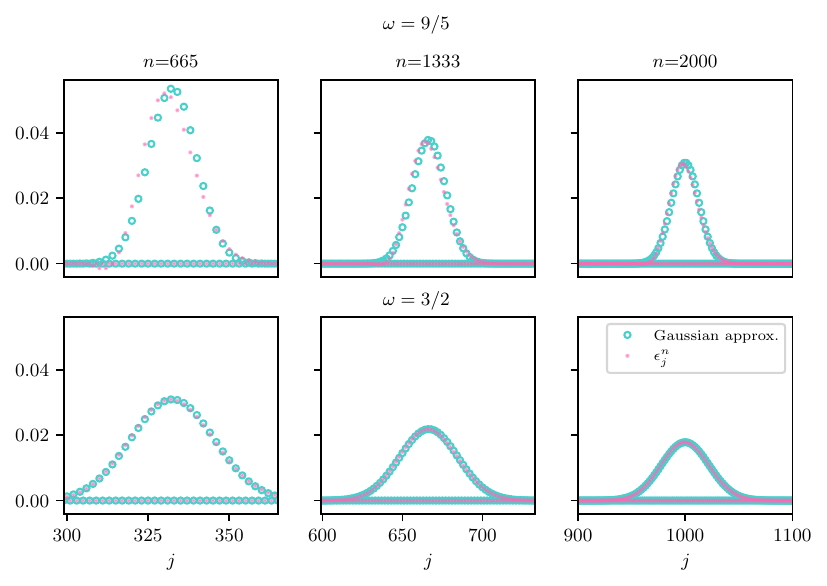}
    \caption{Results and comparison with the truncated right-hand side of \eqref{eq:gaussian} with $\relaxationParameter = \tfrac{9}{5}$ and $\relaxationParameter = \tfrac{3}{2}$ at $\courantNumber = \tfrac{1}{2}$ using the boundary conditions by \Cref{ex:Dirichlet}.}
    \label{fig:comparisonGaussianWithEstimates}
\end{figure}

Through the simulations presented in \Cref{fig:comparisonGaussianWithEstimates}, conducted with $\courantNumber=\tfrac{1}{2}$ and using the boundary scheme in \Cref{ex:Dirichlet}, we see that $\globalTruncationError_{\indexSpace}^{\indexTime}$ behaves closer and closer to the Gaussian profile in \eqref{eq:gaussian}.
This is true to a lesser extent when $\relaxationParameter$ is close to two, since the diffusion term is small (the saddle points are close to be degenerate), and dispersive effects, quite visible in the first snapshot of the top row, are still significant in the considered range of $\indexTime$.

\begin{example}[An analogous situation with an explicit (binomial) formula]
    Let us consider the problem with Dirichlet boundary condition and a manufactured bulk scheme:
    \begin{align*}
        &\globalTruncationError_{\indexSpace}^0 = 0, \qquad \indexSpace\in\naturals, \\
        &\globalTruncationError_{0}^1 = 1, \qquad \globalTruncationError_{\indexSpace}^1 = 0, \quad \indexSpace\geq 1, \\
        \indexTime\geq 1 
        \qquad 
        &\globalTruncationError_{0}^{\indexTime + 1} = 0, 
        \qquad 
        \globalTruncationError_{\indexSpace}^{\indexTime + 1}
        =
        \globalTruncationError_{\indexSpace}^{\indexTime - 1}
        +\courantNumber(\globalTruncationError_{\indexSpace-1}^{\indexTime}-\globalTruncationError_{\indexSpace+1}^{\indexTime})
        +\courantNumber(1-\courantNumber) (\globalTruncationError_{\indexSpace-1}^{\indexTime-1}-2\globalTruncationError_{\indexSpace}^{\indexTime-1}+\globalTruncationError_{\indexSpace+1}^{\indexTime-1}), \quad 
        \indexSpace\geq 1.
    \end{align*}
    The bulk scheme is a leap-frog scheme with dissipation ``in the past'', and has been devised based on the symbol of the upwind scheme and the negative of its conjugate.
    These symbols are essentially analogous to those of \eqref{eq:schemeErrorEventualDissipative} in terms of tangency properties to $\unitCircle$.
    By the techniques presented in \Cref{sec:detailedResults}, the residue theorem easily entails that for $\indexTime\geq 2$ and $\indexSpace\geq 1$, we have 
    \begin{equation*}
        \globalTruncationError_{\indexSpace}^{\indexTime}
        =
        \courantNumber
        \times
        \underbrace{
        \courantNumber^{\indexSpace-1}(1-\courantNumber)^{\indexTime-2-(\indexSpace-1)}
        \binom{\indexTime-2}{\indexSpace-1}
        }_{=\mathbb{P}(X=\indexSpace-1)}
        \mathds{1}_{\indexSpace\leq \indexTime-1}.
    \end{equation*}
    Here, $X$ is a discrete random variable distributed as a binomial with $X\sim B(\indexTime-2, \courantNumber)$.
    The fact that this distribution represents a peak moving at group velocity $\courantNumber$ can be grasped from its mode, which equals $\lfloor \courantNumber(\indexTime - 1)\rfloor$ outside well-known exceptional cases.
    Moreover, it is very well-known that for large $\indexTime$, $X\sim B(\indexTime-2, \courantNumber)$ is well approximated by the normal distribution $\mathcal{N}(\indexTime\courantNumber, \indexTime\courantNumber(1-\courantNumber))$.
    This fact yields something analogous to \Cref{property:Gaussian} and \Cref{prop:GaussianPeak} through a different way of reasoning.
\end{example}

\section{Detailed results and proofs}\label{sec:detailedResults}

Before digging into specific claims and their proofs, let us introduce a useful tool to map discrete time variables into the complex plane.
This consists of the $\timeShiftOperator$-transform, defined along with its inverse by 
\begin{equation*}
    \zTransformed{\globalTruncationError}_{\indexSpace}(\timeShiftOperator)\definitionEquality \sum_{\indexTime = 0}^{+\infty}\timeShiftOperator^{-\indexTime}\globalTruncationError_{\indexSpace}^{\indexTime}
    \qquad \text{and}\qquad 
    {\globalTruncationError}_{\indexSpace}^{\indexTime}
    =
    \frac{1}{2\pi i}\oint_C \timeShiftOperator^{\indexTime-1}\zTransformed{\globalTruncationError}_{\indexSpace}(\timeShiftOperator)\differential{\timeShiftOperator},
\end{equation*}
where $C$ is a positively oriented closed path fully contained within the region of convergence (which is $\neighborhoodInfinity$ in the considered cases). 
Remark that it is not enough that $C$ encloses the origin and all the singularities (poles, branch points, etc.) of $\timeShiftOperator^{\indexTime-1}\zTransformed{\globalTruncationError}_{\indexSpace}(\timeShiftOperator)$, since this would allow it to quit the region of convergence, for instance, by path deformation.

Observe that the $\timeShiftOperator$-transform becomes an \strong{ordinary generating function}, see \cite{flajolet2009analytic}, upon formally replacing $\timeShiftOperator$ by $\timeShiftOperator^{-1}$, which maps $\unitCircle$ onto itself, and $\unitDisk$ into $\neighborhoodInfinity$ (and \strong{viceversa}).
This correspondence shall be used multiple times.

\subsection{Leap-frog bulk scheme}\label{sec:proofsLeapFrog}

Using the $\timeShiftOperator$-transform introduced above, \eqref{eq:schemeErrorEventual} can be reinterpreted as a recurrence in space, depending on the complex parameter $\timeShiftOperator$.
The transform of \eqref{eq:schemeErrorEventual} and the use \eqref{eq:zeroInitialError}--\eqref{eq:schemeErrorInitial} give 
\begin{multline}\label{eq:ztransformedProblem}
    \timeShiftOperator^2 \zTransformed{\globalTruncationError}_{0}(\timeShiftOperator) - \timeShiftOperator \sum_{k\in\naturals}\coefficientEventualBoundaryScheme_k \zTransformed{\globalTruncationError}_{k}(\timeShiftOperator) - \sum_{k\in\naturals}\coefficientEventualBoundarySchemeOld_k \zTransformed{\globalTruncationError}_{k}(\timeShiftOperator)= \timeShiftOperator
    \qquad \text{and} \\ 
    (\timeShiftOperator^2 - 1)\zTransformed{\globalTruncationError}_{\indexSpace}(\timeShiftOperator) - \courantNumber\timeShiftOperator\zTransformed{\globalTruncationError}_{\indexSpace-1}(\timeShiftOperator) + \courantNumber\timeShiftOperator\zTransformed{\globalTruncationError}_{\indexSpace+1}(\timeShiftOperator) = 0, \quad \indexSpace\geq 1.
\end{multline}
As any linear recurrence relation, the general solution can be sought solving the associated characteristic equation, obtained replacing $\zTransformed{\globalTruncationError}_{\indexSpace}(\timeShiftOperator)$ by the geometric term $\fourierShift(\timeShiftOperator)^{\indexSpace}$.
Of particular usefulness is the root $\stableRoot$ such that $\stableRoot(\timeShiftOperator)\in\unitDisk$ for $\timeShiftOperator\in\neighborhoodInfinity$, which explicitly reads
\begin{equation*}
    \stableRoot(\timeShiftOperator) = \frac{1}{2\courantNumber\timeShiftOperator}(1-\timeShiftOperator^2 + \sqrt{\timeShiftOperator^4 + 2(2\courantNumber^2-1)\timeShiftOperator^2 + 1}).
\end{equation*}
The $L^2$ stable solution to the bulk equation reads $\zTransformed{\globalTruncationError}_{\indexSpace}(\timeShiftOperator) = \coefficientStableSolution(\timeShiftOperator)\stableRoot(\timeShiftOperator)^{\indexSpace}$ for $\indexSpace\geq 0$, where the coefficient $\coefficientStableSolution(\timeShiftOperator)$ is found by enforcing the boundary condition, \idEst{} the first equation of \eqref{eq:ztransformedProblem}:
\begin{equation*}
    \zTransformed{\globalTruncationError}_{\indexSpace}(\timeShiftOperator) = \displaystyle \frac{\timeShiftOperator \, \stableRoot(\timeShiftOperator)^{\indexSpace}}{\timeShiftOperator^2  - \timeShiftOperator \displaystyle \sum_{k\in\naturals}\coefficientEventualBoundaryScheme_k \stableRoot(\timeShiftOperator)^{k} - \displaystyle \sum_{k\in\naturals}\coefficientEventualBoundarySchemeOld_k \stableRoot(\timeShiftOperator)^{k}}.
\end{equation*}

\begin{lemma}[Branch points of $\zTransformed{\globalTruncationError}_{\indexSpace}(\timeShiftOperator)$]\label{lemma:BranchPoints}
    Let $|\courantNumber|<1$, then the function $\complex \ni \timeShiftOperator\mapsto \zTransformed{\globalTruncationError}_{\indexSpace}(\timeShiftOperator) \in \complex$ has four branch points on $\unitCircle$, corresponding to the zeros of $\timeShiftOperator\mapsto \timeShiftOperator^4 + 2(2\courantNumber^2-1)\timeShiftOperator^2 + 1$, given by 
    \begin{equation*}
        e^{\pm i\branchPointAngle}
        \quad \text{and}\quad
        e^{\pm i(\branchPointAngle - \pi)}, 
        \qquad 
        \text{with}
        \quad
        \branchPointAngle
        \definitionEquality
        \tfrac{1}{2}\arccos(1-2\courantNumber^2)
        \in [0, \tfrac{\pi}{2}).
    \end{equation*}
\end{lemma}
The proof of the previous lemma can be found in \cite{bellotti2025perfectly}.
Using the inverse $\timeShiftOperator$-transform and simple algebraic manipulations, we obtain 
\begin{align}
        {\globalTruncationError}_{\indexSpace}^{\indexTime}
        =
        \oint_C 
        g(\timeShiftOperator)
        {\timeShiftOperator^{\indexTime}\stableRoot(\timeShiftOperator)^{\indexSpace}}
        \differential{\timeShiftOperator}
        = &\oint_C g(\timeShiftOperator)e^{\indexTime f(\timeShiftOperator; \nu)}\differential{\timeShiftOperator}, \label{eq:definition_Error_LeapFrog}
        \\
        &\text{where}\quad
        \begin{cases}
            g(\timeShiftOperator)
            \definitionEquality
            (2\pi i)^{-1}
            \Bigl ( \timeShiftOperator^2  - \timeShiftOperator \displaystyle \sum_{k\in\naturals}\coefficientEventualBoundaryScheme_k \stableRoot(\timeShiftOperator)^{k} - \displaystyle \sum_{k\in\naturals}\coefficientEventualBoundarySchemeOld_k \stableRoot(\timeShiftOperator)^{k}\Bigr )^{-1}, \\
            f(\timeShiftOperator; \nu) \definitionEquality \log(\timeShiftOperator) + \nu \log(\stableRoot(\timeShiftOperator)),
        \end{cases}\label{eq:definitionFandG}
\end{align}
and $\nu = \indexSpace/\indexTime$ is a placeholder for the ratio between space and time index.
In all the rest of the paper, we use the principal determination of the logarithm and the argument.
We see in the remark below that $\nu$ can be considered as a \strong{group velocity} associated to modes which are saddle points.

\begin{remark}[On the link between group velocity and saddle points]
    Let $\nu\geq 0$ be given.
    In a steepest descent approximation \eqref{eq:definition_Error_LeapFrog} for $\indexTime\gg 1$, we look for the points $\saddlePoint(\nu)\in\complex$ such that $f'(\saddlePoint(\nu); \nu) = 0$.
    This results in
    \begin{equation*}
        \nu = -\frac{\stableRoot(\saddlePoint(\nu))}{\saddlePoint(\nu) \stableRoot'(\saddlePoint(\nu))},
    \end{equation*}
   whose right-hand side is nothing but the definition of group velocity for a mode $\saddlePoint(\nu)\in\unitCircle$ according to \cite[Equation (3.18)]{trefethen1984instability}.
   If a saddle point with group velocity $\nu = 0$ exists \cite{trefethen1982group}, such mode is said to be ``\strong{glancing}'' \cite{coulombel2015fully}.
\end{remark}
\begin{remark}[On the link between group velocity and diagonals in multivariate analytic combinatorics]
    If we also take the $\timeShiftOperator$-transform of $\globalTruncationError_{\indexSpace}^{\indexTime}$ in space, we obtain $\doubleZTransformed{\globalTruncationError}(\timeShiftOperator, \fourierShift)$, which reads (see \eqref{eq:doubleGeneratingFunction}):
    \begin{multline*}
    \doubleZTransformed{\globalTruncationError}(\timeShiftOperator^{-1}, \fourierShift^{-1}) 
    =
    \frac{\mathscr{P}(\timeShiftOperator, \fourierShift)}{\mathscr{Q}(\timeShiftOperator, \fourierShift)}
    \quad \text{with}\quad 
    \mathscr{P}(\timeShiftOperator, \fourierShift) = \timeShiftOperator({1 - \timeShiftOperator^2 + \courantNumber\timeShiftOperator(\fourierShift^{-1}   + \stableRoot(\timeShiftOperator^{-1}))})
    \\
    \text{and}\quad
    \mathscr{Q}(\timeShiftOperator, \fourierShift) = 
    {(1 - \timeShiftOperator\sum_{k\geq 0}\coefficientEventualBoundaryScheme_k \stableRoot(\timeShiftOperator^{-1})^k - \timeShiftOperator^2\sum_{k\geq 0}\coefficientEventualBoundarySchemeOld_k\stableRoot(\timeShiftOperator^{-1})^k)(1-\timeShiftOperator^2 - \courantNumber\timeShiftOperator (\fourierShift-\fourierShift^{-1}))}.
\end{multline*}
Now, considering $p, q \in\naturals^{*}$ coprime---according to \cite[Definition 3.12]{melczer2021invitation}---analyze the $(q, p)$-\strong{diagonal} of $\doubleZTransformed{\globalTruncationError}(\timeShiftOperator^{-1}, \fourierShift^{-1})$ essentially means to study the asymptotics of the one-index sequence $\tilde{\indexTime}\mapsto\globalTruncationError_{p\tilde{\indexTime}}^{q\tilde{\indexTime}}$.
According to \cite[Theorem 5.1]{melczer2021invitation}, the main contributions to the asymptotics for $\tilde{\indexTime}\to+\infty$ come from  $(\timeShiftOperator, \fourierShift) \in \complex^2$ fulfilling
\begin{equation}\label{eq:tmpNew}
    \mathscr{Q}(\timeShiftOperator, \fourierShift) = 0
    \qquad \text{and}\qquad 
    p\timeShiftOperator\partial_{\timeShiftOperator}\mathscr{Q}(\timeShiftOperator, \fourierShift)
    -q\fourierShift\partial_{\fourierShift}\mathscr{Q}(\timeShiftOperator, \fourierShift)= 0.
\end{equation}
Looking for solutions $\fourierShift = \fourierShift(\timeShiftOperator)$ of the first equation $\mathscr{Q}(\timeShiftOperator, \fourierShift(\timeShiftOperator)) = 0$ (thus solutions of the characteristic equation of the bulk scheme), and taking the total derivative in $\timeShiftOperator$, provide 
\begin{equation*}
    \partial_{\timeShiftOperator} \mathscr{Q}(\timeShiftOperator, \fourierShift(\timeShiftOperator)) + \partial_{\fourierShift} \mathscr{Q}(\timeShiftOperator, \fourierShift(\timeShiftOperator)) \fourierShift'(\timeShiftOperator) = 0, 
    \qquad \text{hence}\qquad 
    \frac{\partial_{\fourierShift} \mathscr{Q}(\timeShiftOperator, \fourierShift(\timeShiftOperator))}{\partial_{\timeShiftOperator} \mathscr{Q}(\timeShiftOperator, \fourierShift(\timeShiftOperator))}
    =
    -\frac{1}{\fourierShift'(\timeShiftOperator) }.
\end{equation*}
Plugging it into the second equation of \eqref{eq:tmpNew}, we obtain
\begin{equation*}
    \frac{p}{q}
    =
    -\frac{\fourierShift(\timeShiftOperator)}{\timeShiftOperator \fourierShift'(\timeShiftOperator) },
\end{equation*}
which states that $p/q$ is the group velocity of the mode $(\timeShiftOperator, \fourierShift(\timeShiftOperator))$ fulfilling the characteristic equation.
\end{remark}

\begin{figure}[h]
\begin{center}
    \begin{tikzpicture}[thick]

  \def\xr{3} \def\yr{3}

  \draw [->] (-\xr-1,0) -- (\xr+1,0) node [above left] {$\textnormal{Re}$};
  \draw [->] (0,-\yr-1) -- (0,\yr+1) node [left] {$\textnormal{Im}$};


  \draw[color=gray] (0, 0) circle (\yr);


  \node (bp1) at ({\yr*cos(60)},{\yr*sin(60)}) {};
  \node (bp2) at ({\yr*cos(60-180)},{\yr*sin(60-180)}) {};
  \node (bp3) at ({\yr*cos(-60)},{\yr*sin(-60)}) {};
  \node (bp4) at ({\yr*cos(-60+180)},{\yr*sin(-60+180)}) {};

  \draw (bp1) node[] {$\times$} node[above right] {$e^{i\branchPointAngle}$};
  \draw (bp2) node[] {$\times$} node[below left] {$-e^{i\branchPointAngle}$};
  \draw (bp3) node[] {$\times$} node[below right] {$e^{-i\branchPointAngle}$};
  \draw (bp4) node[] {$\times$} node[above left] {$-e^{-i\branchPointAngle}$};


  \draw[dashed] (bp1.center) -- (bp4.center);
  \draw[dashed] (bp2.center) -- (bp3.center);
  

  \draw[->, color=YellowGreen] ({\yr*cos(-80+180-2)},{\yr*sin(-80+180-2)})  arc (-80+180-2:90:\yr) node[midway,below] {$\nu\searrow0$};
  \draw[->, color=SpringGreen] ({\yr*cos(-80+180+2)},{\yr*sin(-80+180+2)})  arc (-80+180+2:180:\yr) node[midway,right] {$\nu\nearrow|\courantNumber|$};


  \node (sp1) at ({\yr*cos(30)},{\yr*sin(30)}) {};
  \node (sp2) at ({\yr*cos(30-180)},{\yr*sin(30-180)}) {};
  \node (sp3) at ({\yr*cos(-30)},{\yr*sin(-30)}) {};
  \node (sp4) at ({\yr*cos(-30+180)},{\yr*sin(-30+180)}) {};

  \draw (sp1) node[] {$\bullet$} node[above right] {$e^{i\vartheta_{\textnormal{SP}}}$};
  \draw (sp2) node[] {$\bullet$} node[below left] {$-e^{i\vartheta_{\textnormal{SP}}}$};
  \draw (sp3) node[] {$\bullet$} node[below right] {$e^{-i\vartheta_{\textnormal{SP}}}$};
  \draw (sp4) node[] {$\bullet$} node[above left] {$-e^{-i\vartheta_{\textnormal{SP}}}$};

  \draw[->, color=BrickRed] ({\yr*cos(28)},{\yr*sin(28)})  arc (28:0:\yr) node[midway,right] {$\nu\nearrow|\courantNumber|$};
  \draw[->, color=NavyBlue] ({\yr*cos(32)},{\yr*sin(32)})  arc (32:58:\yr) node[midway,left] {$\nu\searrow 0$};


   \node (ks1) at ({\yr*cos(-80+180)},{\yr*sin(-80+180)}) {};
   \node (ks4) at ({\yr*cos(-80)},{\yr*sin(-80)}) {};

  \draw (ks1) node[color=ForestGreen] {$\diamond$} node[above, color=ForestGreen] {$\stableRoot(e^{i\vartheta_{\textnormal{SP}}}) = -e^{-i\ksAtsaddlePointAngle}$};
  \draw (ks4) node[color=ForestGreen] {$\diamond$} node[below, color=ForestGreen] {$\stableRoot(-e^{i\vartheta_{\textnormal{SP}}}) = e^{-i\ksAtsaddlePointAngle}$};

\end{tikzpicture}
\end{center}\caption{\label{fig:tikzSaddlePoints}Schematic representation of branch points (crosses) and possible branch cuts (dashed lines); saddle points (solid dots); and associated values of $\stableRoot$ (empty diamonds). Trends for $\nu\to 0$ and $\nu\to|\courantNumber|$ are also illustrated.}
\end{figure}
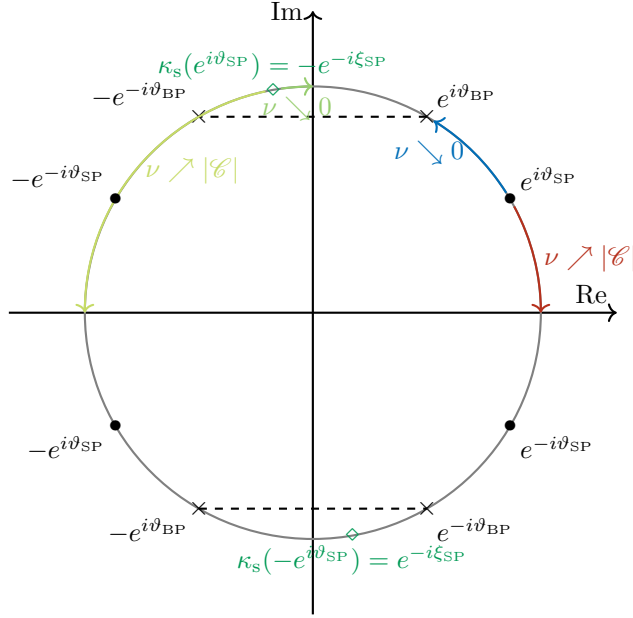

\begin{lemma}[Saddle points]\label{lemma:saddlePoints}
    Assume $-1<\courantNumber<0$.
    Let $\nu\geq 0$.
    Then, the saddle points of the function $\timeShiftOperator \mapsto f(\timeShiftOperator; \nu)$ are as follows.
    \begin{itemize}
        \item For $\nu = 0$, associated to a \strong{near-wall} zone: no finite saddle points.
        \item For $\nu\in (0, |\courantNumber|)$, associated to a \strong{transition} zone: four non-degenerate saddle points on $\unitCircle$. 
        More precisely, these saddle points are 
        \begin{equation}\label{eq:saddlePointsPhase}
            e^{\pm i \saddlePointAngle}
            \quad \text{and}\quad
            e^{\pm i (\saddlePointAngle - \pi)}
            , 
            \qquad \text{where}\qquad
            \saddlePointAngle = \saddlePointAngle(\nu)
            \definitionEquality
            \tfrac{1}{2}
            \arccos
            \Bigl( 
            \frac{1+\nu^2 - 2\courantNumber^2}{1-\nu^2}
            \Bigr )
            \in (0, \branchPointAngle).
        \end{equation}
        Moreover, we have 
        \begin{multline}\label{eq:ksSaddlePoints}
            \stableRoot(e^{\pm i \saddlePointAngle}) = e^{\mp i (\ksAtsaddlePointAngle- \pi)}
            \quad \text{and}\quad 
            \stableRoot(e^{\pm i (\saddlePointAngle- \pi)}) = e^{\mp i\ksAtsaddlePointAngle}, 
            \\
            \text{where}\qquad 
            \ksAtsaddlePointAngle = \ksAtsaddlePointAngle(\nu) \definitionEquality 
            \arccos\Bigl (-\frac{\nu}{\courantNumber}\sqrt{\frac{1-\courantNumber^2}{1-\nu^2}}\Bigr )
            \in (0, \tfrac{\pi}{2}).
        \end{multline}
        We deduce that $f(e^{\pm i \saddlePointAngle}; \nu), f(e^{\pm i (\saddlePointAngle- \pi)}; \nu) \in i\reals$.
        Finally, the saddle points are non-degenerate: $f''(e^{\pm i \saddlePointAngle}; \nu) \neq 0$ and $f''(e^{\pm i (\saddlePointAngle- \pi)}; \nu) \neq 0$.
        
        \item For $\nu = -\courantNumber$, associated to a \strong{front} zone: two degenerate (double) saddle points on $\unitCircle$.
        
        These saddle points are 
        \begin{equation*}
            \pm 1, 
            \qquad \text{with}\qquad 
            \stableRoot(\pm 1) = \mp 1
            \quad \text{and}\quad 
            f''(\pm 1; -\courantNumber) = 0.
        \end{equation*}
        
        \item For $\nu\in(-\courantNumber, 1)$, associated to a zone \strong{ahead-of-the-front}: saddle points are real, two in $\unitDisk$ and two in $\neighborhoodInfinity$. 
    \end{itemize}
\end{lemma}
\Cref{lemma:saddlePoints} on \strong{degeneracy/non-degeneracy} of saddle points share some similarities with the \strong{Taylor expansion of the logarithm of symbols} in \cite{thomee1965stability}, relevant when no boundary is present.
This point will be made clearer once schemes on $\relatives$ are discussed, \confer{} \Cref{sec:GreenFunction}, and the reader is referred to \Cref{rem:linkwithThomee} and \ref{rem:sameSP} on this concern.
Moreover, as no saddle point exists for $\nu = 0$, there are no glancing modes.
This is different from the case without boundary, see \Cref{sec:GreenFunction}.
\begin{remark}[Saddle points away from branch points, and coalescence]
    When $\nu \in (0,|\courantNumber|]$, there is no issue when deforming contours in $\neighborhoodInfinity$ to pass arbitrarily close to the saddle points, as they do no coincide with the branch points.
    Moreover, note that $ \lim_{\nu\searrow 0}\saddlePointAngle(\nu) = \branchPointAngle$.
    On the other hand, the celebrated \strong{coalescence} takes place for $\nu\nearrow|\courantNumber|$, that is $\lim_{\nu\nearrow|\courantNumber|}\saddlePointAngle(\nu) = 0$.
    The overall situation is depicted in \Cref{fig:tikzSaddlePoints}.
\end{remark}

\subsubsection{Stable boundary conditions}\label{sec:stable}

Before giving precise asymptotic results for $\globalTruncationError_{\indexSpace}^{\indexTime}$ for large $\indexTime$, let us state results on the asymptotics of related quantities summed in $\indexSpace\in\naturals$.
We start by a result on the zero and first-order moments in space.
This is of little use, for instance to estimate the $L^1$ norm of the sequence, as the sequence strongly oscillates, \confer{} \Cref{fig:my_label} and \ref{fig:cumulativeSum}.
\begin{proposition}[Moments in space]\label{prop:momentsInSpaceStable}
Consider $\indexTime\gg 1$.
    Let \Cref{ass:stableBulk} and \ref{ass:stableBC}  be fulfilled.
    Then, the zero order moment of $\globalTruncationError_{\indexSpace}^{\indexTime}$ solution of \eqref{eq:zeroInitialError}--\eqref{eq:schemeErrorInitial}--\eqref{eq:schemeErrorEventual} is given by
    \begin{equation*}
    \sum_{\indexSpace\geq 0}\globalTruncationError_{\indexSpace}^{\indexTime}
    = 
    \frac{\courantNumber(-1)^{\indexTime}}{1 + \sum_{k\geq 0}(\coefficientEventualBoundaryScheme_k-\coefficientEventualBoundarySchemeOld_k)} + \bigO{\indexTime^{-3/2}}
    \quad \text{and}\quad
    \sum_{\indexSpace\geq 0}(-1)^{\indexSpace}\globalTruncationError_{\indexSpace}^{\indexTime}
    = 
    \frac{\courantNumber}{\sum_{k\geq 0}(-1)^{k}(\coefficientEventualBoundaryScheme_k+\coefficientEventualBoundarySchemeOld_k) - 1} + \bigO{\indexTime^{-3/2}},
\end{equation*}
whereas the first order moment satisfies
\begin{multline*}
\sum_{\indexSpace\geq 0}\indexSpace\globalTruncationError_{\indexSpace}^{\indexTime}= 
-\frac{\courantNumber^2 (-1)^{\indexTime}}{1+\sum_{k\geq 0}(\coefficientEventualBoundaryScheme_k-\coefficientEventualBoundarySchemeOld_k)}
\indexTime\\
+\frac{\courantNumber(-1)^{\indexTime}}{1+\sum_{k\geq 0}(\coefficientEventualBoundaryScheme_k-\coefficientEventualBoundarySchemeOld_k)}
    \overbrace{\Biggl ( 
    \courantNumber
    +\frac{\sum_{k\geq 0}k(\coefficientEventualBoundaryScheme_k-\coefficientEventualBoundarySchemeOld_k) + \courantNumber \sum_{k\geq 0}(2\coefficientEventualBoundarySchemeOld_k-\coefficientEventualBoundaryScheme_k)}{1+\sum_{k\geq 0}(\coefficientEventualBoundaryScheme_k-\coefficientEventualBoundarySchemeOld_k)}
    \Biggr )}^{=0\text{ for \Cref{ex:upwind} and \ref{ex:upwindWithDiffusion}}}
    + \bigO{\indexTime^{-3/2}}.
\end{multline*}
\end{proposition}
Unlike the $L^1$ norm, which is not well-described by the zero-order moment, we can find the asymptotic of the $L^2$ norm, given as follows.
\begin{proposition}[$L^2$ norm asymptotically constant]\label{prop:asymptoticL2Stable}
     Let \Cref{ass:stableBulk} and \ref{ass:stableBC}  be fulfilled.
     Then, the $L^2$ norm of $\globalTruncationError_{\indexSpace}^{\indexTime}$, solution of \eqref{eq:zeroInitialError}--\eqref{eq:schemeErrorInitial}--\eqref{eq:schemeErrorEventual}, satisfies
     \begin{equation}\label{eq:asymptoticL2Norm}
         \lim_{\indexTime\to+\infty}
         \lVert \globalTruncationError^{\indexTime}\rVert_2
         =
         \Biggl ( 
         -2\pi
         \int_0^{2\pi}
         \underbrace{g(e^{i\vartheta})g(e^{-i\vartheta})}_{\text{boundary dep.}}
         \underbrace{\frac{
         \sqrt{e^{4 i\vartheta}+2(2\courantNumber^2-1)e^{2 i\vartheta}+1}}{
         e^{2 i\vartheta}+1}}_{\text{(bulk-dependent) weight}}
         \differential{\vartheta}
         \Biggr )^{1/2}.
     \end{equation}
     Moreover, we have that for large $\indexTime\gg 1$, $\lVert \globalTruncationError^{\indexTime}\rVert_2 = \lim_{k\to+\infty}
         \lVert \globalTruncationError^{k}\rVert_2 + \bigO{\indexTime^{-3/2}}$.
\end{proposition}
Equation \eqref{eq:asymptoticL2Norm} can be regarded as a \strong{weighted Parseval identity} in the asymptotic limit, since it involves the term $g(e^{i\vartheta})g(e^{-i\vartheta})$.
The integral on the right-hand side is generally hard---although not impossible for given boundary conditions---to compute analytically. 
Thus, it is conveniently approximated using quadrature formul\ae{}.
In the case of Dirichlet boundary condition in \Cref{ex:Dirichlet}, \eqref{eq:asymptoticL2Norm} explicitly reads---after consistently dealing with the multi-valuedness of the square root and using \cite[Identity 17.7.20]{Abramowitz1964} on degenerate elliptic integrals of third kind:
\begin{equation*}
    \lim_{\indexTime\to+\infty}
         \lVert \globalTruncationError^{\indexTime}\rVert_2
         =
    \sqrt{1-\sqrt{1-\courantNumber^2}}.
\end{equation*}

\begin{proposition}[Near-wall zone]\label{prop:nearWallStable}
Consider $\globalTruncationError_{\indexSpace}^{\indexTime}$ solution of \eqref{eq:zeroInitialError}--\eqref{eq:schemeErrorInitial}--\eqref{eq:schemeErrorEventual} and let \Cref{ass:stableBulk} and \ref{ass:stableBC} hold.
    Let $\indexSpace\in\naturals$ be fixed.
    In the limit $\indexTime\gg1$, a good approximation of $\globalTruncationError_{\indexSpace}^{\indexTime}$ is given by
         \begin{align}
       \globalTruncationError_{\indexSpace}^{\indexTime}
       \sim 
       \sqrt{\frac{2}{\pi |\courantNumber|}} (1-\courantNumber^2)^{1/4}
       \Biggl ( 
        &\textnormal{Re}\Biggl ( 
        \frac{e^{i((1-\indexTime)\branchPointAngle + (1-\indexSpace)\frac{\pi}{2} - \frac{\pi}{4})}}{1-\mathscr{G}_+^0(\branchPointAngle)}
        \Biggl ( 
       \frac{
       \mathscr{G}_+^1(\branchPointAngle)}{1-\mathscr{G}_+^0(\branchPointAngle)}
       +  \indexSpace
       \Biggr )
        \Biggr )\nonumber\\
    +(-1)^{\indexTime}&\textnormal{Re}
    \Biggl ( 
    \frac{e^{i((\indexTime-1)\branchPointAngle + (1-\indexSpace)\frac{\pi}{2}+\frac{\pi}{4})}}{1+\mathscr{G}_-^{0}(-\branchPointAngle)}
    \Biggl ( 
       -\frac{
       \mathscr{G}_-^1(-\branchPointAngle)}{1+\mathscr{G}_-^{0}(-\branchPointAngle)}
       +  \indexSpace
       \Biggr )
       \Biggr )
       \Biggr )\indexTime^{-3/2}
       +\bigO{\indexTime^{-5/2}}.\label{eq:nearWallStable}
   \end{align}
  where
  \begin{equation*}
      \mathscr{G}_{\pm}^0(\vartheta)
      \definitionEquality
      e^{i\vartheta}\sum_{k\geq 0}\coefficientEventualBoundaryScheme_ke^{-ik\frac{\pi}{2}}
      \pm e^{2i\vartheta}\sum_{k\geq 0}\coefficientEventualBoundarySchemeOld_k e^{-ik\frac{\pi}{2}} 
      \quad \text{and}\quad
      \mathscr{G}_{\pm}^1(\vartheta)
      \definitionEquality
      e^{i\vartheta}\sum_{k\geq 1}k \coefficientEventualBoundaryScheme_ke^{-ik\frac{\pi}{2}}
      \pm e^{2i\vartheta}\sum_{k\geq 1}k \coefficientEventualBoundarySchemeOld_k e^{-ik\frac{\pi}{2}}.
  \end{equation*}
  
   In the particular, for the upwind scheme of \Cref{ex:upwind}, the previous expression reduces to
   \begin{align}
        \globalTruncationError_{\indexSpace}^{\indexTime} &\sim \sqrt{\frac{2}{\pi |\courantNumber|}} (1-\courantNumber^2)^{1/4}\nonumber\\
        \times
        \Biggl (
        &\frac{\courantNumber}{1+\courantNumber} 
        \Biggl (\frac{1}{(\sqrt{1-\courantNumber}-\sqrt{1+\courantNumber})^2} + \frac{(-1)^{\indexTime+\indexSpace-1}}{(\sqrt{1-\courantNumber}+\sqrt{1+\courantNumber})^2} \Biggr )
        \sin\Bigl (\branchPointAngle\indexTime +\indexSpace\frac{\pi}{2} - \frac{\pi}{4}\Bigr )\label{eq:nearWallUpwind}\\
        +&\frac{\indexSpace}{\sqrt{1+\courantNumber}} 
        \Biggl (\frac{1}{\sqrt{1-\courantNumber}-\sqrt{1+\courantNumber}} + \frac{(-1)^{\indexTime+\indexSpace-1}}{\sqrt{1-\courantNumber}+\sqrt{1+\courantNumber}} \Biggr )
        \cos\Bigl (\branchPointAngle\indexTime + \indexSpace\frac{\pi}{2} - \frac{\pi}{4}\Bigr )
        \Biggr )
        \indexTime^{-3/2} 
        + \bigO{\indexTime^{-5/2}}.\nonumber
    \end{align}
\end{proposition}

\begin{proposition}[Transition zone]\label{prop:transitionZoneStable}
    Consider $\globalTruncationError_{\indexSpace}^{\indexTime}$ solution of \eqref{eq:zeroInitialError}--\eqref{eq:schemeErrorInitial}--\eqref{eq:schemeErrorEventual} and let \Cref{ass:stableBulk} and \ref{ass:stableBC} hold.
    Let $0< \indexSpace < |\courantNumber|\indexTime$ so that $\nu = \frac{\indexSpace}{\indexTime} \in (0, |\courantNumber|)$.
    With the notations introduced in \Cref{lemma:saddlePoints}, set
    \begin{multline*}
        \sigma = \sigma(\nu) \definitionEquality f''(e^{i\saddlePointAngle(\nu)}; \nu)= -\Bigl (1+\frac{1}{\nu}\Bigr )e^{-2i\saddlePointAngle} \\-  \frac{e^{i(\ksAtsaddlePointAngle-6\saddlePointAngle)}((2\courantNumber^2-1)e^{6i\saddlePointAngle}+3e^{4i\saddlePointAngle}+3(2\courantNumber^2-1)e^{2i\saddlePointAngle} + 1 + 8\nu^3 e^{3i\saddlePointAngle} \Bigl (\frac{1-\courantNumber^2}{1-\nu^2}\Bigr )^{3/2})}{{8 \courantNumber \nu^2 \Bigl (\frac{1-\courantNumber^2}{1-\nu^2}\Bigr )^{3/2}}}\in \complex,
    \end{multline*}
    which is only dictated by the bulk scheme.
    Then, in the limit $\indexTime\gg 1$, $\globalTruncationError_{\indexSpace}^{\indexTime}$ is well approximated by 
    \begin{multline}\label{eq:expansionBeyndShock}
        \globalTruncationError_{\indexSpace}^{\indexTime} \sim \sqrt{\frac{2}{\pi |\sigma|}}\\ \times
        \Biggl (\textnormal{cos}\left((\indexTime-1)\saddlePointAngle-\indexSpace \ksAtsaddlePointAngle-\tfrac{1}{2}\textnormal{Arg}(\sigma))\right) \Bigl( \frac{\mathscr{G}_{\textnormal{R}}^{\textnormal{Re}}}{(\mathscr{G}_{\textnormal{R}}^{\textnormal{Re}})^2+(\mathscr{G}_{\textnormal{R}}^{\textnormal{Im}})^2}\,(-1)^\indexSpace -  \frac{\mathscr{G}_{\textnormal{L}}^{\textnormal{Re}}}{(\mathscr{G}_{\textnormal{L}}^{\textnormal{Re}})^2+(\mathscr{G}_{\textnormal{L}}^{\textnormal{Im}})^2} (-1)^\indexTime \Bigr)\\+ \textnormal{sin} \left((\indexTime-1)\saddlePointAngle-\indexSpace \ksAtsaddlePointAngle-\tfrac{1}{2}\textnormal{Arg}(\sigma))\right)  \Bigl( \frac{\mathscr{G}_{\textnormal{R}}^{\textnormal{Im}}}{(\mathscr{G}_{\textnormal{R}}^{\textnormal{Re}})^2+(\mathscr{G}_{\textnormal{R}}^{\textnormal{Im}})^2}\,(-1)^\indexSpace -  \frac{\mathscr{G}_{\textnormal{L}}^{\textnormal{Im}}}{(\mathscr{G}_{\textnormal{L}}^{\textnormal{Re}})^2+(\mathscr{G}_{\textnormal{L}}^{\textnormal{Im}})^2} (-1)^\indexTime \Bigr) \Biggr ) \indexTime^{-1/2} ,
    \end{multline}
    where $\mathscr{G}_{\textnormal{R}}^{\textnormal{Re}}$ and $\mathscr{G}_{\textnormal{R}}^{\textnormal{Im}}$ (respectively, $\mathscr{G}_{\textnormal{L}}^{\textnormal{Re}}$ and $\mathscr{G}_{\textnormal{L}}^{\textnormal{Im}}$) are the real and the imaginary part of $\mathscr{G}_{\textnormal{R}}(\saddlePointAngle, \ksAtsaddlePointAngle)$ (respectively $\mathscr{G}_{\textnormal{L}}(\saddlePointAngle, \ksAtsaddlePointAngle)$), defined by
    \begin{multline*}
     \mathscr{G}_{\textnormal{R}}(\vartheta, \ksAtsaddlePointAngleplain) \definitionEquality e^{i\vartheta}  - \sum_{k\geq 0}(-1)^k \coefficientEventualBoundaryScheme_k e^{-i k \ksAtsaddlePointAngleplain} -e^{-i\vartheta}\sum_{k\geq 0}(-1)^k \coefficientEventualBoundarySchemeOld_k  e^{-i k\ksAtsaddlePointAngleplain}
     \\ 
     \text{and}
     \qquad 
     \mathscr{G}_{\textnormal{L}}(\vartheta, \ksAtsaddlePointAngleplain) \definitionEquality  e^{i\vartheta}  + \sum_{k\geq 0}\coefficientEventualBoundaryScheme_k e^{-i k \ksAtsaddlePointAngleplain} -e^{-i\vartheta}\sum_{k\geq 0} \coefficientEventualBoundarySchemeOld_k  e^{-i k\ksAtsaddlePointAngleplain}.
\end{multline*}
    We stress that  $\sigma$, $\saddlePointAngle$, and $\ksAtsaddlePointAngle$ are functions of $\nu$.
    
    In the particular, for the upwind scheme of \Cref{ex:upwind}, the previous expression reduces to
    \begin{equation}\label{eq:expansionBeyndShock_upwindScheme}
        \globalTruncationError_{\indexSpace}^{\indexTime} \sim \sqrt{\frac{2}{\pi |\sigma|}} \Biggl (
        \frac{(-1)^{\indexSpace}}{-(1+\courantNumber)+\sqrt{(1-\courantNumber^2)\frac{1+\nu}{1-\nu}}} 
        -\frac{(-1)^{\indexTime}}{(1+\courantNumber)+\sqrt{(1-\courantNumber^2)\frac{1+\nu}{1-\nu}}} \Biggr ) \cos((\indexTime - 1)\saddlePointAngle- \indexSpace \ksAtsaddlePointAngle-\tfrac{1}{2}\textnormal{Arg}(\sigma)) \indexTime^{-1/2}.
    \end{equation}
    This particular profile is made up of two self-similar (since depending---up to the scale factor---on $\indexTime$ and $\indexSpace$ only through $\nu$) envelopes
    \begin{equation}\label{eq:envelopes}
        \pm \sqrt{\frac{2}{\pi |\sigma|}} \frac{1}{\nu-\courantNumber}\sqrt{(1-\nu^2)\frac{1-\courantNumber}{1+\courantNumber}} \indexTime^{-1/2}
        \qquad \text{and}\qquad
        \pm \sqrt{\frac{2}{\pi |\sigma|}} \frac{1-\nu}{\nu-\courantNumber} \indexTime^{-1/2}
    \end{equation}
    with a modulation by $\cos((\indexTime-1)\saddlePointAngle- \indexSpace \ksAtsaddlePointAngle-\tfrac{1}{2}\textnormal{Arg}(\sigma))$.
\end{proposition}

\begin{proposition}[Front zone]\label{prop:frontZoneStable}
    Consider $\globalTruncationError_{\indexSpace}^{\indexTime}$ solution of \eqref{eq:zeroInitialError}--\eqref{eq:schemeErrorInitial}--\eqref{eq:schemeErrorEventual} and let \Cref{ass:stableBulk} and \ref{ass:stableBC} hold.
    Let $\indexTime \gg 1$ and $\indexSpace\in\naturals$ such that $\indexSpace+\courantNumber\indexTime = \bigO{1}$.
    Then, a good approximation of $\globalTruncationError_{\indexSpace}^{\indexTime}$ is given by 
    \begin{multline}\label{eq:approximationFrontShockByAiry}
        \globalTruncationError_{\indexSpace}^{\indexTime} \sim 
        \courantNumber
        \Biggl( 
        \frac{(-1)^{\indexTime}}{1+\sum_{k\geq 0}\coefficientEventualBoundaryScheme_k - \sum_{k\geq 0}\coefficientEventualBoundarySchemeOld_k}
        -
        \frac{(-1)^{\indexSpace}}{1-\sum_{k\geq 0}(-1)^k\coefficientEventualBoundaryScheme_k - \sum_{k\geq 0}(-1)^k\coefficientEventualBoundarySchemeOld_k}
        \Biggr )\\
        \times
        \frac{1}{(\frac{\courantNumber}{2}(\courantNumber^2-1) \indexTime)^{1/3}} \textnormal{Ai}\Biggl ( \frac{\indexSpace + \courantNumber\indexTime}{(\frac{\courantNumber}{2}(\courantNumber^2-1) \indexTime)^{1/3}} \Biggr ),
    \end{multline}
    where the boundary-condition dependent terms enclosed in the parentheses are well-defined thanks to \Cref{ass:stableBC}.
\end{proposition}

Let us now discuss the origin of \eqref{eq:approximationFrontShockByAiry} thanks to elementary computations with a steepest descent approximation featuring \strong{degenerate saddle points}.
Assume that $\courantNumber = -\frac{p}{q}\in\mathbb{Q}$ where $p, q \in\naturals^{*}$ are coprime.
\Cref{ass:stableBulk} gives $p<q$.
In this way, $|\courantNumber|\indexTime\in\naturals$ whenever $\indexTime\in q\naturals$.
We thus pose $\indexTime=q\tilde{\indexTime}$ with $\tilde{\indexTime}\in\naturals$ and $\indexSpace = p\tilde{\indexTime}$.
Inserting into \eqref{eq:approximationFrontShockByAiry}, we obtain 
\begin{equation*}
        \globalTruncationError_{p\tilde{\indexTime}}^{q\tilde{\indexTime}} \sim 
        -\frac{p}{q}
        \Biggl( 
        \frac{(-1)^{q\tilde{\indexTime}}}{1+\sum_{k\geq 0}\coefficientEventualBoundaryScheme_k - \sum_{k\geq 0}\coefficientEventualBoundarySchemeOld_k}
        -
        \frac{(-1)^{p\tilde{\indexTime}}}{1-\sum_{k\geq 0}(-1)^k\coefficientEventualBoundaryScheme_k - \sum_{k\geq 0}(-1)^k\coefficientEventualBoundarySchemeOld_k}
        \Biggr )
        \frac{1}{(-\frac{p}{2}(\frac{p^2}{q^2}-1) \tilde{\indexTime})^{1/3}} \textnormal{Ai}(0),
\end{equation*}
where $\coefficientEventualBoundaryScheme_k$ and $\coefficientEventualBoundarySchemeOld_k$ might also depend on $p$ and $q$ due to the possible dependence of the boundary scheme in the Courant number $\courantNumber$.
Setting $\nu = \indexSpace/\indexTime = p/q = -\courantNumber$, the saddle points of $f(\timeShiftOperator; \nu)$ are $\timeShiftOperator=\pm 1$, according to \Cref{lemma:saddlePoints}.
Taylor expansions around the saddle points give 
\begin{equation*}
    f(\timeShiftOperator; p/q)
    =
    i\pi\frac{p}{q}
    -\tfrac{1}{6}\frac{p^2-q^2}{p^2}  
    (\timeShiftOperator-1)^3 + \bigO{(\timeShiftOperator-1)^4}
    \quad \text{and}\quad 
    f(\timeShiftOperator; p/q)
    =
    i\pi
    +\tfrac{1}{6}\frac{p^2-q^2}{p^2} 
    (\timeShiftOperator+1)^3 + \bigO{(\timeShiftOperator+1)^4},
\end{equation*}
where the absence of second-order term indicates that we face  degenerate saddle points.
Let us now find (local) directions of steepest descent for each saddle point, and derive the approximation $\globalTruncationError_{p\tilde{\indexTime}}^{q\tilde{\indexTime}} \sim \mathscr{I}_{1}^{\tilde{\indexTime}} + \mathscr{I}_{-1}^{\tilde{\indexTime}} $, where $\mathscr{I}_{1}^{\tilde{\indexTime}} $ (respectively, $\mathscr{I}_{-1}^{\tilde{\indexTime}} $) is the contribution from the saddle point at $\timeShiftOperator = 1$ (respectively, at $\timeShiftOperator = -1$).
\begin{itemize}
    \item Consider the neighborhood of $\timeShiftOperator = 1$.
    Let $\rho>0$ and write 
    \begin{multline*}
        f(1+\rho e^{i\varphi}; p/q) = i\pi\frac{p}{q}
    -\tfrac{1}{6}\frac{p^2-q^2}{p^2}  
    \rho^3 e^{3i\varphi} + \bigO{\rho^4},
        \\
        \text{hence}\qquad 
        \begin{cases*}
            \textnormal{Re}(f(1+\rho e^{i\varphi}; p/q) ) =
    -\tfrac{1}{6} \frac{p^2-q^2}{p^2}  
    \rho^3 \cos(3\varphi) + \bigO{\rho^4}, \\
            \textnormal{Im}(f(1+\rho e^{i\varphi}; p/q) ) = \pi\frac{p}{q}
    -\tfrac{1}{6} \frac{p^2-q^2}{p^2}  
    \rho^3 \sin(3\varphi) + \bigO{\rho^4}.
        \end{cases*}
    \end{multline*}
    Taking into account that $-\tfrac{1}{6} \frac{p^2-q^2}{p^2}  
    \rho^3>0$, we look for rapid decay of the real part, hence $\cos(3\varphi) = -1$, and lack of oscillations, so $\sin(3\varphi) = 0$.
    This entails $\varphi = \frac{\pi}{3}(2k + 1)$ with $k\in\relatives$. We take $k = -1$, hence $\varphi = -\frac{\pi}{3}$ and $k = 0$, hence $\varphi = \frac{\pi}{3}$.
    We thus obtain, injecting the truncated third-order expansion into \eqref{eq:definition_Error_LeapFrog} where the path has been deformed to pass through the saddle point with the requested angles
    \begin{align*}
        \mathscr{I}_1^{\tilde{\indexTime}}&= g(1) e^{\tilde{\indexTime}i\pi p } \Biggl ( 
        \int_R^0 \textnormal{exp}\Bigl ( \tilde{\indexTime}\frac{q}{6p^2}(p^2-q^2)\rho^3 \Bigr ) e^{-i\frac{\pi}{3}}\differential{\rho} +  
        \int_0^R \textnormal{exp}\Bigl ( \tilde{\indexTime}\frac{q}{6p^2}(p^2-q^2)\rho^3 \Bigr ) e^{i\frac{\pi}{3}}\differential{\rho}\Biggr )\\
        &\sim\frac{(-1)^{p\tilde{\indexTime}}}{1-\sum_{k\geq 0}(-1)^k\coefficientEventualBoundaryScheme_k - \sum_{k\geq 0}(-1)^k\coefficientEventualBoundarySchemeOld_k}
        \frac{\sqrt{3}}{2\pi} \Bigl (\tilde{\indexTime}\frac{q}{6p^2}(q^2-p^2)\Bigr)^{-1/3}
        \int_0^{+\infty}e^{-\rho^3}\differential{\rho}\\
        &=\frac{p}{q} \frac{(-1)^{p\tilde{\indexTime}}}{1-\sum_{k\geq 0}(-1)^k\coefficientEventualBoundaryScheme_k - \sum_{k\geq 0}(-1)^k\coefficientEventualBoundarySchemeOld_k}
        \frac{1}{(-\frac{p}{2}(\frac{p^2}{q^2}-1) \tilde{\indexTime})^{1/3}} \underbrace{\frac{1}{3^{2/3}\Gamma(\frac{2}{3})}}_{=\textnormal{Ai}(0)},
    \end{align*}
    where the approximate equality is obtained by letting $R\to+\infty$  and considering a change of variable in the integral.
    The last equality relies on the fact that $\int_{0}^{+\infty}e^{-\rho^3}\differential{\rho}=\Gamma(4/3) = 2\pi/({3\sqrt{3}\Gamma(2/3)})$, using the Euler's reflection formula for the Gamma function, and on straightforward rearrangements of the terms.
    This equation is---without much surprise---\cite[Chapter VII, Equation (4.5)]{wong2001asymptotic}.
    \item For the neighborhood of $\timeShiftOperator = -1$, computations are analogous except for the fact that the directions of steepest descent are along 
        \begin{equation}\label{eq:anglesSteepestDescentZM1}
        \varphi = -\frac{2\pi}{3}
        \qquad \text{and}\qquad
        \varphi = \frac{2\pi}{3}
    \end{equation}
    due to the fact that the third-order term in the Taylor expansion has opposite sign compared to $\timeShiftOperator = 1$.
    Analogous computations yield 
    \begin{equation*}
        \mathscr{I}_{-1}^{\tilde{\indexTime}} \sim 
        -\frac{p}{q}
        \frac{(-1)^{q\tilde{\indexTime}}}{1+\sum_{k\geq 0}\coefficientEventualBoundaryScheme_k - \sum_{k\geq 0}\coefficientEventualBoundarySchemeOld_k}
        \frac{1}{(-\frac{p}{2}(\frac{p^2}{q^2}-1) \tilde{\indexTime})^{1/3}} \textnormal{Ai}(0).
\end{equation*}
\end{itemize}

\subsubsection{Unstable boundary conditions}

In \Cref{prop:nearWallUnstable} and \ref{prop:transitionUnstable} below,  the leading-order contribution in the asymptotic is simply given by the residue of $g(\timeShiftOperator)e^{\indexTime f(\timeShiftOperator; \nu)}$ at its simple pole $\timeShiftOperator=-1$.
The third one features this residue with a ``distortion'' induced by saddle points of $f(\timeShiftOperator; \nu)$ coalescing to this pole.
The residue is as follows.
\begin{lemma}[Residue of $g(\timeShiftOperator)$ at $\timeShiftOperator = -1$]
    Let \Cref{ass:stableBulk} and \ref{ass:unstableBC} hold.
    Then 
    \begin{equation}
    \label{eq:residue_g}
        2\pi i \, \textnormal{Res}_{-1}
        [g(\timeShiftOperator)]
        =-
        \Bigl (2
            + \sum_{k\geq 0}
            \coefficientEventualBoundaryScheme_k
            + \frac{1}{\courantNumber} \sum_{k\geq 0}
            k (\coefficientEventualBoundaryScheme_k-\coefficientEventualBoundarySchemeOld_k)\Bigr )^{-1},
    \end{equation}
    which is well-defined thanks to \eqref{eq:simpleZeroAtMinusOne}.
\end{lemma}
\begin{proof}
    Let $k \in\naturals$. We have that $\stableRoot(\timeShiftOperator)^k = 1 - k / \courantNumber(\timeShiftOperator + 1) + \bigO{(\timeShiftOperator+1)^2}$.
    This gives the claim.
\end{proof}

\begin{proposition}[Near-wall zone]\label{prop:nearWallUnstable}
    Consider $\globalTruncationError_{\indexSpace}^{\indexTime}$ solution of \eqref{eq:zeroInitialError}--\eqref{eq:schemeErrorInitial}--\eqref{eq:schemeErrorEventual} and let \Cref{ass:stableBulk} and \ref{ass:unstableBC} hold.
    Let $\indexSpace\in\naturals$ be fixed.
    In the limit $\indexTime\gg1$, a good approximation of $\globalTruncationError_{\indexSpace}^{\indexTime}$ is given by
    \begin{equation}\label{eq:nearWallUnstable}
        \globalTruncationError_{\,\indexSpace}^{\indexTime} = -(-1)^n \Bigl( 2 + \sum_{k\geq 0}\coefficientEventualBoundaryScheme_k + \frac{1}{\courantNumber}\sum_{k\geq 0}k(\coefficientEventualBoundaryScheme_k-\coefficientEventualBoundarySchemeOld_k)\Bigr)^{-1} + \tau_{\indexSpace}^{\indexTime},
    \end{equation}
    where the expression of $\tau_{\indexSpace}^{\indexTime}$ is given by the right-hand side of \eqref{eq:nearWallStable}, thus $\tau_{\indexSpace}^{\indexTime}=\bigO{\indexTime^{-3/2}}$.
\end{proposition}

\begin{proposition}[Transition zone]\label{prop:transitionUnstable}
    Consider $\globalTruncationError_{\indexSpace}^{\indexTime}$ solution of \eqref{eq:zeroInitialError}--\eqref{eq:schemeErrorInitial}--\eqref{eq:schemeErrorEventual} and let \Cref{ass:stableBulk} and \ref{ass:unstableBC} hold.
    Let $0< \indexSpace < |\courantNumber|\indexTime$ so that $\nu = \frac{\indexSpace}{\indexTime} \in (0, |\courantNumber|)$. Then, in the limit $\indexTime\gg 1$, $\globalTruncationError_{\indexSpace}^{\indexTime}$ is well approximated by 
    \begin{equation}\label{eq:unstable_error_transition}
    \globalTruncationError_{\,\indexSpace}^{\indexTime} = -(-1)^n \Bigl(2 + \sum_{k\geq 0}\coefficientEventualBoundaryScheme_k + \frac{1}{\courantNumber}\sum_{k\geq 0}k(\coefficientEventualBoundaryScheme_k-\coefficientEventualBoundarySchemeOld_k)\Bigr)^{-1} + \tau_{\indexSpace}^{\indexTime},
    \end{equation}
    where the expression of $\tau_{\indexSpace}^{\indexTime}$ is given by the right-hand side of \eqref{eq:expansionBeyndShock}, thus $\tau_{\indexSpace}^{\indexTime}=\bigO{\indexTime^{-1/2}}$.
\end{proposition}

\begin{proposition}[Front zone]\label{prop:FrontUnstable}
     Consider $\globalTruncationError_{\indexSpace}^{\indexTime}$ solution of \eqref{eq:zeroInitialError}--\eqref{eq:schemeErrorInitial}--\eqref{eq:schemeErrorEventual} and let \Cref{ass:stableBulk} and \ref{ass:unstableBC} hold. Let $\indexTime \gg 1$ and $\indexSpace\in\naturals$ such that $\indexSpace+\courantNumber\indexTime = \bigO{1}$.
    Then, a good approximation of $\globalTruncationError_{\indexSpace}^{\indexTime}$ is given by
    \begin{equation}\label{eq:frontUnstable}
        \globalTruncationError_{\indexSpace}^{\indexTime} \sim 
        -(-1)^{\indexTime}
        \Bigl( 2 + \sum_{k\geq 0}\coefficientEventualBoundaryScheme_k + \frac{1}{\courantNumber}\sum_{k\geq 0}k(\coefficientEventualBoundaryScheme_k-\coefficientEventualBoundarySchemeOld_k)\Bigr)^{-1}
        \Biggl ( 
        \frac{1}{3} - \int_0^{\frac{\indexSpace + \courantNumber\indexTime}{(\frac{\courantNumber}{2}(\courantNumber^2-1) \indexTime)^{1/3}}}
        \textnormal{Ai}(y)\differential{y}
        \Biggr ).
    \end{equation}
\end{proposition}

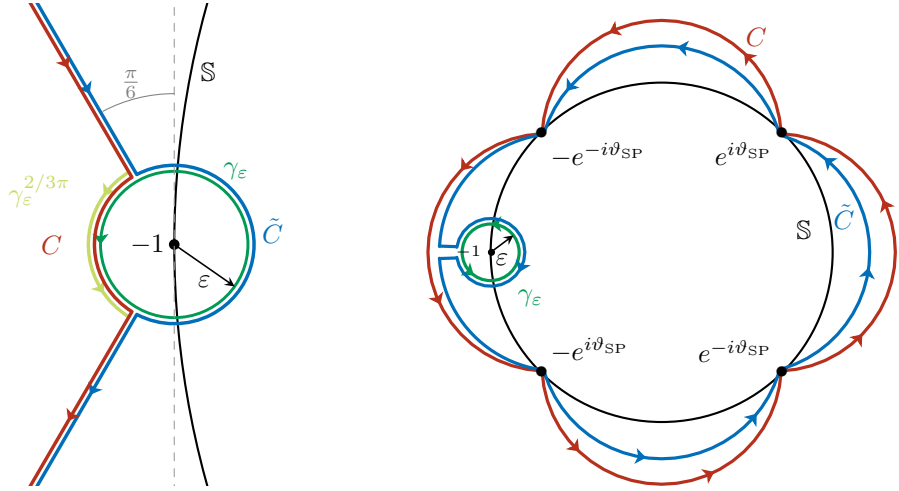
\begin{figure}
\begin{center}
    \begin{tikzpicture}[
        line cap=round, 
        line join=round,
        arrow inside/.style={
            postaction={decorate,decoration={
                markings,
                mark=at position #1 with {\arrow{Stealth[length=4pt, width=6pt, inset=1pt]}}
            }}
        }
    ]

    \coordinate (Pole) at (0,0);
    \def\eps{0.97}      
    \def\pathR{1.05}    
     \def\pathRR{1.13}    
    \def\off{0.02}     
    \def\armL{3.5}     
    \def\bigR{12.0}    
    \def\Lfig{2.0}     

    \clip (-\Lfig-1.2, -\Lfig-1.2) rectangle (\Lfig+1.2, \Lfig+1.2);

    \draw [line width=0.8pt, black] ($(Pole)+(\bigR, 0)$) ++(180-22:\bigR) arc (180-22:180+22:\bigR) node[near start, right]{$\unitCircle$};

    \draw [line width=1.1pt, color=ForestGreen, arrow inside=0.5] (Pole) circle (\eps);
    \fill [black] (Pole) circle (2pt);

    

    \coordinate (TopPoint) at ($(Pole) + (120:\pathR)$);
    \coordinate (BotPointBlue) at ($(Pole) + (230:\pathR)$);
    \coordinate (BotPointRed) at ($(Pole) + (240:\pathR)$);
    
    \coordinate (BlueTop) at ($(Pole) + (-118:\pathR+\armL)$);
    \coordinate (RedTop) at ($(TopPoint) + (210:\off)$);
    
    \coordinate (BlueBot) at ($(BotPointBlue) + (30:\off)$);
    \coordinate (RedBot) at ($(BotPointRed) + (210:\off)$);

    \draw [line width=1.3pt, color=NavyBlue, arrow inside=0.2, arrow inside=0.8, shift={(30:\off)}]
    ($(Pole) + (119.5:\pathR+\armL)$) -- ($(Pole) + (118:\pathR)$)
   arc[start angle=118, end angle=-118, radius=\pathR]
    -- ($(Pole)+(240.5:\pathR+\armL)$);

\draw[line width=1.3pt, color=SpringGreen, arrow inside=0.2, arrow inside=0.8, shift={(210:\off)}]
    ($(Pole)$) ++(122:\pathRR)
    arc[start angle=122, end angle=238, radius=\pathRR];

    \draw [line width=1.3pt, color=BrickRed, arrow inside=0.2, arrow inside=0.8, shift={(210:\off)}]
    ($(Pole) + (120.5:\pathR+\armL)$) -- ($(Pole) + (122:\pathR)$)
    arc[start angle=122, end angle=238, radius=\pathR]
    -- ($(Pole)+(239.5:\pathR+\armL)$);

    \draw [dashed, gray!80, thin] ($(Pole)+(0,-3.5)$) -- ($(Pole)+(0,3.5)$);
    \draw [line width=0.6pt, ->, >=stealth] (Pole) -- node[pos=0, left] {$-1$} node[midway, below] {$\varepsilon$} ($(Pole)+(-35:\eps)$);

        \node[color=ForestGreen] at ($(Pole)+(50:\eps+0.3)$) {$\gamma_\varepsilon$};
        \node[color=SpringGreen, left] at ($(Pole)+(150:\pathRR+0.3)$) {$\gamma_{\varepsilon}^{2/3\pi}$};
        \node[color=BrickRed, left] at ($(Pole)+(180:\pathR+0.3)$) {$C$};
        \node[color=NavyBlue, right] at ($(Pole)+(10:\pathR)$) {$\tilde{C}$};
        
        \draw[gray, thin] ($(Pole)+(0, 2.0)$) arc (90:118:2.0);
        \node[gray] at ($(Pole)+(105:2.2)$) {$\frac{\pi}{6}$};

    \end{tikzpicture}
        \begin{tikzpicture}[scale=0.75,
        line cap=round, line join=round,
        arrow inside blue/.style={
            postaction={decorate,decoration={markings,
                        mark=at position 0.45 with \arrow{Stealth[length=4pt, width=6pt, inset=1pt]},
                        mark=at position 0.85 with \arrow{Stealth[length=4pt, width=6pt, inset=1pt]}
                        },
              }
        },
        arrow inside red/.style={
            postaction={decorate,decoration={markings,
                        mark=at position 0.25 with \arrow{Stealth[length=4pt, width=6pt, inset=1pt]},
                        mark=at position 0.65 with \arrow{Stealth[length=4pt, width=6pt, inset=1pt]}
                        },
              }
        },
        arrow inside rev/.style={
            postaction={decorate,decoration={
                markings,
                mark=at position #1 with \arrow{Stealth[length=4pt, width=6pt, inset=1pt,  reversed]}}
            }
            }
    ]

    \def\xr{3} \def\yr{3}
    
    \draw [line width=0.8pt, color=black] (0,0) circle (3cm);
    \node[right, inner sep=2pt] at (10:2.3cm) {$\unitCircle$};
    
    \draw [shift={(0,2)}, line width=1.2pt, color=BrickRed, arrow inside red] 
        plot[domain=0.05:pi-0.05, variable=\t]({2.1*cos(\t r)},{2.1*sin(\t r)});
    
    \draw [shift={(-2,0)}, line width=1.2pt, color=BrickRed, arrow inside red] 
        plot[domain=0.5*pi+0.05:1.5*pi-0.05, variable=\t]({2.1*cos(\t r)},{2.1*sin(\t r)});
    
    \draw [shift={(2,0)}, line width=1.2pt, color=BrickRed, arrow inside red] 
        plot[domain=-0.5*pi+0.05:0.5*pi-0.05, variable=\t]({2.1*cos(\t r)},{2.1*sin(\t r)});
    
    \draw [shift={(0,-2)}, line width=1.2pt, color=BrickRed, arrow inside red] 
        plot[domain=pi+0.05:2*pi-0.05, variable=\t]({2.1*cos(\t r)},{2.1*sin(\t r)});

    \draw [shift={(0,1.5)}, line width=1.2pt, color=NavyBlue, arrow inside blue] 
        plot[domain=0.3:pi-0.3, variable=\t]({2.15*cos(\t r)},{2.15*sin(\t r)});
    
    \draw [shift={(0,-1.5)}, line width=1.2pt, color=NavyBlue, arrow inside blue] 
        plot[domain=pi+0.3:2*pi-0.3, variable=\t]({2.15*cos(\t r)},{2.15*sin(\t r)});
    
    \draw [shift={(1.5,0)}, line width=1.2pt, color=NavyBlue, arrow inside blue] 
        plot[domain=-0.5*pi+0.3:0.5*pi-0.3, variable=\t]({2.15*cos(\t r)},{2.15*sin(\t r)});

    \draw [line width=1.2pt, color=ForestGreen, arrow inside red=0.52] (-3,0) circle (0.5cm);

    \draw [shift={(-3,0)}, line width=1.2pt, color=NavyBlue, arrow inside rev=0.45] 
        plot[domain=-pi+0.15:pi-0.15, variable=\t]({0.6*cos(\t r)},{0.6*sin(\t r)});

    \draw [shift={(-1.8,0)}, line width=1.2pt, color=NavyBlue] 
        plot[domain=0.5*pi+0.15:pi-0.05, variable=\t]({2.1*cos(\t r)},{2.1*sin(\t r)});
        
    \draw [line width=1.2pt, color=NavyBlue] (-3.9,0.11) -- (-3.6,0.1);

    \draw [shift={(-1.8,0)}, line width=1.2pt, color=NavyBlue] 
        plot[domain=pi+0.05:1.5*pi-0.15, variable=\t]({2.1*cos(\t r)},{2.1*sin(\t r)});\
        
    \draw [line width=1.2pt, color=NavyBlue] (-3.9,-0.11) -- (-3.6,-0.1);

    \draw [line width=0.6pt, ->, >=stealth] (-3,0) -- node[pos=0, left] {\tiny $-1$} node[midway, below] {$\varepsilon$} (-2.6,0.3);

    \node (sp1) at ({2.11},{2.11}) {};
    \node (sp2) at ({-2.11},{-2.11}) {};
    \node (sp3) at ({2.11},{-2.11}) {};
    \node (sp4) at ({-2.11},{2.11}) {};
    
    \draw (sp1) node[] {$\bullet$} node[below left] {$e^{i\vartheta_{\textnormal{SP}}}$};
    \draw (sp2) node[] {$\bullet$} node[above right] {$-e^{i\vartheta_{\textnormal{SP}}}$};
    \draw (sp3) node[] {$\bullet$} node[above left] {$e^{-i\vartheta_{\textnormal{SP}}}$};
    \draw (sp4) node[] {$\bullet$} node[below right] {$-e^{-i\vartheta_{\textnormal{SP}}}$};
    
        \draw [fill=black] (-2.986,0) circle (1.5pt);
        \node[color=BrickRed] at (1.66,3.81) {$C$};
        \node[color=NavyBlue] at (3.2,0.63) {$\tilde{C}$};
        \node[color=ForestGreen] at (-2.3,-0.8) {$\gamma_\varepsilon$};

    \end{tikzpicture}
\end{center}
\caption{\label{fig:paths}Left: paths used in \eqref{eq:brokenIntegral}. Right: paths used in the proofs of \Cref{prop:transitionZoneStable}, \ref{prop:nearWallUnstable}, and \ref{prop:transitionUnstable}.}
\end{figure}

Comparing  \Cref{prop:FrontUnstable} to 
Results similar to \Cref{prop:FrontUnstable} hold with stable boundary conditions and constant boundary datum \cite{chin1975dispersion} or without boundary and with initial datum being a step function \cite{chin1978dispersion, bouche2003comparison}: they also feature the primitive of the Airy function.
We thus understand that the long-time behavior of our unstable boundary conditions with Dirac delta-datum is similar to the one of stable boundary conditions endowed with a ``resonant'' boundary datum $\propto (-1)^{\indexTime}$. 
Indeed, these are just two different ways of generating a simple pole in the function $g(\timeShiftOperator)$ at $\timeShiftOperator = -1$.

To easily illustrate the origin of \Cref{prop:FrontUnstable}, assume as in the stable case that $\courantNumber = -\frac{p}{q}\in\mathbb{Q}$ where $p, q \in\naturals^{*}$ and coprime.
We thus pose $\indexTime=q\tilde{\indexTime}$ with $\tilde{\indexTime}\in\naturals$ and $\indexSpace = p\tilde{\indexTime}$.
We select a deformation following the directions of steepest descent in \eqref{eq:anglesSteepestDescentZM1}, see \Cref{fig:paths} on the left.
We obtain 
\begin{equation}\label{eq:brokenIntegral}
    \oint_C g(\timeShiftOperator)e^{q\tilde{\indexTime}f(\timeShiftOperator; p/q)}\differential{\timeShiftOperator}  
    = \int_{\gamma_{\varepsilon}^{2/3\pi}}g(\timeShiftOperator)e^{q\tilde{\indexTime}f(\timeShiftOperator; p/q)}\differential{\timeShiftOperator} 
    + \oint_{C\smallsetminus \gamma_{\varepsilon}^{2/3\pi}} g(\timeShiftOperator)e^{q\tilde{\indexTime}f(\timeShiftOperator; p/q)}\differential{\timeShiftOperator}.
\end{equation}
The path of the second integral can be changed without changing the integral's value, as the singularity is not enclosed:
\begin{equation*}
    \oint_{C\smallsetminus \gamma_{\varepsilon}^{2/3\pi}} g(\timeShiftOperator)e^{q\tilde{\indexTime}f(\timeShiftOperator; p/q)}\differential{\timeShiftOperator} = 
    \oint_{\tilde{C}} g(\timeShiftOperator)e^{q\tilde{\indexTime}f(\timeShiftOperator; p/q)}\differential{\timeShiftOperator}.
\end{equation*}
Now, in the vicinity of $\timeShiftOperator = -1$, we write $g(\timeShiftOperator) = \textnormal{Res}_{-1}[g]\times (\timeShiftOperator + 1)^{-1}+g_{\textnormal{reg}}(\timeShiftOperator)$, where $g_{\textnormal{reg}}$ is regular. 
This entails that 
\begin{equation*}
    \oint_{C\smallsetminus \gamma_{\varepsilon}^{2/3\pi}} g(\timeShiftOperator)e^{q\tilde{\indexTime}f(\timeShiftOperator; p/q)}\differential{\timeShiftOperator} = 
    \oint_{\tilde{C}} g(\timeShiftOperator)e^{q\tilde{\indexTime}f(\timeShiftOperator; p/q)}\differential{\timeShiftOperator}
    =
    \oint_{\tilde{C}} g_{\textnormal{reg}}(\timeShiftOperator)e^{q\tilde{\indexTime}f(\timeShiftOperator; p/q)}\differential{\timeShiftOperator}
    =\bigO{\tilde{\indexTime}^{-1/3}},
\end{equation*}
where the last equality comes from the same arguments at the end of \Cref{sec:stable}: the path passes arbitrarily close to the saddle points $\timeShiftOperator = -1$ without enclosing it as a pole.
Regarding the contribution of $\gamma_{\varepsilon}^{2/3\pi}$, we treat $\timeShiftOperator = -1$ as a (partially) path-enclosed pole.
The fact that the pole is simple is particularly useful, as we can use \cite[Lemma 34.1]{agarwal2011introduction} which says that, since we make a turn in the counterclockwise sense $2\pi/3$ (a third of a tour) around the singularity, we have  
\begin{align*}
    \int_{\gamma_{\varepsilon}^{2/3\pi}}
    g(\timeShiftOperator)e^{q\tilde{\indexTime}f(\timeShiftOperator; p/q)}\differential{\timeShiftOperator} 
    &= 
    \frac{2\pi i}{3}\textnormal{Res}_{-1}[ g(\timeShiftOperator)e^{q\tilde{\indexTime}f(\timeShiftOperator; p/q)}]
    =
    \frac{2\pi i}{3} e^{q\tilde{\indexTime}f(-1; p/q)} \textnormal{Res}_{-1}[ g(\timeShiftOperator)]\\
    &=
    -\frac{(-1)^{q\tilde{\indexTime}}}{3}
        \Bigl (2
            + \sum_{k\geq 0}
            \coefficientEventualBoundaryScheme_k
            + \frac{1}{\courantNumber} \sum_{k\geq 0}
            k (\coefficientEventualBoundaryScheme_k-\coefficientEventualBoundarySchemeOld_k)\Bigr )^{-1},
\end{align*}
which is the expected approximation of $\globalTruncationError_{p\tilde{\indexTime}}^{q\tilde{\indexTime}}$ up to terms $\bigO{\tilde{\indexTime}^{-1/3}}$.
This computation reveals that the leading-order dynamics in this region are governed by the residue at the pole.
However, this contribution is \strong{shaded} by the fact that the deformed contour---required to step into the saddle point along steepest descent directions---only encompasses a \strong{portion of tour}.
The factor $\frac{1}{3}$ in front of the residue rings---rightly so---a bell concerning the primitive of the Airy function, as $\int_0^{+\infty}\textnormal{Ai}(\spaceVariable)\differential{\spaceVariable} = \tfrac{1}{3}$.
\begin{figure}
    \centering
    \includegraphics[width=1\textwidth]{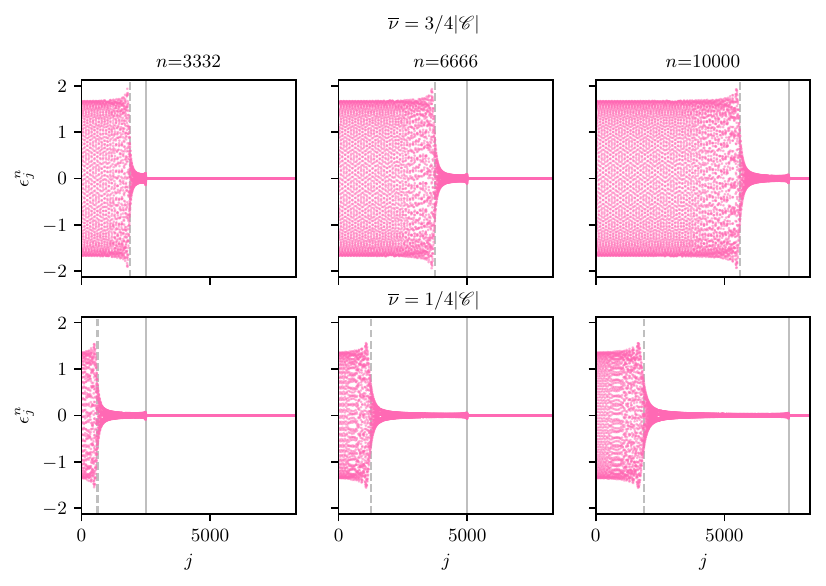}
    \caption{Illustration of $\globalTruncationError_{\indexSpace}^{\indexTime}$ obtained with boundary conditions generating instabilities propagating at an arbitrary group velocity $\overline{\nu}\in(0, |\courantNumber|]$, see \Cref{rem:anyGroupVelocity}.
    The full grey line corresponds to $\indexTime|\courantNumber|$, whereas the dashed one to $\indexTime\overline{\nu}$.}
    \label{fig:generalGroupVelocity}
\end{figure}

\begin{remark}[Instabilities propagating at any group velocity can be synthesized]\label{rem:anyGroupVelocity}
    The assumption on the pole of the boundary function $g(\timeShiftOperator)$ is considered for the sake of illustration.
    For instance, with the leap-frog scheme, every group velocity between $0$ and $|\courantNumber|$ is available due to the presence of associated saddle points.
    It is therefore easy to manufacture unstable boundary conditions where a pole of $g(\timeShiftOperator)$ coincides with a saddle point of $f(\timeShiftOperator; \overline{\nu})$ with associated group velocity $\overline{\nu}\in (0, |\courantNumber|]$.
    One can check that
    \begin{equation*}
        \coefficientEventualBoundaryScheme_0
        =
        2\sqrt{\frac{1-\courantNumber^2}{1-\overline{\nu}^2}}
        \qquad \text{and}\qquad 
        \coefficientEventualBoundarySchemeOld_0 = -1
    \end{equation*}
    finely does the job.
    An example with $\courantNumber=-\tfrac{3}{4}$ and $\overline{\nu} = \tfrac{1}{4}|\courantNumber|$ and $\overline{\nu} = \tfrac{3}{4}|\courantNumber|$ is given in \Cref{fig:generalGroupVelocity}.
\end{remark}

\subsection{A dissipative bulk scheme}
By using the $\timeShiftOperator$-transform to reinterpret \eqref{eq:schemeErrorEventualDissipative} as a recurrence in space, we obtain 
\begin{multline}\label{eq:ztransformedProblem_dissipative}
    \timeShiftOperator^2 \zTransformed{\globalTruncationError}_{0}(\timeShiftOperator) - \timeShiftOperator \sum_{k\in\naturals}\coefficientEventualBoundaryScheme_k \zTransformed{\globalTruncationError}_{k}(\timeShiftOperator) - \sum_{k\in\naturals}\coefficientEventualBoundarySchemeOld_k \zTransformed{\globalTruncationError}_{k}(\timeShiftOperator)= \timeShiftOperator
    \qquad \text{and} \\ 
    (\timeShiftOperator^2 +1-\omega)\zTransformed{\globalTruncationError}_{\indexSpace}(\timeShiftOperator) + \tfrac{1}{2}\left((\omega-2)-\omega \courantNumber \right)\timeShiftOperator\zTransformed{\globalTruncationError}_{\indexSpace-1}(\timeShiftOperator) + \tfrac{1}{2}\left((\omega-2)+\omega \courantNumber \right) \timeShiftOperator\zTransformed{\globalTruncationError}_{\indexSpace+1}(\timeShiftOperator) = 0, \quad \indexSpace\geq 1.
\end{multline}
In this case, it is sufficient to adapt the expression of $\stableRoot(\timeShiftOperator)$, keeping the expressions of $g(\timeShiftOperator)$ and $f(\timeShiftOperator; \nu)$ given in \eqref{eq:definitionFandG} unchanged.
We have
\begin{lemma}[Branch points of $\zTransformed{\globalTruncationError}_{\indexSpace}(\timeShiftOperator)$]
    Let $\relaxationParameter\in (0, 2)$ and $|\courantNumber|< 1$, then the function $\complex\ni \timeShiftOperator\mapsto\zTransformed{\globalTruncationError}_{\indexSpace}(\timeShiftOperator)$ for \eqref{eq:zeroInitialErrorDissipative}--\eqref{eq:schemeErrorInitialDissipative}--\eqref{eq:schemeErrorEventualBoundaryDissipative}--\eqref{eq:schemeErrorEventualDissipative} has, except when $\relaxationParameter = \frac{2}{1+\courantNumber}$, branch points corresponding to the zeros of $\timeShiftOperator\mapsto\timeShiftOperator^4 + ((\courantNumber^2-1)\relaxationParameter^2 + 2(\relaxationParameter-1))\timeShiftOperator^2 + (\relaxationParameter-1)^2$, which lay in $\unitDisk$.
\end{lemma}
The previous lemma states that we can forget about branch point singularities, for they yield exponentially damped behaviors.
\begin{lemma}[Saddle points]\label{lemma:saddlePointSmooth}
    Assume $\relaxationParameter\in(0, 2)$ and $0<\courantNumber\leq 1$.
    Then, the saddle points of the function $\timeShiftOperator\mapsto f(\timeShiftOperator; \nu)$ for \eqref{eq:zeroInitialErrorDissipative}--\eqref{eq:schemeErrorInitialDissipative}--\eqref{eq:schemeErrorEventualBoundaryDissipative}--\eqref{eq:schemeErrorEventualDissipative} are as follows.
\begin{itemize}
    \item For $\nu\in (0, \courantNumber)$, the saddle points are in $\unitDisk$.
    \item For $\nu = \courantNumber$, the saddle points are $\pm (\relaxationParameter-1)\in\unitDisk$, but more importantly
    \begin{equation*}
        \pm 1,
        \qquad \text{with}\qquad 
        \stableRoot(\pm 1) = \pm 1
        \quad \text{and}\quad 
        f''(\pm 1; \courantNumber) 
        =
        \frac{2}{\courantNumber^2}\Bigl (\frac{1}{\relaxationParameter} - \frac{1}{2} \Bigr ) (1-\courantNumber^2).
    \end{equation*}
    This saddle points are associated with a \strong{Gaussian peak}.
    \item For $\nu\in (\courantNumber, 1)$, saddle points are both in $\unitDisk$ and in $\neighborhoodInfinity$.
\end{itemize}
\end{lemma}
As for \Cref{lemma:saddlePoints}, we connect \Cref{lemma:saddlePointSmooth} to \cite{thomee1965stability} in \Cref{rem:linkwithThomee} and \ref{rem:sameSP}.

Thanks to the previous discussion, the solution shall only be of some significance around $\indexSpace\sim \courantNumber\indexTime$, where resemblance to a Gaussian is established. 
\begin{proposition}[Gaussian peak]\label{prop:GaussianPeak}
Consider $\globalTruncationError_{\indexSpace}^{\indexTime}$ solution of \eqref{eq:zeroInitialErrorDissipative}--\eqref{eq:schemeErrorInitialDissipative}--\eqref{eq:schemeErrorEventualBoundaryDissipative}--\eqref{eq:schemeErrorEventualDissipative} and let $\relaxationParameter\in (0, 2)$ with $0<\courantNumber<1$, and \Cref{ass:stableBCDissipative} hold.
    Let $\indexTime \gg 1$ and $\indexSpace\in\naturals$ such that $\indexSpace-\courantNumber\indexTime = \bigO{1}$.
    Then, a good approximation of $\globalTruncationError_{\indexSpace}^{\indexTime}$ is given by 
     \begin{multline}\label{eq:gaussian}
        \globalTruncationError_{\indexSpace}^{\indexTime} 
    \sim 
    \courantNumber
    \Biggl ( \frac{1}{1-\sum_{k\geq 0}(\coefficientEventualBoundaryScheme_k + \coefficientEventualBoundarySchemeOld_k)}
    - \frac{(-1)^{\indexTime+\indexSpace}}{1+\sum_{k\geq 0}(-1)^k (\coefficientEventualBoundaryScheme_k - \coefficientEventualBoundarySchemeOld_k)}\Biggr )\\
    \times
    \underbrace{\frac{1}{\sqrt{2\pi}}
    \frac{1}{\sqrt{2\indexTime ( \frac{1}{\relaxationParameter}-\frac{1}{2} )(1-\courantNumber^2)}} 
    \textnormal{exp}\Bigl ({-\frac{1}{2}\frac{(\indexSpace-\courantNumber\indexTime)^2}{2\indexTime(\frac{1}{\relaxationParameter}-\frac{1}{2})(1-\courantNumber^2)}} \Bigr )}_{L^1-\text{normalized Gaussian}}.
    \end{multline}
\end{proposition}
Although this result is novel in presence of boundary conditions and in the context of multi-step schemes, it is expected in view of the results on $\relatives$ by \cite[Theorem 1.2]{randles2015convolution} and \cite[Theorem 1.6]{coulombel2022generalized} for one-step schemes.

\subsection{Proofs}

We now provide detailed proofs of the claims of the  previous section.

\subsubsection{Brief reminders on asymptotic analysis of integrals: non-degenerate saddle points}

Along the $\timeShiftOperator$-transform, a second tool that we extensively exploit is a set of techniques of asymptotic analysis of integrals based on saddle point--steepest descent techniques.
We state all results loosely enough to concentrate on their ``physical'' meaning rather than focusing on precisely and rigorously describe reminder terms.

Let us first revise an important result when schemes are considered without boundary condition, so that Fourier analysis is available, \confer{} \Cref{sec:GreenFunction}.
Let $\frequency_{\textnormal{SP}}\in [a, b]$ be the unique point such that $f'(\frequency_{\textnormal{SP}}) = 0$ and $f''(\frequency_{\textnormal{SP}})\neq 0$, then, see \cite[Equation (6.5.12)]{bender2013advanced}, we have the \strong{stationary phase approximation} for $\indexTime\gg 1$
\begin{equation}\label{eq:stationaryPhaseApproximation}
    \int_a^b g(\frequency)e^{i\indexTime f(\frequency)}\differential{\frequency} 
    \sim \sqrt{\frac{2\pi}{\indexTime |f''(\frequency_{\textnormal{SP}})|}} g(\frequency_{\textnormal{SP}})
    e^{i\indexTime f(\frequency_{\textnormal{SP}}) + i\frac{\pi}{4}\textnormal{sgn}(f''(\frequency_{\textnormal{SP}}))}.
\end{equation}
This expression makes sense as long as $g(\frequency_{\textnormal{SP}})$ is well-defined.
When several critical points are present, contributions from each of them sum.
For complex integrals, a very similar result holds, see \cite[Equation (7.113)]{arfken2011mathematical}.
Let $\saddlePoint\in\complex$ the unique point such that $f'(\saddlePoint) = 0$ and $f''(\saddlePoint)\neq 0$, then we have the \strong{steepest descent approximation} for $\indexTime\gg 1$
\begin{equation}\label{eq:saddlePointComplex}
    \int_C g(\timeShiftOperator)e^{\indexTime f(\timeShiftOperator)}\differential{\timeShiftOperator} 
    \sim \sqrt{\frac{2\pi}{\indexTime |f''(\saddlePoint)|}} g(\saddlePoint)
    e^{\indexTime f(\saddlePoint) + i(\frac{\pi}{2}-\frac{1}{2}\textnormal{Arg}(f''(\saddlePoint)))}.
\end{equation}
Note that the term $\frac{\pi}{2}-\frac{1}{2}\textnormal{Arg}(f''(\saddlePoint))$ is nothing but the direction of steepest descent, say $\phi$, obtained by setting $\textnormal{Arg}(f''(\saddlePoint)) + 2\phi = \pi$.
Especially when the original contour $C$ is a closed and several saddle points need to be straddled, geometrical constraints while deform the contour may lead to contributions featuring $-\frac{\pi}{2}-\frac{1}{2}\textnormal{Arg}(f''(\saddlePoint))$ instead, obtained by having $\textnormal{Arg}(f''(\saddlePoint)) + 2\phi = -\pi$.
An example of this is presented in the proof of \Cref{prop:transitionZoneStable}, detailed in \Cref{sec:proofLeapFrog} below.
Notice that \eqref{eq:saddlePointComplex} makes sense as long as $g$ is regular at the saddle point: in the sequel, we deal with issues coming from singular $g$'s at saddle points of $f$, and the degenerate case where $f''(\saddlePoint)= 0$, whose occurrence is remarkably and concisely discussed in  \cite[Appendix III]{born1975principles}.

\subsubsection{Leap-frog bulk scheme}\label{sec:proofLeapFrog}

\begin{proof}[Proof of \Cref{prop:nearWallStable}]
 Simple manipulations give 
    \begin{multline*}
    \stableRoot(\timeShiftOperator^{-1}) = \frac{\timeShiftOperator^{-1}}{2\courantNumber}(\timeShiftOperator^2-1 + \sqrt{\timeShiftOperator^4 + 2(2\courantNumber^2-1)\timeShiftOperator^{2} + 1})\qquad \text{and}\\
    \zTransformed{\globalTruncationError}_{\indexSpace}(\timeShiftOperator^{-1}) = \displaystyle \frac{\timeShiftOperator \, \stableRoot(\timeShiftOperator^{-1})^{\indexSpace}}{1 - \timeShiftOperator \displaystyle \sum_{k\geq 0}\coefficientEventualBoundaryScheme_k \stableRoot(\timeShiftOperator^{-1})^{k} - \timeShiftOperator^2 \displaystyle \sum_{k\geq 0}\coefficientEventualBoundarySchemeOld_k \stableRoot(\timeShiftOperator^{-1})^{k}}.
\end{multline*}
By the assumption on stability, the singularities of $\zTransformed{\globalTruncationError}_{\indexSpace}(\timeShiftOperator^{-1})$ closest to the origin are the four branch points in \Cref{lemma:BranchPoints}.
Let us consider the neighborhood of $e^{i\branchPointAngle}$ in detail.
Computations provided in \cite{bellotti2025perfectly} yield
    \begin{equation*}
        \sqrt{\timeShiftOperator^{4} + 2(2\courantNumber^2-1)\timeShiftOperator^{2} + 1} =2^{3/2} |\courantNumber|^{1/2}(1-\courantNumber^2)^{1/4} e^{i(\branchPointAngle-\frac{\pi}{4})}({1-\timeShiftOperator/e^{i\branchPointAngle}})^{1/2} + \bigO{(1-\timeShiftOperator/e^{i\branchPointAngle})^{3/2}},
    \end{equation*}
    and thus, using the fact that $\sin(\branchPointAngle) = |\courantNumber|$, we obtain
    \begin{align*}
        \stableRoot(\timeShiftOperator^{-1}) = 
        e^{-i\frac{\pi}{2}}
        - 2^{1/2} |\courantNumber|^{-1/2}(1-\courantNumber^2)^{1/4} e^{-i\frac{\pi}{4}}({1-\timeShiftOperator/e^{i\branchPointAngle}})^{1/2}
        +\bigO{1-\timeShiftOperator/e^{i\branchPointAngle}}.
    \end{align*}
    This entails
    \begin{align*}
        \stableRoot(\timeShiftOperator^{-1})^{\indexSpace} = 
        e^{-i\indexSpace\frac{\pi}{2}}
        - 2^{1/2} |\courantNumber|^{-1/2}(1-\courantNumber^2)^{1/4} e^{-i((\indexSpace-1)\frac{\pi}{2}+\frac{\pi}{4})} \indexSpace({1-\timeShiftOperator/e^{i\branchPointAngle}})^{1/2}
        +\bigO{1-\timeShiftOperator/e^{i\branchPointAngle}}.
    \end{align*}
   On the other hand 
   \begin{multline*}
       1 - \timeShiftOperator \displaystyle \sum_{k\geq 0}\coefficientEventualBoundaryScheme_k \stableRoot(\timeShiftOperator^{-1})^{k} - \timeShiftOperator^2 \displaystyle \sum_{k\geq 0}\coefficientEventualBoundarySchemeOld_k \stableRoot(\timeShiftOperator^{-1})^{k} 
       = 
       1 - e^{i\branchPointAngle} \displaystyle \sum_{k\geq 0}\coefficientEventualBoundaryScheme_k e^{-ik\frac{\pi}{2}} - e^{2 i\branchPointAngle} \displaystyle \sum_{k\geq 0}\coefficientEventualBoundarySchemeOld_k e^{-ik\frac{\pi}{2}}\\
       + 2^{1/2} |\courantNumber|^{-1/2}(1-\courantNumber^2)^{1/4} e^{-i\frac{\pi}{4}}
       \Bigl ( 
       e^{i\branchPointAngle} \displaystyle \sum_{k\geq 1} k \coefficientEventualBoundaryScheme_k e^{i(1-k)\frac{\pi}{2}} +e^{2 i\branchPointAngle} \displaystyle \sum_{k\geq 1} k \coefficientEventualBoundarySchemeOld_k e^{i (1-k)\frac{\pi}{2}}
       \Bigr )
       ({1-\timeShiftOperator/e^{i\branchPointAngle}})^{1/2}
        \\
        +\bigO{1-\timeShiftOperator/e^{i\branchPointAngle}},
   \end{multline*}
   which by the Neumann series entails 
   \begin{multline*}
       \Bigl ( 
       1 - \timeShiftOperator \displaystyle \sum_{k\geq 0}\coefficientEventualBoundaryScheme_k \stableRoot(\timeShiftOperator^{-1})^{k} - \timeShiftOperator^2 \displaystyle \sum_{k\geq 0}\coefficientEventualBoundarySchemeOld_k \stableRoot(\timeShiftOperator^{-1})^{k}
       \Bigr )^{-1}
       =
       \frac{1}{1 - e^{i\branchPointAngle} \displaystyle \sum_{k\geq 0}\coefficientEventualBoundaryScheme_k e^{-ik\frac{\pi}{2}} - e^{2 i\branchPointAngle} \displaystyle \sum_{k\geq 0}\coefficientEventualBoundarySchemeOld_k e^{-ik\frac{\pi}{2}}}\\
       \times 
       \Biggl ( 
       1 - 
       \frac{2^{1/2} |\courantNumber|^{-1/2}(1-\courantNumber^2)^{1/4} e^{-i\frac{\pi}{4}}
       \Bigl ( 
       e^{i\branchPointAngle} \displaystyle \sum_{k\geq 1} k \coefficientEventualBoundaryScheme_k e^{i(1-k)\frac{\pi}{2}} +e^{2 i\branchPointAngle} \displaystyle \sum_{k\geq 1} k \coefficientEventualBoundarySchemeOld_k e^{i (1-k)\frac{\pi}{2}}
       \Bigr )}{1 - e^{i\branchPointAngle} \displaystyle \sum_{k\geq 0}\coefficientEventualBoundaryScheme_k e^{-ik\frac{\pi}{2}} - e^{2 i\branchPointAngle} \displaystyle \sum_{k\geq 0}\coefficientEventualBoundarySchemeOld_k e^{-ik\frac{\pi}{2}}}
       ({1-\timeShiftOperator/e^{i\branchPointAngle}})^{1/2}
       \Biggr )
       \\
       +\bigO{1-\timeShiftOperator/e^{i\branchPointAngle}}.
   \end{multline*}
   Blending all together results in 
   \begin{multline*}
       \zTransformed{\globalTruncationError}_{\indexSpace}(\timeShiftOperator^{-1})
       =
       \{\text{zero-order terms}\}
       -
       \frac{2^{1/2} |\courantNumber|^{-1/2}(1-\courantNumber^2)^{1/4}e^{i(\branchPointAngle + (1-\indexSpace)\frac{\pi}{2} -\frac{\pi}{4})}}{1 - e^{i\branchPointAngle} \displaystyle \sum_{k\geq 0}\coefficientEventualBoundaryScheme_k e^{-ik\frac{\pi}{2}} - e^{2 i\branchPointAngle} \displaystyle \sum_{k\geq 0}\coefficientEventualBoundarySchemeOld_k e^{-ik\frac{\pi}{2}}}\\
       \times
       \Biggl ( 
       \frac{
       e^{i\branchPointAngle} \displaystyle \sum_{k\geq 1} k \coefficientEventualBoundaryScheme_k e^{-ik \frac{\pi}{2}} +e^{2 i\branchPointAngle} \displaystyle \sum_{k\geq 1} k \coefficientEventualBoundarySchemeOld_k e^{-ik\frac{\pi}{2}}}{1 - e^{i\branchPointAngle} \displaystyle \sum_{k\geq 0}\coefficientEventualBoundaryScheme_k e^{-ik\frac{\pi}{2}} - e^{2 i\branchPointAngle} \displaystyle \sum_{k\geq 0}\coefficientEventualBoundarySchemeOld_k e^{-ik\frac{\pi}{2}}}
       +  \indexSpace
       \Biggr )
       ({1-\timeShiftOperator/e^{i\branchPointAngle}})^{1/2}
       +\bigO{1-\timeShiftOperator/e^{i\branchPointAngle}}.
   \end{multline*}
   Around $e^{-i\branchPointAngle}$, we obtain the conjugate of the terms explicitly described in the previous expansion.
   Through analogous computations, the expansion around $-e^{-i\branchPointAngle}$ is
      \begin{multline*}
       \zTransformed{\globalTruncationError}_{\indexSpace}(\timeShiftOperator^{-1})
       =
       \{\text{zero-order terms}\}
       -
       \frac{2^{1/2} |\courantNumber|^{-1/2}(1-\courantNumber^2)^{1/4}e^{i(-\branchPointAngle + (1-\indexSpace)\frac{\pi}{2} +\frac{\pi}{4})}}{1 + e^{-i\branchPointAngle} \displaystyle \sum_{k\geq 0}\coefficientEventualBoundaryScheme_k e^{-ik\frac{\pi}{2}} - e^{-2 i\branchPointAngle} \displaystyle \sum_{k\geq 0}\coefficientEventualBoundarySchemeOld_k e^{-ik\frac{\pi}{2}}}\\
       \times
       \Biggl ( 
       -\frac{
       e^{-i\branchPointAngle} \displaystyle \sum_{k\geq 1} k \coefficientEventualBoundaryScheme_k e^{-ik \frac{\pi}{2}} -e^{-2 i\branchPointAngle} \displaystyle \sum_{k\geq 1} k \coefficientEventualBoundarySchemeOld_k e^{-ik\frac{\pi}{2}}}{1 + e^{-i\branchPointAngle} \displaystyle \sum_{k\geq 0}\coefficientEventualBoundaryScheme_k e^{-ik\frac{\pi}{2}} - e^{-2 i\branchPointAngle} \displaystyle \sum_{k\geq 0}\coefficientEventualBoundarySchemeOld_k e^{-ik\frac{\pi}{2}}}
       +  \indexSpace
       \Biggr )
       ({1+\timeShiftOperator/e^{-i\branchPointAngle}})^{1/2}
       +\bigO{1+\timeShiftOperator/e^{-i\branchPointAngle}},
   \end{multline*}
   along with its complex conjugate around $-e^{i\branchPointAngle}$.
   Using \cite[Figure VI.5 and Theorem VI.5]{flajolet2009analytic}, we arrive at
   \begin{multline*}
       \globalTruncationError_{\indexSpace}^{\indexTime}
       \sim 
       \sqrt{\frac{2}{\pi |\courantNumber|}} (1-\courantNumber^2)^{1/4}
       \Biggl ( 
        \textnormal{Re}\Biggl ( 
        \frac{e^{i((1-\indexTime)\branchPointAngle + (1-\indexSpace)\frac{\pi}{2} - \frac{\pi}{4})}}{1 - e^{i\branchPointAngle}  \sum_{k}\coefficientEventualBoundaryScheme_k e^{-ik\frac{\pi}{2}} - e^{2 i\branchPointAngle}  \sum_{k}\coefficientEventualBoundarySchemeOld_k e^{-ik\frac{\pi}{2}}}\\
        \times
        \Biggl ( 
       \frac{
       e^{i\branchPointAngle}  \sum_{k} k \coefficientEventualBoundaryScheme_k e^{-ik \frac{\pi}{2}} +e^{2 i\branchPointAngle}  \sum_{k} k \coefficientEventualBoundarySchemeOld_k e^{-ik\frac{\pi}{2}}}{1 - e^{i\branchPointAngle}  \sum_{k}\coefficientEventualBoundaryScheme_k e^{-ik\frac{\pi}{2}} - e^{2 i\branchPointAngle}  \sum_{k}\coefficientEventualBoundarySchemeOld_k e^{-ik\frac{\pi}{2}}}
       +  \indexSpace
       \Biggr )
        \Biggr )\\
    +(-1)^{\indexTime}\textnormal{Re}
    \Biggl ( 
    \frac{e^{i((\indexTime-1)\branchPointAngle + (1-\indexSpace)\frac{\pi}{2}+\frac{\pi}{4})}}{1 + e^{-i\branchPointAngle}  \sum_{k}\coefficientEventualBoundaryScheme_k e^{-ik\frac{\pi}{2}} - e^{-2 i\branchPointAngle}  \sum_{k}\coefficientEventualBoundarySchemeOld_k e^{-ik\frac{\pi}{2}}}\\
    \times
    \Biggl ( 
       -\frac{
       e^{-i\branchPointAngle}  \sum_{k} k \coefficientEventualBoundaryScheme_k e^{-ik \frac{\pi}{2}} -e^{-2 i\branchPointAngle}  \sum_{k} k \coefficientEventualBoundarySchemeOld_k e^{-ik\frac{\pi}{2}}}{1 + e^{-i\branchPointAngle}  \sum_{k}\coefficientEventualBoundaryScheme_k e^{-ik\frac{\pi}{2}} - e^{-2 i\branchPointAngle}  \sum_{k}\coefficientEventualBoundarySchemeOld_k e^{-ik\frac{\pi}{2}}}
       +  \indexSpace
       \Biggr )
       \Biggr )
       \Biggr )\indexTime^{-3/2}
       +\bigO{\indexTime^{-5/2}}.
   \end{multline*}
    The general part of the claim follows.
    
   In the upwind case, see \Cref{ex:upwind}, we have 
    \begin{multline*}
       1 - e^{i\branchPointAngle}  \sum_{k}\coefficientEventualBoundaryScheme_k e^{-ik\frac{\pi}{2}} - e^{2 i\branchPointAngle}  \sum_{k}\coefficientEventualBoundarySchemeOld_k e^{-ik\frac{\pi}{2}} = 
       1-(1+\courantNumber)e^{i\branchPointAngle} 
       +\courantNumber e^{i(\branchPointAngle-\frac{\pi}{2})}\\
       =
       1-(1+\courantNumber)\cos(\branchPointAngle) +\courantNumber\sin(\branchPointAngle) - i ((1+\courantNumber)\sin(\branchPointAngle) + \courantNumber\cos(\branchPointAngle))\\
       =(\sqrt{1-\courantNumber^2}-(1+\courantNumber)) (\sqrt{1-\courantNumber^2} - i\courantNumber)
       =(\sqrt{1-\courantNumber^2}-(1+\courantNumber))e^{i\branchPointAngle}
   \end{multline*}
   through simple trigonometric identities, and 
   \begin{equation*}
       1 + e^{-i\branchPointAngle}  \sum_{k}\coefficientEventualBoundaryScheme_k e^{-ik\frac{\pi}{2}} - e^{-2 i\branchPointAngle}  \sum_{k}\coefficientEventualBoundarySchemeOld_k e^{-ik\frac{\pi}{2}}
       =(\sqrt{1-\courantNumber^2}+(1+\courantNumber)) (\sqrt{1-\courantNumber^2} + i\courantNumber)
       =(\sqrt{1-\courantNumber^2}+(1+\courantNumber)) e^{-i\branchPointAngle}.
   \end{equation*}
   Moreover 
   \begin{equation*}
       e^{i\branchPointAngle}  \sum_{k} k \coefficientEventualBoundaryScheme_k e^{-ik \frac{\pi}{2}} +e^{2 i\branchPointAngle}  \sum_{k} k \coefficientEventualBoundarySchemeOld_k e^{-ik\frac{\pi}{2}}
       =-\courantNumber e^{i(\branchPointAngle-\frac{\pi}{2})},
   \end{equation*}
   and 
   \begin{equation*}
       e^{-i\branchPointAngle}  \sum_{k} k \coefficientEventualBoundaryScheme_k e^{-ik \frac{\pi}{2}} -e^{-2 i\branchPointAngle}  \sum_{k} k \coefficientEventualBoundarySchemeOld_k e^{-ik\frac{\pi}{2}}
       =\courantNumber e^{i(\frac{\pi}{2}-\branchPointAngle)}.
   \end{equation*}
   Into the general expression for the expansion, this yields the claim after some trigonometry.
\end{proof}

We naturally place this proof of \Cref{prop:momentsInSpaceStable} after the one of \Cref{prop:nearWallStable}, as it is equally based on techniques of analytic combinatorics.
Indeed, it stems from univariate analytic combinatorics, since the fact of taking moments absorbs the second index $\indexSpace$.
\begin{proof}[Proof of \Cref{prop:momentsInSpaceStable}]

Let us introduce the equivalent of the $\timeShiftOperator$-transform in space:
\begin{equation*}
    \doubleZTransformed{\globalTruncationError}(\timeShiftOperator, \fourierShift)
    \definitionEquality \sum_{\indexTime=0}^{+\infty}\sum_{\indexSpace=0}^{+\infty}\timeShiftOperator^{-\indexTime}\fourierShift^{-\indexSpace}\globalTruncationError_{\indexSpace}^{\indexTime} 
    = \sum_{\indexSpace=0}^{+\infty}\fourierShift^{-\indexSpace}\zTransformed{\globalTruncationError}_{\indexSpace} (\timeShiftOperator).
\end{equation*}
Applying to \eqref{eq:ztransformedProblem} gives
\begin{equation*}
    \underbrace{\Bigl [\timeShiftOperator^2-1 - \courantNumber\timeShiftOperator (\fourierShift^{-1}-\fourierShift)\Bigr ]}_{\text{bulk scheme}}\doubleZTransformed{\globalTruncationError}(\timeShiftOperator, \fourierShift) = \underbrace{(\timeShiftOperator^2 - 1 + \courantNumber\timeShiftOperator\fourierShift) \zTransformed{\globalTruncationError}_{0}(\timeShiftOperator) + \courantNumber\timeShiftOperator \zTransformed{\globalTruncationError}_{1}(\timeShiftOperator)}_{\text{time-space boundary schemes}}.
\end{equation*}
As the expressions of $\zTransformed{\globalTruncationError}_0(\timeShiftOperator)$ and  $\zTransformed{\globalTruncationError}_1(\timeShiftOperator)$ are explicit, we obtain 
\begin{equation*}
    \doubleZTransformed{\globalTruncationError}(\timeShiftOperator, \fourierShift) = \frac{\timeShiftOperator({\timeShiftOperator^2 - 1 + \courantNumber\timeShiftOperator(\fourierShift   + \stableRoot(\timeShiftOperator))})}{(\timeShiftOperator^2 - \timeShiftOperator\sum_{k\geq 0}\coefficientEventualBoundaryScheme_k \stableRoot(\timeShiftOperator)^k - \sum_{k\geq 0}\coefficientEventualBoundarySchemeOld_k\stableRoot(\timeShiftOperator)^k)(\timeShiftOperator^2-1 - \courantNumber\timeShiftOperator (\fourierShift^{-1}-\fourierShift))}.
\end{equation*}
To exploit the results from \cite[Chapter III]{flajolet2009analytic}, we rewrite as
\begin{equation}\label{eq:doubleGeneratingFunction}
    \doubleZTransformed{\globalTruncationError}(\timeShiftOperator^{-1}, \fourierShift^{-1}) = \frac{\timeShiftOperator({1 - \timeShiftOperator^2 + \courantNumber\timeShiftOperator(\fourierShift^{-1}   + \stableRoot(\timeShiftOperator^{-1}))})}{(1 - \timeShiftOperator\sum_{k\geq 0}\coefficientEventualBoundaryScheme_k \stableRoot(\timeShiftOperator^{-1})^k - \timeShiftOperator^2\sum_{k\geq 0}\coefficientEventualBoundarySchemeOld_k\stableRoot(\timeShiftOperator^{-1})^k)(1-\timeShiftOperator^2 - \courantNumber\timeShiftOperator (\fourierShift-\fourierShift^{-1}))}.
\end{equation}
For the zero-order moment, we have to consider the function $\timeShiftOperator\mapsto \doubleZTransformed{\globalTruncationError}(\timeShiftOperator^{-1}, \fourierShift^{-1}=1)$.
Its singularities closest to the origin are all on $\unitCircle$ thanks to the stability assumption: the four branch points and a simple pole at $\timeShiftOperator = -1$ (notice that $\timeShiftOperator = 1$ cancels both numerator and denominator).
The branch points give contributions of order $\bigO{\indexTime^{-3/2}}$ to the asymptotics, and we are thus left to analyze the behavior near the pole.
This results in
\begin{equation*}
     \doubleZTransformed{\globalTruncationError}(\timeShiftOperator^{-1}, \fourierShift^{-1} = 1)
     =
     \frac{\courantNumber}{1 + \sum_{k\geq 0}\coefficientEventualBoundaryScheme_k - \sum_{k\geq 0}\coefficientEventualBoundarySchemeOld_k}(\timeShiftOperator+1)^{-1} + \bigO{1},
\end{equation*}
hence the asymptotics
\begin{equation*}
    \sum_{\indexSpace\geq 0}\globalTruncationError_{\indexSpace}^{\indexTime}
    \sim 
    \frac{\courantNumber(-1)^{\indexTime}}{1 + \sum_{k\geq 0}\coefficientEventualBoundaryScheme_k - \sum_{k\geq 0}\coefficientEventualBoundarySchemeOld_k} + \bigO{\indexTime^{-3/2}}.
\end{equation*}
An analogous computation with $\timeShiftOperator\mapsto \doubleZTransformed{\globalTruncationError}(\timeShiftOperator^{-1}, \fourierShift^{-1}=-1)$ gives 
\begin{equation*}
    \sum_{\indexSpace\geq 0}(-1)^{\indexSpace}\globalTruncationError_{\indexSpace}^{\indexTime}
    \sim 
    \frac{\courantNumber}{\sum_{k\geq 0}(-1)^{k}(\coefficientEventualBoundaryScheme_k+\coefficientEventualBoundarySchemeOld_k) - 1} + \bigO{\indexTime^{-3/2}}.
\end{equation*}

Observe that another proof of the asymptotics for $\sum_{\indexSpace\geq 0}\globalTruncationError_{\indexSpace}^{\indexTime}$ can be given.
To this end, we notice by linearity that $\sum_{\indexSpace\in\naturals}\zTransformed{\globalTruncationError}_{\indexSpace}(\timeShiftOperator)$ is the $\timeShiftOperator$-transform of $\sum_{\indexSpace\geq 0}\globalTruncationError_{\indexSpace}^{\indexTime}$. We obtain
\begin{equation*}
    \sum_{\indexSpace\geq 0} \zTransformed{\globalTruncationError}_{\indexSpace}(\timeShiftOperator) = \displaystyle \frac{\timeShiftOperator \, \sum_{\indexSpace\geq 0} \stableRoot(\timeShiftOperator)^{\indexSpace}}{\timeShiftOperator^2  - \timeShiftOperator \displaystyle \sum_{k\in\naturals}\coefficientEventualBoundaryScheme_k \stableRoot(\timeShiftOperator)^{k} - \displaystyle \sum_{k\in\naturals}\coefficientEventualBoundarySchemeOld_k \stableRoot(\timeShiftOperator)^{k}} = 
    \displaystyle \frac{\timeShiftOperator}{\Bigl( \timeShiftOperator^2  - \timeShiftOperator \displaystyle \sum_{k\in\naturals}\coefficientEventualBoundaryScheme_k \stableRoot(\timeShiftOperator)^{k} - \displaystyle \sum_{k\in\naturals}\coefficientEventualBoundarySchemeOld_k \stableRoot(\timeShiftOperator)^{k} \Bigr ) (1-\stableRoot(\timeShiftOperator))} .
\end{equation*}
The leading singularity of the right-hand side, a first-order pole, is at $\timeShiftOperator = -1$, since $\stableRoot(-1) = 1$.
This yields the asymptotics.

Then, we consider 
\begin{equation*}
    \partial_{\fourierShift}\doubleZTransformed{\globalTruncationError}(\timeShiftOperator^{-1}, \fourierShift^{-1}) =
    \frac{\courantNumber \timeShiftOperator^2(({1 - \timeShiftOperator^2 + \courantNumber\timeShiftOperator(\fourierShift^{-1}   + \stableRoot(\timeShiftOperator^{-1}))})(1+\fourierShift^{-2}) - (1-\timeShiftOperator^2 - \courantNumber\timeShiftOperator (\fourierShift-\fourierShift^{-1})))}{(1 - \timeShiftOperator\sum_{k\geq 0}\coefficientEventualBoundaryScheme_k \stableRoot(\timeShiftOperator^{-1})^k - \timeShiftOperator^2\sum_{k\geq 0}\coefficientEventualBoundarySchemeOld_k\stableRoot(\timeShiftOperator^{-1})^k)(1-\timeShiftOperator^2 - \courantNumber\timeShiftOperator (\fourierShift-\fourierShift^{-1}))^2},
\end{equation*}
and in particular
\begin{equation*}
    \timeShiftOperator\mapsto \partial_{\fourierShift}\doubleZTransformed{\globalTruncationError}(\timeShiftOperator^{-1}, \fourierShift^{-1} = 1) =
    \frac{\courantNumber \timeShiftOperator^2({1 - \timeShiftOperator^2 + 2\courantNumber\timeShiftOperator(1+ \stableRoot(\timeShiftOperator^{-1}))})}{(1 - \timeShiftOperator\sum_{k\geq 0}\coefficientEventualBoundaryScheme_k \stableRoot(\timeShiftOperator^{-1})^k - \timeShiftOperator^2\sum_{k\geq 0}\coefficientEventualBoundarySchemeOld_k\stableRoot(\timeShiftOperator^{-1})^k)(1-\timeShiftOperator)^2(1+\timeShiftOperator)^2}.
\end{equation*}
Again, the leading singularity is a second-order pole at $\timeShiftOperator = -1$ (the function is regular at $\timeShiftOperator = 1$).
The expansion near this singularity reads 
\begin{multline*}
    \timeShiftOperator\mapsto \partial_{\fourierShift}\doubleZTransformed{\globalTruncationError}(\timeShiftOperator^{-1}, \fourierShift^{-1} = 1)
    =
    -\frac{\courantNumber^2}{1+\sum_{k\geq 0}(\coefficientEventualBoundaryScheme_k-\coefficientEventualBoundarySchemeOld_k)}(\timeShiftOperator+1)^{-2}
    \\
    +
    \frac{\courantNumber}{1+\sum_{k\geq 0}(\coefficientEventualBoundaryScheme_k-\coefficientEventualBoundarySchemeOld_k)}
    \Biggl ( 
    2\courantNumber
    +\frac{\sum_{k\geq 0}k(\coefficientEventualBoundaryScheme_k-\coefficientEventualBoundarySchemeOld_k) + \courantNumber \sum_{k\geq 0}(2\coefficientEventualBoundarySchemeOld_k-\coefficientEventualBoundaryScheme_k)}{1+\sum_{k\geq 0}(\coefficientEventualBoundaryScheme_k-\coefficientEventualBoundarySchemeOld_k)}
    \Biggr ) (\timeShiftOperator+1)^{-1} + \bigO{1},
\end{multline*}
yielding the asymptotics
\begin{equation*}
\sum_{\indexSpace\geq 0}\indexSpace\globalTruncationError_{\indexSpace}^{\indexTime}\sim 
-\frac{\courantNumber^2 (-1)^{\indexTime}\indexTime}{1+\sum_{k\geq 0}(\coefficientEventualBoundaryScheme_k-\coefficientEventualBoundarySchemeOld_k)}+\frac{\courantNumber(-1)^{\indexTime}}{1+\sum_{k\geq 0}(\coefficientEventualBoundaryScheme_k-\coefficientEventualBoundarySchemeOld_k)}
    \Biggl ( 
    \courantNumber
    +\frac{\sum_{k\geq 0}k(\coefficientEventualBoundaryScheme_k-\coefficientEventualBoundarySchemeOld_k) + \courantNumber \sum_{k\geq 0}(2\coefficientEventualBoundarySchemeOld_k-\coefficientEventualBoundaryScheme_k)}{1+\sum_{k\geq 0}(\coefficientEventualBoundaryScheme_k-\coefficientEventualBoundarySchemeOld_k)}
    \Biggr )
    + \bigO{\indexTime^{-3/2}}.
\end{equation*}
\end{proof}

\begin{proof}[Proof of \Cref{prop:asymptoticL2Stable}]
    Let $\indexSpace\in\naturals$.
    We first look at the ordinary generating function of the sequence of the squares $\indexTime\mapsto (\globalTruncationError_{\indexSpace}^{\indexTime})^2$.
    Thanks to the celebrated formula by Hadamard \cite[Equations (4) and (7)]{hadamard1899theoreme}, this reads 
    \begin{align*}
        \sum_{\indexTime\geq 0}
        \timeShiftOperator^{\indexTime}
        (\globalTruncationError_{\indexSpace}^{\indexTime})^2
        =
        \frac{1}{2\pi i}
        \oint 
        \zTransformed{\globalTruncationError}_{\indexSpace}(w^{-1})
        \zTransformed{\globalTruncationError}_{\indexSpace}(w \timeShiftOperator^{-1})
        \frac{\differential{w}}{w}
        =
        \frac{1}{2\pi}
        \int_0^{2\pi}
        \zTransformed{\globalTruncationError}_{\indexSpace}(e^{-i\vartheta}/\sqrt{\timeShiftOperator})
        \zTransformed{\globalTruncationError}_{\indexSpace}(e^{i\vartheta}/\sqrt{\timeShiftOperator})
        \differential{\vartheta}.
    \end{align*}
    Without much surprise, this formula is sometimes called ``Parseval integral representation'', e.g.
    \cite{render1997hadamard, pohlen2009hadamard, parol2021contribution}.
    
    Summing over $\indexSpace$ and using linearity, we obtain the generating function of the square of the $L^2$ norm:
    \begin{align*}
        \sum_{\indexTime\geq 0}
        \timeShiftOperator^{\indexTime}
        \Biggl ( \sum_{\indexSpace\geq 0}(\globalTruncationError_{\indexSpace}^{\indexTime})^2
        \Biggr )
        &=
        \frac{1}{2\pi}
        \int_0^{2\pi}
        \sum_{\indexSpace\geq 0}
        \zTransformed{\globalTruncationError}_{\indexSpace}(e^{-i\vartheta}/\sqrt{\timeShiftOperator})
        \zTransformed{\globalTruncationError}_{\indexSpace}(e^{i\vartheta}/\sqrt{\timeShiftOperator})
        \differential{\vartheta}\\
        &=-2\pi
        \int_0^{2\pi}
        \timeShiftOperator^{-1}
        g(e^{-i\vartheta}/\sqrt{\timeShiftOperator}) 
        g(e^{i\vartheta}/\sqrt{\timeShiftOperator}) 
        \sum_{\indexSpace\geq 0}
        (\stableRoot(e^{-i\vartheta}/\sqrt{\timeShiftOperator}) 
        \stableRoot(e^{i\vartheta}/\sqrt{\timeShiftOperator}))^{\indexSpace}
        \differential{\vartheta}\\
        &=-2\pi
        \int_0^{2\pi}
        \frac{
        \timeShiftOperator^{-1}
        g(e^{-i\vartheta}/\sqrt{\timeShiftOperator}) 
        g(e^{i\vartheta}/\sqrt{\timeShiftOperator})}{1-\stableRoot(e^{-i\vartheta}/\sqrt{\timeShiftOperator}) 
        \stableRoot(e^{i\vartheta}/\sqrt{\timeShiftOperator})}
        \differential{\vartheta},
    \end{align*}
    where the last equality comes from the geometric series.
    
    To look at the asymptotics of $\indexTime\mapsto \sum_{\indexSpace\geq 0}(\globalTruncationError_{\indexSpace}^{\indexTime})^2$ in the limit of large $\indexTime$, we can use the results in \cite{flajolet2009analytic}, since the generating function is known and given as above.
    One can see that the leading singularity in the integrand, whatever the value of $\vartheta$, is a first-order pole at $\timeShiftOperator = 1$.
    By \Cref{ass:stableBC}, the numerator of the integrand is regular.
    We get 
    \begin{equation*}
        \frac{
        \timeShiftOperator^{-1}
        g(e^{-i\vartheta}/\sqrt{\timeShiftOperator}) 
        g(e^{i\vartheta}/\sqrt{\timeShiftOperator})}{1-\stableRoot(e^{-i\vartheta}/\sqrt{\timeShiftOperator}) 
        \stableRoot(e^{i\vartheta}/\sqrt{\timeShiftOperator})}
        =
        \frac{g(e^{i\vartheta})g(e^{-i\vartheta})
         \sqrt{e^{4 i\vartheta}+2(2\courantNumber^2-1)e^{2 i\vartheta}+1}}{
         e^{2 i\vartheta}+1}
         (1-\timeShiftOperator)^{-1} 
         +\bigO{1},
    \end{equation*}
    hence the claim.
    Concerning the reminder $\bigO{\indexTime^{-3/2}}$, this comes from the presence of non-leading singularities on the branch points by \Cref{lemma:BranchPoints}.
\end{proof}

\begin{proof}[Proof of \Cref{prop:transitionZoneStable}]
Let us recall that for $\nu \in (0, |\courantNumber|)$ there are four non-degenerate saddle points on $\unitCircle$, whose expression is given by \eqref{eq:ksSaddlePoints}. In particular, the saddle points have principal arguments $\saddlePointAngle$, $-\saddlePointAngle$, $-\saddlePointAngle+\pi$, $\saddlePointAngle-\pi$. For notational convenience, and since no ambiguity will arise, the subscript $\textnormal{SP}$ is omitted throughout this proof. Let us focus on the expression of $f''(e^{\pm i \vartheta}; \nu)$. Since explicit computations yield
\begin{equation*}
    \stableRoot''(e^{i\vartheta})
    =
    \frac{(2\courantNumber^2-1)e^{6i\vartheta}+3e^{4i\vartheta}+3(2\courantNumber^2-1)e^{2i\vartheta} + 1 + 8\nu^3 e^{3i\vartheta} \Bigl (\frac{1-\courantNumber^2}{1-\nu^2}\Bigr )^{3/2}}{8 \courantNumber \nu^3 e^{6i\vartheta} \Bigl (\frac{1-\courantNumber^2}{1-\nu^2}\Bigr )^{3/2}},
\end{equation*}
we can rewrite the expression for \begin{equation*}
    f''(e^{i\vartheta}; \nu)=-\Bigl (1+\frac{1}{\nu}\Bigr )e^{-2i\vartheta} + \nu e^{i(\ksAtsaddlePointAngleplain-\pi)} \stableRoot''(e^{i\vartheta}),
\end{equation*}
as
\begin{multline*}
    f''(e^{i\vartheta}; \nu) = -\Bigl (1+\frac{1}{\nu}\Bigr )e^{-2i\vartheta} \\ -   \frac{e^{i(\ksAtsaddlePointAngleplain-6\vartheta)}}{{8 \courantNumber \nu^2 \Bigl (\frac{1-\courantNumber^2}{1-\nu^2}\Bigr )^{3/2}}} \Bigl ((2\courantNumber^2-1)e^{6i\vartheta}+3e^{4i\vartheta}+3(2\courantNumber^2-1)e^{2i\vartheta} + 1 + 8\nu^3 e^{3i\vartheta} \Bigl (\frac{1-\courantNumber^2}{1-\nu^2}\Bigr )^{3/2}\Bigr ). 
\end{multline*}
We now employ \eqref{eq:saddlePointComplex} with descent direction $\frac{\pi}{2}-\tfrac{1}{2}\text{Arg}(f''(e^{i\vartheta}; \nu))$ for $\saddlePoint = e^{i\vartheta}$, $\frac{\pi}{2}-\tfrac{1}{2}\text{Arg}(f''(e^{-i\vartheta}; \nu)) = \frac{\pi}{2}+\tfrac{1}{2}\text{Arg}(f''(e^{i\vartheta}; \nu))$ for $\saddlePoint = e^{-i\vartheta}$, and remark that the corresponding terms are complex conjugate; as well as $-\frac{\pi}{2}-\tfrac{1}{2} \text{Arg}(f''(e^{i(-\vartheta+\pi)}; \nu)) = -\frac{\pi}{2}+\tfrac{1}{2}\text{Arg}(f''(e^{i\vartheta}; \nu))$ for $\saddlePoint = e^{i(-\vartheta+\pi)}$, $-\frac{\pi}{2}-\tfrac{1}{2}\text{Arg}(f''(e^{i(\vartheta-\pi)}; \nu)) = -\frac{\pi}{2}-\tfrac{1}{2}\text{Arg}(f''(e^{i\vartheta}; \nu))$ for $\saddlePoint = e^{i(\vartheta-\pi)}$, and the same conjugacy arguments holds for this pair of saddle points.
For a more visual understanding, refer to the right-hand side of \Cref{fig:paths}: we integrate on the path $C$ featuring steepest descent directions when touching the saddle points.
Using \eqref{eq:ksSaddlePoints}, this yields
\begin{multline*}
    \globalTruncationError_{\indexSpace}^{\indexTime} \sim \sqrt{\frac{2}{\pi \indexTime |f''(e^{i\vartheta}; \nu)|}}\Biggl (\overbrace{\textnormal{Re}\Biggl ( \frac{\text{exp}(i(\indexTime\vartheta+\indexSpace(\pi-\ksAtsaddlePointAngleplain)-\tfrac{1}{2}\textnormal{Arg}(f''(e^{i\vartheta}; \nu))))}{e^{2i\vartheta}  - e^{i\vartheta}  \sum_{k\geq 0}(-1)^k \coefficientEventualBoundaryScheme_k e^{-i k \ksAtsaddlePointAngleplain} -  \sum_{k\geq 0}(-1)^k \coefficientEventualBoundarySchemeOld_k  e^{-i k\ksAtsaddlePointAngleplain}}\Biggr )}^{\text{saddle points}\,e^{i\vartheta}\,\text{and}\,e^{-i\vartheta}} \\
    \underbrace{-\textnormal{Re}\Biggl ( \frac{\text{exp}(i(\indexTime (\vartheta - \pi) - \indexSpace\ksAtsaddlePointAngleplain-\tfrac{1}{2}\textnormal{Arg}(f''(e^{i\vartheta}; \nu))))}{e^{2i\vartheta}  + e^{i\vartheta}  \sum_{k\geq 0} \coefficientEventualBoundaryScheme_k e^{-i k \ksAtsaddlePointAngleplain} -  \sum_{k\geq 0} \coefficientEventualBoundarySchemeOld_k  e^{-i k\ksAtsaddlePointAngleplain}}\Biggr )}_{\text{saddle points}\,e^{i(-\vartheta+\pi)}\,\text{and}\,e^{i(\vartheta - \pi)}} \Biggr ),
\end{multline*}
where the minus sign in front of the real part on the second line comes from the term $-\frac{\pi}{2}$ for the saddle points $e^{i(-\vartheta+\pi)}$ and $e^{i(\vartheta - \pi)}$.
Let us introduce the function of $\mathscr{G}_{\textnormal{R}}$ (the letter R stand for right-half plane) as
\begin{equation*}
     \mathscr{G}_{\textnormal{R}}(\vartheta, \ksAtsaddlePointAngleplain) \definitionEquality{} 
     2\pi i \, e^{i\vartheta} g(e^{i\vartheta})
     = e^{i\vartheta}  - \sum_{k\geq 0}(-1)^k \coefficientEventualBoundaryScheme_k e^{-i k \ksAtsaddlePointAngleplain} -e^{-i\vartheta}\sum_{k\geq 0}(-1)^k \coefficientEventualBoundarySchemeOld_k  e^{-i k\ksAtsaddlePointAngleplain},
\end{equation*}
and let us denote by $\mathscr{G}_{\textnormal{R}}^{\textnormal{Re}}$ and $\mathscr{G}_{\textnormal{R}}^{\textnormal{Im}}$ the real and the imaginary part of $\mathscr{G}_{\textnormal{R}}$ respectively, \idEst
\begin{equation*}
    \mathscr{G}_{\textnormal{R}}^{\textnormal{Re}}(\vartheta, \ksAtsaddlePointAngleplain)=\cos(\vartheta) - \sum_{k\geq 0}(-1)^k \coefficientEventualBoundaryScheme_k \cos(k\ksAtsaddlePointAngleplain)  - \cos(\vartheta) \sum_{k\geq 0}(-1)^k \coefficientEventualBoundarySchemeOld_k  \cos(k\ksAtsaddlePointAngleplain) + \sin(\vartheta) \sum_{k\geq 0}(-1)^k \coefficientEventualBoundarySchemeOld_k  \sin(k\ksAtsaddlePointAngleplain),
\end{equation*}
\begin{equation*}
    \mathscr{G}_{\textnormal{R}}^{\textnormal{Im}}(\vartheta, \ksAtsaddlePointAngleplain)=\sin(\vartheta) + \sum_{k\geq 0}(-1)^k \coefficientEventualBoundaryScheme_k \sin(k\ksAtsaddlePointAngleplain) +  \cos(\vartheta) \sum_{k\geq 0}(-1)^k \coefficientEventualBoundarySchemeOld_k  \sin(k\ksAtsaddlePointAngleplain) + \sin(\vartheta) \sum_{k\geq 0}(-1)^k \coefficientEventualBoundarySchemeOld_k  \cos(k\ksAtsaddlePointAngleplain). 
\end{equation*}
Then, 
\begin{multline*}
\textnormal{Re}\Biggl ( \frac{\text{exp}(i(\indexTime\vartheta+\indexSpace(\pi-\ksAtsaddlePointAngleplain)-\tfrac{1}{2}\textnormal{Arg}(f''(e^{i\vartheta}; \nu))))}{e^{2i\vartheta}  - e^{i\vartheta}  \sum_{k\geq 0}(-1)^k \coefficientEventualBoundaryScheme_k e^{-i k \ksAtsaddlePointAngleplain} -  \sum_{k\geq 0}(-1)^k \coefficientEventualBoundarySchemeOld_k  e^{-i k\ksAtsaddlePointAngleplain}}\Biggr ) \\
=
  (-1)^{\indexSpace}\textnormal{Re}\Biggl ( \frac{\text{exp}(i((\indexTime-1)\vartheta-\indexSpace\ksAtsaddlePointAngleplain-\tfrac{1}{2}\textnormal{Arg}(f''(e^{i\vartheta}; \nu))))}{\mathscr{G}_{\textnormal{R}}(\vartheta, \ksAtsaddlePointAngleplain)} \Biggr) \\ 
  = (-1)^{\indexSpace}\textnormal{Re}\Biggl ( \frac{\text{exp}(i((\indexTime-1)\vartheta-\indexSpace\ksAtsaddlePointAngleplain-\tfrac{1}{2}\textnormal{Arg}(f''(e^{i\vartheta}; \nu))))}{(\mathscr{G}_{\textnormal{R}}^{\textnormal{Re}}(\vartheta, \ksAtsaddlePointAngleplain))^2+(\mathscr{G}_{\textnormal{R}}^{\textnormal{Im}}(\vartheta,\ksAtsaddlePointAngleplain))^2} (\mathscr{G}_{\textnormal{R}}^{\textnormal{Re}}(\vartheta, \ksAtsaddlePointAngleplain) - i \mathscr{G}_{\textnormal{R}}^{\textnormal{Im}}(\vartheta, \ksAtsaddlePointAngleplain))\Biggr) \\ 
  = (-1)^{\indexSpace} \Biggl ( \frac{\text{cos}((\indexTime-1)\vartheta-\indexSpace \ksAtsaddlePointAngleplain-\tfrac{1}{2}\textnormal{Arg}(f''(e^{i\vartheta}; \nu)))}{(\mathscr{G}_{\textnormal{R}}^{\textnormal{Re}}(\vartheta, \ksAtsaddlePointAngleplain))^2+(\mathscr{G}_{\textnormal{R}}^{\textnormal{Im}}(\vartheta, \ksAtsaddlePointAngleplain))^2} \mathscr{G}_{\textnormal{R}}^{\textnormal{Re}}(\vartheta, \ksAtsaddlePointAngleplain) \\+ \frac{\text{sin}((\indexTime-1)\vartheta-\indexSpace \ksAtsaddlePointAngleplain-\tfrac{1}{2}\textnormal{Arg}(f''(e^{i\vartheta}; \nu)))}{(\mathscr{G}_{\textnormal{R}}^{\textnormal{Re}}(\vartheta, \ksAtsaddlePointAngleplain))^2+(\mathscr{G}_{\textnormal{R}}^{\textnormal{Im}}(\vartheta, \ksAtsaddlePointAngleplain))^2} \mathscr{G}_{\textnormal{R}}^{\textnormal{Im}}(\vartheta, \ksAtsaddlePointAngleplain) \Biggr ).
\end{multline*}
and analogously for the expression concerning the saddle points $e^{i(-\vartheta+\pi)}$ and $e^{i(\vartheta-\pi)}$, where we define the function (the letter L stands for left-half plane)
\begin{equation*}
     \mathscr{G}_{\textnormal{L}}(\vartheta, \ksAtsaddlePointAngleplain) \definitionEquality{} 
     2\pi i \, e^{i\vartheta} g(e^{i(\vartheta-\pi)})
     = e^{i\vartheta}  + \sum_{k\geq 0}\coefficientEventualBoundaryScheme_k e^{-i k \ksAtsaddlePointAngleplain} -e^{-i\vartheta}\sum_{k\geq 0} \coefficientEventualBoundarySchemeOld_k  e^{-i k\ksAtsaddlePointAngleplain}.
\end{equation*}

To explicitly evaluate the expression of $\mathscr{G}_{\textnormal{R}}$ and $\mathscr{G}_{\textnormal{L}}$ for the the upwind case, see  \Cref{ex:upwind}, let us set $\coefficientEventualBoundaryScheme_0 = 1+\courantNumber$, $\coefficientEventualBoundaryScheme_1 = -\courantNumber$. Note that 
\begin{equation*}
    \sin(\vartheta) = \sqrt{\frac{\courantNumber^2-\nu^2}{1-\nu^2}}, \qquad \sin(\ksAtsaddlePointAngleplain) =\frac{1}{|\courantNumber|} \sqrt{\frac{\courantNumber^2-\nu^2}{1-\nu^2}}, \qquad \cos(\vartheta) = \sqrt{\frac{1-\courantNumber^2}{1-\nu^2}}, \qquad \cos(\ksAtsaddlePointAngleplain) =-\frac{\nu}{\courantNumber} \sqrt{\frac{1-\courantNumber^2}{1-\nu^2}},
\end{equation*}
which entails 
$\mathscr{G}_{\textnormal{R}}^{\textnormal{Im}} = \mathscr{G}_{\textnormal{L}}^{\textnormal{Im}} = 0$. 
Moreover, since $\coefficientEventualBoundarySchemeOld_k = 0$ for all $k\in\naturals$, the expressions of $\mathscr{G}_{\textnormal{R}}^{\textnormal{Re}}$ and $\mathscr{G}_{\textnormal{L}}^{\textnormal{Re}}$ reduce to
\begin{equation*}
    \mathscr{G}_{\textnormal{R}}^{\textnormal{Re}} = - (1 + \courantNumber) + \sqrt{(1-\courantNumber^2)\frac{1+\nu}{1-\nu}}  
    \qquad \text{and} \qquad
    \mathscr{G}_{\textnormal{L}}^{\textnormal{Re}} = (1 + \courantNumber) + \sqrt{(1-\courantNumber^2)\frac{1+\nu}{1-\nu}}  ,
\end{equation*}
hence the claim.
\end{proof}

\begin{proof}[Proof of \Cref{prop:frontZoneStable}]
Consider $\nu= \indexSpace/\indexTime$ and $\delta = \indexSpace-\courantNumber\indexTime = \bigO{1}$, so that $\nu = \courantNumber+\delta/\indexTime$. Call $\alpha = \delta/\indexTime$, meant to be small. Let us solve the fourth-order equation associated to the saddle points, 
whose solutions have the following explicit expressions
\begin{equation*}
    \timeShiftOperator_{\textnormal{SP}, r}^{\pm 1}(\alpha) = \pm \sqrt{\frac{\courantNumber^2 +2\courantNumber\alpha-\alpha^2 - 1 + 2(-1)^r \sqrt{(\alpha-2\courantNumber)(1-\courantNumber^2)}\sqrt{\alpha}}{\courantNumber^2-2\courantNumber\alpha+\alpha^2-1}},
\end{equation*}
see \eqref{eq:realSaddlePoints}.
Here $r = 0, 1$ and $\lim_{\alpha\to 0}\timeShiftOperator_{\textnormal{SP},r}^{\pm 1}(\alpha) = \pm 1$.
Each of these two clusters give a contribution to the leading-order expansion of $\globalTruncationError_{\indexSpace}^{\indexTime} $ for large $\indexTime$, which we thus write as 
\begin{equation*}
    \globalTruncationError_{\indexSpace}^{\indexTime} = \frac{1}{2\pi i} \oint_{C} g(\timeShiftOperator)e^{\indexTime f(\timeShiftOperator; \alpha)}\differential{\timeShiftOperator} \sim \mathscr{I}^{\indexTime}_{1, \indexSpace} + \mathscr{I}^{\indexTime}_{- 1, \indexSpace},
\end{equation*}
where $\mathscr{I}^{\indexTime}_{\pm 1, \indexSpace}$ stem from the two saddle points coalescing to $ \pm 1$.
Let us discuss $\mathscr{I}^{\indexTime}_{1, \indexSpace}$ in detail, as the computations for $\mathscr{I}^{\indexTime}_{-1, \indexSpace}$ are analogous.
We follow the procedure by Chester-Friedman-Ursell \cite{chester1957extension} which deals with two coalescing saddle-points and fosters uniform asymptotics around $\alpha=0$.
In this spirit, we look for a change of basis allowing to write
\begin{equation*}
    f(\timeShiftOperator_{\textnormal{SP},r}^{+1}(\alpha);\alpha) = \frac{1}{3}u^3 - \zeta(\alpha) u + \eta(\alpha),
\end{equation*}
where $\zeta(\alpha)$ and $\eta(\alpha)$ are to be determined.
From \cite[Equation (4.8)]{wong2001asymptotic}, we have 
\begin{equation*}
    \zeta(\alpha)^{3/2} = \frac{3}{4}(f(\timeShiftOperator_{\textnormal{SP},0}^{1}(\alpha);\alpha)-f(\timeShiftOperator_{\textnormal{SP}, 1}^{1}(\alpha);\alpha)) = \sqrt{\frac{2}{\courantNumber(\courantNumber^2-1)}}\alpha^{3/2} + \bigO{\alpha^{5/2}}.
\end{equation*}
Since we are interested in the limit of small $\alpha$, we retain only the leading order, thus utilize
\begin{equation*}
    \zeta(\alpha)\sim \Bigl (\frac{2}{\courantNumber(\courantNumber^2-1)} \Bigr )^{1/3}\alpha,
\end{equation*}
analogously to \cite[Equation (3.5)]{berry2005tsunami}.
From Equation (4.9) in Wong's monograph
\begin{equation*}
    \eta(\alpha) = \frac{1}{2}(f(\timeShiftOperator_{\textnormal{SP},0}^{1}(\alpha);\alpha)+f(\timeShiftOperator_{\textnormal{SP},1}^{1}(\alpha);\alpha)) = i\pi(\alpha-\courantNumber).
\end{equation*}
This provides 
\begin{multline}\label{eq:tmp1}
    \mathscr{I}^{\indexTime}_{1, \indexSpace}
    \sim e^{\indexTime i\pi(\alpha-\courantNumber)} 
    \frac{1}{2\pi i}\int e^{\indexTime (\frac{1}{3}u^3 - (\frac{2}{\courantNumber(\courantNumber^2-1)}  )^{1/3}\alpha u)} g(\timeShiftOperator(u)) \frac{\differential{\timeShiftOperator(u)}}{\differential{u}}\differential{u}\\
    =(-1)^{\indexSpace} 
    \frac{1}{2\pi i}\int e^{\frac{1}{3}\indexTime u^3 - (\frac{2}{\courantNumber(\courantNumber^2-1)}  )^{1/3}\delta u} g(\timeShiftOperator(u)) \frac{\differential{\timeShiftOperator(u)}}{\differential{u}}\differential{u}.
\end{multline}
We are now left to deal with $g(\timeShiftOperator(u)) \frac{\differential{\timeShiftOperator(u)}}{\differential{u}}$.
We follow the procedure by \cite{bleistein1967uniform} and write 
\begin{equation*}
    g(\timeShiftOperator(u)) \frac{\differential{\timeShiftOperator(u)}}{\differential{u}} = a_0(\alpha) + a_1(\alpha) u + (u^2 - \zeta(\alpha))\psi_2(u; \alpha) \sim a_0(\alpha),
\end{equation*}
where $a_0$, $a_1$, and $\psi_2$ can be determined, and the last approximation is done as the contributions from the second and last term decay to zero quicker than the first one, see Equation (4.21) in Wong's book.
By Equation (4.16) in the same reference,
\begin{equation*}
    a_0(\alpha) = \frac{1}{2}\Bigl ( g(\timeShiftOperator(u)) \frac{\differential{\timeShiftOperator(u)}}{\differential{u}} \Bigr |_{u = \zeta(\alpha)^{1/2}} + g(\timeShiftOperator(u)) \frac{\differential{\timeShiftOperator(u)}}{\differential{u}} \Bigr |_{u = -\zeta(\alpha)^{1/2}} \Bigr ),
\end{equation*}
namely $a_0(\alpha)$ is the average of the contribution from the two coalescing saddle points.
We make the leading-order approximation $g(\timeShiftOperator(\pm \zeta(\alpha)^{1/2})) \sim g(1) = -\frac{1}{2\courantNumber}$ as $|\alpha|\ll 1$.
By virtue of \cite[Equation (4.11)]{wong2001asymptotic} and expanding in $\alpha$, we gain 
\begin{equation*}
    \frac{\differential{\timeShiftOperator(u)}}{\differential{u}} \Bigr |_{u = \pm \zeta(\alpha)^{1/2}}
    =
    \Biggl ( (-1)^{r+1}\frac{2\zeta(\alpha)^{1/2}}{f''(\timeShiftOperator_{\textnormal{SP},r}^{1}(\alpha); \alpha)}\Biggr )^{1/2}
    = \Bigl ( \frac{2\courantNumber^2}{1-\courantNumber^2}\Bigr )^{1/3}+ \bigO{\alpha^{1/2}}
\end{equation*}
where the neglected terms between $\pm$ may be different.
Overall, this gives the estimation
\begin{equation*}
    g(\timeShiftOperator(u)) \frac{\differential{\timeShiftOperator(u)}}{\differential{u}} \sim \frac{1}{2}\Biggl (\frac{1}{\frac{\courantNumber}{2}(\courantNumber^2-1)}\Biggr )^{1/3},
\end{equation*}
hence the right-hand side of \eqref{eq:tmp1} is estimated, after a change of variable $v = \indexTime^{1/3}u$  in the integral, by
\begin{multline*}
    \mathscr{I}^{\indexTime}_{1, \indexSpace}
    \sim 
    \frac{(-1)^{\indexSpace}}{2}
    \frac{1}{(\frac{\courantNumber}{2}(\courantNumber^2-1)\indexTime)^{1/3}}
    \frac{1}{2\pi i}\int e^{\frac{1}{3} v^3 - \frac{\delta}{(\frac{\courantNumber}{2}(\courantNumber^2-1)\indexTime)^{1/3}} v} \differential{v}\\
    =
    \frac{(-1)^{\indexSpace}}{2}
    \frac{1}{(\frac{\courantNumber}{2}(\courantNumber^2-1)\indexTime)^{1/3}} \text{Ai}\Bigl ( \frac{\indexSpace+\courantNumber\indexTime}{(\frac{\courantNumber}{2}(\courantNumber^2-1)\indexTime)^{1/3}}\Bigr ),
\end{multline*}
hence the claim.
\end{proof}

Let us know discuss the case where $\nu \in (|\courantNumber|, 1)$.
By \Cref{lemma:saddlePoints}, we have two saddle points in $\unitDisk$ and two in $\neighborhoodInfinity$, all real.
Those which can be reached by choosing a contour in the region of convergence are the latter two.
Consider $\saddlePoint(\nu)\in\neighborhoodInfinity \cap \reals$.
We now verify that $|\saddlePoint(\nu)\stableRoot(\saddlePoint(\nu))^{\nu} | = |e^{f(\saddlePoint(\nu); \nu)} | = e^{\textnormal{Re}(f(\saddlePoint(\nu); \nu))}<1$, hence an exponential decrease in $\indexTime$ (up to algebraic multiplicative factors, e.g. $\bigO{\indexTime^{-1/2}}$).
Note that $\lim_{\nu\to |\courantNumber|}\textnormal{Re}(f(\saddlePoint(\nu); \nu)) = 0$ (moreover $\textnormal{Re}(f(\saddlePoint(\nu); \nu)) = 0$ for $\nu\in(0, |\courantNumber|]$).
One can verify that the function $\nu\mapsto\textnormal{Re}(f(\saddlePoint(\nu); \nu)) = \log(|\saddlePoint(\nu)|) + \nu \log(|\stableRoot(\saddlePoint(\nu))|)$ with $\saddlePoint(\nu)\in\neighborhoodInfinity \cap \reals$ is decreasing, thus in this zone $e^{\textnormal{Re}(f(\saddlePoint(\nu); \nu))}<1$. 

\begin{proof}[Proof of \Cref{prop:nearWallUnstable} and \ref{prop:transitionUnstable}]
    Consider that $\oint_C = \oint_{\gamma_{\varepsilon}} + \oint_{\tilde{C}\smallsetminus\gamma_{\varepsilon}}$, where $\tilde{C}$ is illustrated on the right of \Cref{fig:paths} in the case of \Cref{prop:transitionUnstable}.
    For \Cref{prop:nearWallUnstable}, $\tilde{C}\smallsetminus\gamma_{\varepsilon}$ is a simple positively-oriented path in the region of convergence  $\neighborhoodInfinity$.
    Therefore, the only singularity enclosed by $\gamma_{\epsilon}$ is the simple pole at $\timeShiftOperator = -1$.
    $\tilde{C}\smallsetminus\gamma_{\varepsilon}$ encloses the four branch points by \Cref{lemma:BranchPoints}, which are the sole singularities for the associated integral.
    
    Concerning \Cref{prop:transitionUnstable}, again the only singularity enclosed by $\gamma_{\epsilon}$ is the simple pole at $\timeShiftOperator = -1$, and $\tilde{C}\smallsetminus\gamma_{\varepsilon}$ passes arbitrarily close to the saddle points for \Cref{prop:transitionUnstable}.

    By the residue theorem
    \begin{equation*}
    \oint_{\gamma_{\varepsilon}} g(\timeShiftOperator)e^{\indexTime f(\timeShiftOperator; \nu)}\differential{\timeShiftOperator} = 2\pi i \textnormal{Res}_{-1} \bigl [ g(\timeShiftOperator)e^{\indexTime f(\timeShiftOperator; \nu)}\bigr ] = -(-1)^\indexTime \Bigl(2 + \sum_{k\geq 0}\coefficientEventualBoundaryScheme_k + \frac{1}{\courantNumber}\sum_{k\geq 0}k(\coefficientEventualBoundaryScheme_k-\coefficientEventualBoundarySchemeOld_k) \Bigr)^{-1} .
    \end{equation*}
    using \eqref{eq:simpleZeroAtMinusOne}.
    The terms $\oint_{\tilde{C}\smallsetminus\gamma_{\varepsilon}}$ are estimated as in \Cref{prop:nearWallStable} and \ref{prop:transitionZoneStable}, yielding the claim.
\end{proof}

\begin{proof}[Proof of \Cref{prop:FrontUnstable}]
In this case, we essentially restart from \eqref{eq:tmp1} (rewritten for the couple of coalescing saddle points towards $\timeShiftOperator = -1$), which reads
    \begin{equation*}
    \globalTruncationError_{\indexSpace}^{\indexTime}
    \sim
    (-1)^{\indexTime} 
    \frac{1}{2\pi i}\int e^{\frac{1}{3}\indexTime u^3 - (\frac{2}{\courantNumber(\courantNumber^2-1)}  )^{1/3}\delta u} g(\timeShiftOperator(u)) \frac{\differential{\timeShiftOperator(u)}}{\differential{u}}\differential{u}.
\end{equation*}
At leading order for small $\alpha$, we notice that 
\begin{equation*}
    g(\timeShiftOperator(u)) \frac{\differential{\timeShiftOperator(u)}}{\differential{u}} = 
    \textnormal{Res}_{-1}[g]u^{-1}+ \dots
    =-
    \Bigl(2 + \sum_{k\geq 0}\coefficientEventualBoundaryScheme_k + \frac{1}{\courantNumber}\sum_{k\geq 0}k(\coefficientEventualBoundaryScheme_k-\coefficientEventualBoundarySchemeOld_k) \Bigr)^{-1}u^{-1} + \dots,
\end{equation*}
due to the presence of a first-order pole in $g(\timeShiftOperator)$ at $\timeShiftOperator=-1$.
This entails the approximation, after a change of variable $v = \indexTime^{1/3}u$:
    \begin{multline*}
    \globalTruncationError_{\indexSpace}^{\indexTime}
    \sim\textnormal{Res}_{-1}[g](-1)^{\indexTime} 
    \frac{1}{2\pi i}\int \frac{e^{\frac{1}{3} v^3 - (\frac{2}{\courantNumber(\courantNumber^2-1)\indexTime}  )^{1/3}\delta v} }{v}\differential{v}\\
    =\textnormal{Res}_{-1}[g](-1)^{\indexTime} 
    \Biggl ( 
    \frac{1}{2\pi i}\int e^{\frac{1}{3} v^3}\frac{1-0}{v}\differential{v}
    -\frac{1}{2\pi i}\int e^{\frac{1}{3} v^3} \frac{1-e^{- (\frac{2}{\courantNumber(\courantNumber^2-1)\indexTime}  )^{1/3}\delta v} }{v}\differential{v}
    \Biggr )\\
    =\textnormal{Res}_{-1}[g](-1)^{\indexTime} 
    \Biggl ( 
    \frac{1}{2\pi i}\int e^{\frac{1}{3} v^3}\int_0^{+\infty} e^{-y v}\differential{y}\differential{v}
    -\frac{1}{2\pi i}\int e^{\frac{1}{3} v^3}\int_0^{(\frac{2}{\courantNumber(\courantNumber^2-1)\indexTime}  )^{1/3}\delta} e^{-y v}\differential{y}\differential{v}
    \Biggr )\\
    =-
    \Bigl(2 + \sum_{k\geq 0}\coefficientEventualBoundaryScheme_k + \frac{1}{\courantNumber}\sum_{k\geq 0}k(\coefficientEventualBoundaryScheme_k-\coefficientEventualBoundarySchemeOld_k) \Bigr)^{-1}(-1)^{\indexTime} 
    \Biggl ( 
    \int_0^{+\infty}\textnormal{Ai}(y)\differential{y}
    -
    \int_0^{(\frac{2}{\courantNumber(\courantNumber^2-1)\indexTime}  )^{1/3}\delta}\textnormal{Ai}(y)\differential{y}
    \Biggr ),
\end{multline*}
where the last equality comes from an exchange of the order of integration in a Fubini-like fashion.
\end{proof}

\subsubsection{A dissipative bulk scheme}

\begin{proof}[Proof of \Cref{prop:GaussianPeak}]
    Consider $\nu= \indexSpace/\indexTime$ and $\delta = \indexSpace-\courantNumber\indexTime = \bigO{1}$, so that $\nu = \courantNumber+\delta/\indexTime$. Call $\alpha = \delta/\indexTime$, meant to be small. 
    We solve the fourth-order equation associated to the saddle points. This gives two solutions $\saddlePoint^{\pm 1}(\alpha)$ such that $\lim_{\alpha\to 0}\saddlePoint^{\pm 1}(\alpha) = \pm 1$.
    Computations provide 
    \begin{align*}
        f(\saddlePoint^{+1}(\alpha); \courantNumber + \alpha) &= 
        -\frac{\alpha^2}{4(\frac{1}{\relaxationParameter}-\frac{1}{2})(1-\courantNumber^2)} + \bigO{\alpha^3},
        \\
        f(\saddlePoint^{-1}(\alpha); \courantNumber + \alpha) &=
        i\pi (1+\courantNumber) + i\pi \alpha  
        -\frac{\alpha^2}{4(\frac{1}{\relaxationParameter}-\frac{1}{2})\textbf{}} + \bigO{\alpha^3},
        \\
        f''(\saddlePoint^{\pm 1}(\alpha); \courantNumber + \alpha) &= \underbrace{\frac{2}{\courantNumber^2}\Bigl ( \frac{1}{\relaxationParameter}-\frac{1}{2}\Bigr )(1-\courantNumber^2)}_{>0} +\bigO{\alpha}.
    \end{align*}
    These truncated expressions into \eqref{eq:saddlePointComplex} give, after straightforward manipulations, \eqref{eq:gaussian}.
\end{proof}

\section{Green functions of the leap-frog scheme on $\relatives$}\label{sec:GreenFunction}

We finish the paper by considering the leap-frog scheme for the Cauchy problem on $\reals$, \idEst{} without boundary.
This reads 
\begin{align}
    &\solutionCauchyProblem_{\indexSpace}^0
    \qquad \text{and}\qquad 
    \solutionCauchyProblem_{\indexSpace}^1 
    \quad\text{given, for }\indexSpace\in\relatives, \label{eq:leapFrogCauchyInitialization} \\
    \indexTime\geq 1, \qquad 
    &\solutionCauchyProblem_{\indexSpace}^{\indexTime + 1}
    =
    \solutionCauchyProblem_{\indexSpace}^{\indexTime - 1}
    +\courantNumber
    (\solutionCauchyProblem_{\indexSpace-1}^{\indexTime}
    -\solutionCauchyProblem_{\indexSpace+1}^{\indexTime}).  \label{eq:leapFrogCauchyBulk}
\end{align}
Generally, one takes $\solutionCauchyProblem_{\indexSpace}^0 = \solutionCauchyProblem^{\initial}(\indexSpace\spaceStep)$ and $\solutionCauchyProblem_{\indexSpace}^1 = \sum_{k\in\relatives} \coefficientInitialBulkScheme_{k} \solutionCauchyProblem_{\indexSpace + k}^0$, where the coefficients $\coefficientInitialBulkScheme_{k}$ satisfy \eqref{eq:orderConstraints} to ensure overall second-order accuracy.
Apart from this common choice, the general solution of \eqref{eq:leapFrogCauchyInitialization}--\eqref{eq:leapFrogCauchyBulk} can be written, thanks to the \strong{superposition principle}, for $\indexTime\geq 2$, as 
\begin{equation*}
    \solutionCauchyProblem_{\indexSpace}^{\indexTime}
    =
    (\firstGreenFunction^{\indexTime}\ast \solutionCauchyProblem^0)_{\indexSpace}
    +
    (\secondGreenFunction^{\indexTime}\ast \solutionCauchyProblem^1)_{\indexSpace}
    =
    \sum_{k \in \relatives}
    \firstGreenFunction_{k}^{\indexTime}
    \solutionCauchyProblem_{\indexSpace-k}^0
    +
    \sum_{k \in \relatives}
    \secondGreenFunction_{k}^{\indexTime}
    \solutionCauchyProblem_{\indexSpace-k}^1,
\end{equation*}
where $\firstGreenFunction_{k}^{\indexTime}$ (respectively, $\secondGreenFunction_{k}^{\indexTime}$) is called ``first \strong{Green function}'' (respectively, ``second \strong{Green function}'').
They are defined by 
\begin{multline*}
    \begin{cases}
        \firstGreenFunction_{\indexSpace}^{\indexTime + 1} = \firstGreenFunction_{\indexSpace}^{\indexTime - 1} + \courantNumber (\firstGreenFunction_{\indexSpace-1}^{\indexTime}-\firstGreenFunction_{\indexSpace+1}^{\indexTime}), \qquad \indexTime\geq 1, \quad \indexSpace\in\relatives, \\
        \firstGreenFunction_{\indexSpace}^0 = \delta_{\indexSpace 0}, \\
        \firstGreenFunction_{\indexSpace}^1 = 0,
    \end{cases}
    \\
    \textnormal{and}
    \qquad \qquad
    \begin{cases}
        \secondGreenFunction_{\indexSpace}^{\indexTime + 1} = \secondGreenFunction_{\indexSpace}^{\indexTime - 1} + \courantNumber (\secondGreenFunction_{\indexSpace-1}^{\indexTime}-\secondGreenFunction_{\indexSpace+1}^{\indexTime}), \qquad \indexTime\geq 1, \quad \indexSpace\in\relatives, \\
        \secondGreenFunction_{\indexSpace}^0 = 0, \\
        \secondGreenFunction_{\indexSpace}^1 = \delta_{\indexSpace 0}.
    \end{cases}
\end{multline*}
This means that the discrete solution can be seen as a superposition of Green functions, whence the interest of studying these latter individually in the limit $\indexTime\gg 1$.
\begin{remark}[Explicit forms]
    Explicit forms of $\firstGreenFunction_{\indexSpace}^{\indexTime}$ and $\secondGreenFunction_{\indexSpace}^{\indexTime}$ are available, see \cite{cheng1999general}.
    However, as in the case with boundary conditions (\confer{} Appendix \ref{app:explicitExpressionUpwind}), they are rather involved and provide little insight into the structure of the Green functions themselves.
\end{remark}

In the case of leap-frog scheme, the two Green functions are tightly linked.
One can thus study one of them and extend the considered property to the other.
\begin{lemma}[Elementary properties of the Green functions]
    For every $\indexTime\in\naturals$ and $\indexSpace\in\relatives$, we have 
    \begin{equation*}
        \secondGreenFunction_{\indexSpace}^{\indexTime} = \firstGreenFunction_{\indexSpace}^{\indexTime + 1}.
    \end{equation*}
    Moreover, the following properties hold (stated for $\firstGreenFunction$ for simplicity).
    \begin{itemize}
        \item \strong{Support}. Let $\indexTime\geq 2$, then $\firstGreenFunction_{\indexSpace}^{\indexTime} = 0$ if $|\indexSpace|\geq \indexTime - 1$.
        \item \strong{Parity}. $\firstGreenFunction_{\indexSpace}^{\indexTime} = 0$ if $\indexTime$ and $\indexSpace$ have different parities. Moreover, for $\indexTime$ even (respectively, odd) $\indexSpace\mapsto\firstGreenFunction_{\indexSpace}^{\indexTime}$ is an even (respectively, odd) function.
        \item \strong{Values at the support-boundary}. Let $\indexTime\geq 2$, then $\firstGreenFunction_{\indexTime - 2}^{\indexTime} = \courantNumber^{\indexTime - 2}$ and $\firstGreenFunction_{-\indexTime + 2}^{\indexTime} = (-\courantNumber)^{\indexTime - 2}$.
    \end{itemize}
\end{lemma}

In what follows, we study the second Green function, as it yields slightly simpler expressions than the first one.
Since we are on the whole space $\relatives$, we can use Fourier analysis, see \cite[Chapter 2]{strikwerda2004finite}.
 The Fourier-transformed bulk scheme gives the characteristic equation $\timeShiftOperator^2 + 2i\courantNumber\sin(\frequency)\timeShiftOperator - 1=0$ for $\frequency\in[-\pi, \pi]$.
Its two roots (also called ``symbols'') belong to $\unitCircle$ for every $\frequency$ and their product is constant and equal to $-1$.
    We indicate the root equal to one when $\frequency = 0$ by $\timeShiftOperator_{\varphi}(\frequency)$, which is given by
    \begin{equation*}
        \timeShiftOperator_{\varphi}(\frequency) = \textnormal{exp}(-i\arcsin(\courantNumber\sin(\frequency))),
        \qquad \textnormal{which features the phase}\qquad 
        \vartheta_{\varphi}(\frequency)\definitionEquality-\arcsin(\courantNumber\sin(\frequency))
        .
    \end{equation*}
    The general solution of the Fourier-transformed bulk scheme thus reads $\varphi(\frequency)\timeShiftOperator_{\varphi}(\frequency)^{\indexTime} + \sigma(\frequency)(-1)^{\indexTime} \timeShiftOperator_{\varphi}(\frequency)^{-\indexTime}$, where $\varphi$ (standing for ``physical'') and $\sigma$ (standing for ``spurious'') are determined by the initial data.
    For the second Green function this gives 
    \begin{align*}
        \fourierTransformed{\secondGreenFunction}^{\indexTime}(\frequency)
        &=
        \frac{1}{\sqrt{2\pi}} \Biggl ( \frac{\textnormal{exp}(-i\indexTime\arcsin(\courantNumber\sin(\frequency)))}{2\sqrt{1-\courantNumber^2\sin^2(\frequency)}} 
        +(-1)^{\indexTime + 1}\frac{\textnormal{exp}(i\indexTime\arcsin(\courantNumber\sin(\frequency)))}{2\sqrt{1-\courantNumber^2\sin^2(\frequency)}} \Biggr )\\
        &=
        \frac{1}{\sqrt{2\pi}}
        \Biggl (
            \frac{(-1)^{\indexTime}+1}{2}\frac{T_n\bigl (\sqrt{1-\courantNumber^2\sin^2(\frequency)}\bigr )}{\sqrt{1-\courantNumber^2\sin^2(\frequency)}} 
            +i \frac{1-(-1)^{\indexTime}}{2}\frac{\courantNumber\sin(\frequency) U_{n-1}\bigl (\sqrt{1-\courantNumber^2\sin^2(\frequency)}\bigr )}{\sqrt{1-\courantNumber^2\sin^2(\frequency)}}
        \Biggr ),
    \end{align*}
    where $T_n$ (respectively, $U_n)$ are the Chebyshev polynomials of first kind (respectively, second kind).
    Taking the inverse Fourier transform yields 
    \begin{equation*}
        {\secondGreenFunction}_{\indexSpace}^{\indexTime} = 
        \frac{1}{{2\pi}} \Biggl ( \int_{-\pi}^{\pi}\frac{\textnormal{exp}(i(\indexSpace\frequency-\indexTime\arcsin(\courantNumber\sin(\frequency))))}{2\sqrt{1-\courantNumber^2\sin^2(\frequency)}} \differential{\frequency}
        +(-1)^{\indexTime + 1}\int_{-\pi}^{\pi} \frac{\textnormal{exp}(i(\indexSpace\frequency+\indexTime\arcsin(\courantNumber\sin(\frequency))))}{2\sqrt{1-\courantNumber^2\sin^2(\frequency)}} \differential{\frequency} \Biggr ).
    \end{equation*}

Letting $\nu = \indexSpace/\indexTime$, we rewrite things as
\begin{multline*}
    {\secondGreenFunction}_{\indexSpace}^{\indexTime} =
    \int_{-\pi}^{\pi}
    g(\frequency)
    e^{i \indexTime f_{-}(\frequency; \nu)}
    \differential{\frequency}
    +(-1)^{\indexTime + 1}
    \int_{-\pi}^{\pi}
    g(\frequency)
    e^{i \indexTime f_{+}(\frequency; \nu)}
    \differential{\frequency}, \qquad \textnormal{where}\\
    g(\frequency)
    =
    \frac{1}{{4\pi\sqrt{1-\courantNumber^2\sin^2(\frequency)}}}
    \quad \textnormal{and}\quad 
    f_{\pm}(\frequency; \nu)
    =
    \nu\frequency\pm \arcsin(\courantNumber\sin(\frequency)).
\end{multline*}
Notice that $f_-$ could be called $f_{\varphi}$, since associated to the physical symbol.
With the following lemma on the saddle points of $f_{\pm}$, we understand the structure of the solution is essentially as the one with boundary conditions, except for the \strong{absence of a near-wall zone}.
Indeed, even when $\nu = 0$, (regular) saddle points exists.
\begin{lemma}[Saddle points]\label{lemma:saddlePointsGreen}
    Assume $|\courantNumber|<1$ and $\courantNumber\neq 0$.
    Let $0\leq |\nu|\leq 1$.
    Then, the saddle points of the function $[-\pi, \pi]\ni\frequency \mapsto f_{\pm}(\frequency; \nu)$ are as follows.
    \begin{itemize}
        \item For $|\nu|\in (|\courantNumber|, 1]$, no saddle points.
        \item For $|\nu|\in [0, |\courantNumber|]$, the saddle points are 
        \begin{equation}\label{eq:spWithoutBoundary}
            \frequency_{\textnormal{SP}}
            =
            (-1)^r 
            \arccos
            \Bigl ( 
            \mp
            \frac{\nu}{\courantNumber}
            \sqrt{\frac{1-\courantNumber^2}{1-\nu^2}}
            \Bigr ) 
            \qquad \text{for}\qquad 
            r = 0, 1,
        \end{equation}
        with $f''(\frequency_{\textnormal{SP}}; \nu)
            =
            \pm (-1)^r
            \textnormal{sgn}(\courantNumber)
            (1-\courantNumber^{2})^{-1/2}
            (1-\nu^2)
            (\courantNumber^2-\nu^2)^{1/2}$.
        Therefore, the saddle points are non-degenerate for $|\nu|\in [0, |\courantNumber|)$ and degenerate for $|\nu| = |\courantNumber|$.
    \end{itemize}
\end{lemma}
As saddle points are present for $\nu = 0$, the scheme admits---contrarily to the case with boundary---glancing modes.
They correspond to $\frequency_{\textnormal{SP}} = \pm \tfrac{\pi}{2}$, which are well-known for the leap-frog scheme.

\begin{remark}[Link with \cite{thomee1965stability}]\label{rem:linkwithThomee}
    Let us study the link between saddle points, their degeneracy, and the Taylor expansion in \cite[Theorem 1]{thomee1965stability} around points (whose number is finite, contrarily to our setting) where the symbol belong to $\unitCircle$.
    This Taylor expansion of the logarithm of the symbol also appears in \cite{coulombel2022green, coulombel2022generalized, coulombel2023sharp, coeuret2025local}, for example.
    
    Now, the physical symbol $\timeShiftOperator_{\varphi}(\frequency)$ is not assumed to belong to $\unitCircle$.
    As $e^{i\indexTime f_-( \frequency_{\textnormal{SP}}; \nu)} = e^{ij \frequency_{\textnormal{SP}}} \timeShiftOperator_{\varphi}(\frequency_{\textnormal{SP}})^{\indexTime}$, the only remarkable saddle points are those such that $\timeShiftOperator_{\varphi}(\frequency_{\textnormal{SP}})\in\unitCircle$, hence one first requirement by Thomée.
    Since we have that $f_-( \frequency; \nu) = -i\log(\timeShiftOperator_{\varphi}(\frequency)) + \nu\xi$, we deduce that $i f_-^{(k)}( \frequency; \nu) = (\log(\timeShiftOperator_{\varphi}(\frequency)))^{(k)}$ for $k\geq 2$.
    Therefore
    \begin{equation*}
        \log(\timeShiftOperator_{\varphi}(\frequency))
        =
        \log(\timeShiftOperator_{\varphi}(\frequency_{\textnormal{SP}}))
        -i\nu (\frequency - \frequency_{\textnormal{SP}})
        +\sum_{k\geq 2}
        \frac{i}{k!}
        f_-^{(k)}( \frequency_{\textnormal{SP}}; \nu) (\frequency - \frequency_{\textnormal{SP}})^k,
    \end{equation*}
    thus
    \begin{equation*}
        \timeShiftOperator_{\varphi}(\frequency)
        = 
        \timeShiftOperator_{\varphi}(\frequency_{\textnormal{SP}})
        \textnormal{exp}
        \Bigl ( 
        -i\nu (\frequency - \frequency_{\textnormal{SP}})
        +\sum_{k\geq 2}
        \frac{1}{k!}
        \bigl ( - \textnormal{Im}(f_-^{(k)}( \frequency_{\textnormal{SP}}; \nu)) + i\textnormal{Re}(f_-^{(k)}( \frequency_{\textnormal{SP}}; \nu))\bigr ) (\frequency - \frequency_{\textnormal{SP}})^k
        \Bigr ).
    \end{equation*}
    Then, the assumption in the work by Thomée, which describes (local) dissipativity, is that 
    \begin{align}
    \overline{k}
    =
        \min\{ k\geq 2 \quad \textnormal{such that}
        \quad 
        f_-^{(k)}( \frequency_{\textnormal{SP}}; \nu)\neq 0
        \}\quad \textnormal{is even and}\label{eq:thomee1}
        \\ 
        \textnormal{Im}(f_-^{(\overline{k})}( \frequency_{\textnormal{SP}}; \nu))>0.\label{eq:thomee2}
    \end{align}
    \begin{itemize}
        \item For the leap-frog scheme, the non-degenerate saddle points fulfill \eqref{eq:thomee1} since $\overline{k} = 2$ but fail with \eqref{eq:thomee2}. This is due to the fact that $f_-\in\reals$ as the symbol identically belongs to $\unitCircle$. 
        Degenerate saddle points already fail concerning \eqref{eq:thomee1}, since $\overline{k} = 3$.
        \item For a dissipative scheme such as that of \Cref{sec:dissipative} on $\relatives$, the only saddle point of $f_-$ would be for $\nu=\courantNumber$ at $\frequency_{\textnormal{SP}} = 0$ and regular. This entails that $\overline{k} = 2$, hence \eqref{eq:thomee1} is fulfilled. We also have 
        \begin{equation*}
            f_-''(0; \courantNumber)
            =
            2i(1-\courantNumber^2)
            \Bigl (
            \frac{1}{\relaxationParameter}-
            \frac{1}{2}
            \Bigr ),
            \qquad
            \text{hence \eqref{eq:thomee2} is met as well.}
        \end{equation*}
    \end{itemize}
    Thus, we see that degeneracy/non-degeneracy of saddle points owes some degree of connection with the assumptions by Thomée concerning (local) dissipativity around frequencies for which the symbol belongs to $\unitCircle$.
\end{remark}

\begin{remark}[Link between saddle point problems on $\relatives$ and $\naturals$]\label{rem:sameSP}
    The saddle points that we now find without boundary \eqref{eq:spWithoutBoundary} coincide with the phase of the values of $\stableRoot$ on the saddle points in the case with boundary, see \eqref{eq:ksSaddlePoints}.
    We equally have that $\vartheta_{\varphi}(\frequency_{\textnormal{SP}})$ coincides with the arguments of the saddle points found in \eqref{eq:saddlePointsPhase}.
    This indicates that when treating the case with boundary, we were not far from the setting of Fourier analysis that can be deployed on $\relatives$.
    
    Furthermore, $f(\timeShiftOperator; \nu) \in \complex$ (for $\naturals$) and $i f_{\pm}(\frequency; \nu) \in i\reals$ (for $\relatives$) coincide on the saddle points up to a phase shift of $\pi$.
    The main difference between the Fourier setting in $\relatives$ and the situation on $\naturals$ is that $i f_{\pm}''(\frequency_{\textnormal{SP}}; \nu) \in i\reals$, whereas generally $f''(e^{i\saddlePointAngle}; \nu) \in \complex\smallsetminus i\reals$.
    Overall, the facts highlighted above show that applying  \Cref{lemma:saddlePoints} on degeneracy/non-degeneracy on $\naturals$ provides information sharing similarities to the Taylor expansions of the logarithms of the associated symbols \cite{thomee1965stability}, see \Cref{rem:linkwithThomee}.
\end{remark}

A (simpler) equivalent of \Cref{prop:asymptoticL2Stable} in the boundary-less setting is as follows.
\begin{proposition}[$L^2$ norm asymptotically constant]\label{prop:asymptoticL2StableGreen}
     Let $|\courantNumber|<1$.
     Then 
     \begin{equation*}
         \lim_{\indexTime\to+\infty}
         \lVert \secondGreenFunction^{\indexTime}\rVert_2
         =
         \lim_{\indexTime\to+\infty}
         \Bigl ( \sum_{\indexSpace\in\relatives}|\secondGreenFunction_{\indexSpace}^{\indexTime}|^2
         \Bigr )^{1/2}
         =
         \frac{1}{\sqrt{2}({1-\courantNumber^2})^{1/4}}.
     \end{equation*}
     Moreover, we have that for large $\indexTime\gg 1$, $\lVert \secondGreenFunction^{\indexTime}\rVert_2 \sim \lim_{k\to+\infty}
         \lVert \secondGreenFunction^{k}\rVert_2 + \bigO{\indexTime^{-1/2}}$.
\end{proposition}
\begin{proof}
By the Parseval equality, we have $\lVert \secondGreenFunction^{\indexTime}\rVert_2^2 = \int_{-\pi}^{\pi} |\fourierTransformed{\secondGreenFunction}^{\indexTime}(\frequency)|^2 \differential{\frequency}$.
This gives 
\begin{align}
    \lVert \secondGreenFunction^{\indexTime}\rVert_2^2
    &=
    \frac{1}{2\pi}
    \Bigl ( 
    \int_{-\pi}^{\pi}
    \frac{1}{2 (1-\courantNumber^2\sin^2(\frequency))}
    \differential{\frequency}
    +(-1)^{\indexTime + 1} 
    \textnormal{Re}\Bigl ( 
    \int_{-\pi}^{\pi}
    \frac{\textnormal{exp}(-2 i \indexTime\arcsin(\courantNumber\sin(\frequency)))}{2 (1-\courantNumber^2\sin^2(\frequency))}
    \differential{\frequency}
    \Bigr )
    \Bigr ) \label{eq:parseval}\\
    &=
    \frac{1}{2\pi}
    \Bigl ( 
    \frac{\pi}{\sqrt{1-\courantNumber^2}}
    +\bigO{\indexTime^{-1/2}}
    \Bigr )
    \Bigr ),
\end{align}
where the estimate on the second integral comes from the stationary phase approximation \eqref{eq:stationaryPhaseApproximation}.
Indeed, the phase in the second integral equals $2 f_-(\frequency; 0)$, which admits two non-degenerate (glancing) saddle points $\frequency_{\textnormal{SP}} = \pm\tfrac{\pi}{2}$ according to \Cref{lemma:saddlePointsGreen}. 
\end{proof}

\begin{remark}[Comparison to ``energy estimates'']
    This asymptotic value has to be compared with the estimate by ``energy method'' in \cite[Section 2.2.3]{coulombel00616497}, which reads
    \begin{equation*}
        \lVert \secondGreenFunction^{2\indexTime}\rVert_2^2 + 
        \lVert \secondGreenFunction^{2\indexTime + 1}\rVert_2^2
        \leq 
        \frac{1+|\courantNumber|}{1-|\courantNumber|},
    \end{equation*}
    where the limit of the left-hand side is 
    \begin{equation*}
        \frac{1}{\sqrt{1-\courantNumber^2}}
         \leq {\frac{1+|\courantNumber|}{1-|\courantNumber|}}.
    \end{equation*}
\end{remark}

\begin{proposition}[Transition zone]\label{prop:greenTransition}
    Let $|\courantNumber|<1$.
    Let $-|\courantNumber|\indexTime<\indexSpace<|\courantNumber|\indexTime$.
    Then, in the limit $\indexTime\gg 1$, $\secondGreenFunction_{\indexSpace}^{\indexTime}$ is well approximated by
    \begin{align*}
        \secondGreenFunction_{\indexSpace}^{\indexTime}
        \sim 
        \frac{1}{\sqrt{2\pi}} |\courantNumber|^{-1/2}&(1-\courantNumber^2)^{-1/4}
    \overbrace{\Bigl (1-\Bigl ( \frac{\indexSpace}{\courantNumber\indexTime}\Bigr )^2 \Bigr )^{-1/4}}^{\text{envelope}}\\
        \times\Biggl (
            &\cos\Bigl ( \indexTime\arcsin\Bigl (\courantNumber\sqrt{\frac{1-(\frac{\indexSpace}{\courantNumber\indexTime})^2}{1-(\frac{\indexSpace}{\indexTime})^2}}\Bigr ) - \indexSpace\arccos\Bigl (\frac{\indexSpace}{\courantNumber\indexTime}\sqrt{\frac{1-\courantNumber^2}{1-(\frac{\indexSpace}{\indexTime})^2}}\Bigr ) - \frac{\pi}{4}\textnormal{sgn}(\courantNumber)\Bigr )\\
            -(-1)^{\indexTime}
            &\cos\Bigl ( \indexTime\arcsin\Bigl (\courantNumber\sqrt{\frac{1-(\frac{\indexSpace}{\courantNumber\indexTime})^2}{1-(\frac{\indexSpace}{\indexTime})^2}}\Bigr ) + \indexSpace\arccos\Bigl (-\frac{\indexSpace}{\courantNumber\indexTime}\sqrt{\frac{1-\courantNumber^2}{1-(\frac{\indexSpace}{\indexTime})^2}}\Bigr ) - \frac{\pi}{4}\textnormal{sgn}(\courantNumber)\Bigr )
        \Biggr )\indexTime^{-1/2}
        +\bigO{\indexTime^{-3/2}}.
    \end{align*}
\end{proposition}
A similar result is stated for $\firstGreenFunction_{\indexSpace}^{\indexTime} + \secondGreenFunction_{\indexSpace}^{\indexTime}$ in \cite[Section 5.3]{bouche2024analyse}.
We note, however, that the quantity $\firstGreenFunction_{\indexSpace}^{\indexTime} + \secondGreenFunction_{\indexSpace}^{\indexTime}$ is not meaningful in actual numerical simulation, for which knowledge of each individual Green function is needed.

\begin{remark}[Glancing modes and slow convergence of the $L^2$ norm]
    By \cite[Proposition A.1]{coulombel2015fully}, we know that if 
    \begin{equation*}
        \sum_{\indexTime\geq 0}|\secondGreenFunction_{0}^{\indexTime}|^2
        \qquad \text{is finite},
    \end{equation*}
    then necessarily there are no glancing modes.
    From \Cref{prop:greenTransition}, we obtain as detailed in Appendix \ref{app:L2TimeGreen} that    
    \begin{equation*}
        \sum_{\indexTime = 0}^{N}|\secondGreenFunction_{0}^{\indexTime}|^2
        \qquad \text{diverges with }N\text{ at rate}
        \quad 
        \frac{1}{\pi |\courantNumber| \sqrt{1-\courantNumber^2}}
        \ln(N).
    \end{equation*}
    This confirms that we face glancing modes, which we have previously identified.
    These glancing saddle points generate a slow decrease of $\secondGreenFunction_0^{\indexTime}$ in $\indexTime$, causing the previous series to diverge (although slowly) analogously to the harmonic series.
    
    This is also the cause of the fact that the $L^2$ norm of $\indexSpace\mapsto \secondGreenFunction_{\indexSpace}^{\indexTime}$ converges to its limit at the (slow) rate of $\bigO{\indexTime^{-1/2}}$.
    Indeed, the phase in the second integral in \eqref{eq:parseval}---responsible for the speed of convergence---admits (glancing) saddle points at $\pm\frac{\pi}{2}$, see the proof of \Cref{prop:asymptoticL2StableGreen}.
    
    Conversely, this can be compared to the case with stable boundary conditions, \confer{} \Cref{sec:stable}.
    Indeed, $\globalTruncationError_{\indexSpace}^{\indexTime}$ can be regarded as a second Green function attached/relative to the boundary.
    There, since we had $\globalTruncationError_{0}^{\indexTime} = \bigO{\indexTime^{-3/2}}$, we obtain 
    \begin{equation*}
        \sum_{\indexTime\geq 0}|\globalTruncationError_{0}^{\indexTime}|^2 < \infty
        \qquad \text{by comparison with} \qquad
        \sum_{\indexTime\geq 1}\frac{1}{\indexTime^3} < +\infty,
    \end{equation*}
    thanks to the absence of glancing saddle points, which secures a sufficiently rapid damping of the solution in the near-wall region.
    Moreover, for the same reason, the $L^2$ norm of $\indexSpace\mapsto\globalTruncationError_{\indexSpace}^{\indexTime}$ converges to its limit at the faster rate $\bigO{\indexTime^{-3/2}}$.
    Somehow, the ``stiffness'' induced by boundary conditions---regardless of their stability---absorbs the glancing saddle points which are present for the leap-frog scheme without boundary.
\end{remark}

The next result is stated for the sum of first and second Green functions in  \cite{bouche2024analyse}, only with a formal justification based on truncated modified equations.
\begin{proposition}[Front zones]\label{prop:frontGreenFunctions}
    Let $-1<\courantNumber<0$ without loss of generality, and $\indexTime\gg 1$.
    \begin{itemize}
        \item \strong{Spurious front} (group velocity $-\courantNumber$): let $\indexSpace\in\naturals$ such that $\indexSpace+\courantNumber\indexTime=\bigO{1}$. Then $\secondGreenFunction_{\indexSpace}^{\indexTime}$ is well approximated by
        \begin{equation*}
            \secondGreenFunction_{\indexSpace}^{\indexTime} \sim 
            (-1)^{\indexTime + 1}\left|\frac{(-1)^{\indexTime} - (-1)^{\indexSpace}}{2}\right|\frac{1}{(\frac{\courantNumber}{2}(\courantNumber^2-1) \indexTime)^{1/3}} \textnormal{Ai}\Biggl ( \frac{\indexSpace + \courantNumber\indexTime}{(\frac{\courantNumber}{2}(\courantNumber^2-1) \indexTime)^{1/3}} \Biggr ).
        \end{equation*}
        \item \strong{Physical front} (group velocity $\courantNumber$): let $-\indexSpace\in\naturals$ such that $\indexSpace-\courantNumber\indexTime=\bigO{1}$. Then $\secondGreenFunction_{\indexSpace}^{\indexTime}$ is well approximated by
        \begin{equation*}
            \secondGreenFunction_{\indexSpace}^{\indexTime} \sim 
            \left|\frac{(-1)^{\indexTime} - (-1)^{\indexSpace}}{2}\right|\frac{1}{(\frac{\courantNumber}{2}(\courantNumber^2-1) \indexTime)^{1/3}} \textnormal{Ai}\Biggl (- \frac{\indexSpace - \courantNumber\indexTime}{(\frac{\courantNumber}{2}(\courantNumber^2-1) \indexTime)^{1/3}} \Biggr ).
        \end{equation*}
    \end{itemize}
\end{proposition}

\Cref{prop:greenTransition} and \ref{prop:frontGreenFunctions} are proved analogously to the proofs in \Cref{sec:proofsLeapFrog}, using \eqref{eq:stationaryPhaseApproximation} for the former and an analogous procedure to \Cref{prop:frontZoneStable} for the latter.
They achieve a full characterization of the Green functions in the long-time regime.

\section{Conclusions}\label{sec:Conclusions}

In this work, we have provided an accurate description of the long-time behavior of some multi-step Finite Difference schemes in presence of boundaries. 
By leveraging the framework of the $\timeShiftOperator$-transform, we have reformulated the numerical scheme in complex variables, thereby enabling a robust characterization of the structure of particular solutions---those linked to (second) Green functions relative to the boundary---by means of steepest descent techniques and analytic combinatorics. 
Such approach has been applied for the well established leap-frog scheme, as a canonical representative of the class of non-dissipative bulk numerical schemes, in presence of stable and unstable boundary conditions, and for a dissipative bulk scheme, in presence of stable boundary conditions. 
In both instances, the theoretical analysis yielded the leading-order behaviors, which have been validated through numerical simulations.
Notice that the solution of the leap-frog scheme is peculiarly rich in different structures as the scheme allows non-damped saddle points associated with any group velocity between zero and (minus) the Courant number.
Our findings explain why---in many simulations in \cite{bellotti2025stability}---structures propagating at speed $-\courantNumber>0$ were the most visible one. Indeed, they correspond to degenerate saddle-points and therefore exhibit the slowest damping in time.

Finally, the same strategies have been deployed for the analysis of the Green functions of the leap-frog scheme on $\relatives$, \idEst{} without boundary conditions, showing similarities with the case with a boundary, despite a simpler setting thanks to the availability of the Fourier transform.

\section*{Acknowledgements}
Tommaso Tenna received funding from the European Union's Horizon Europe research and innovation program under the Marie Skłodowska-Curie Doctoral Network DataHyking (Grant No. 101072546). Tommaso Tenna is member of GNCS-INdAM research group.

Thomas Bellotti thanks Benjamin Boutin (Université de Rennes) for discussions on glancing points and related references.

\bibliographystyle{acm}
\bibliography{biblio}

\appendix

\section{Explicit expression of $\globalTruncationError_{\indexSpace}^{\indexTime}$ for \Cref{ex:upwind}}\label{app:explicitExpressionUpwind}

The proof of the following result is based on the repeated use of the generalized binomial theorem and Neumann series, and not provided for the sake of room.
\begin{proposition}[Explicit expression of $\globalTruncationError_{\indexSpace}^{\indexTime}$ for \Cref{ex:upwind}]\label{prop:explicitExpressionError}
    Let the coefficient linked to the boundary scheme in \Cref{ex:upwind} be
    \begin{align*}
        \beta_1 = 1+\courantNumber, \qquad 
        \beta_2 = -\courantNumber^2, \qquad 
        \beta_{2r} = -\frac{1}{2}\sum_{p = \lfloor\frac{r + 1}{2}\rfloor}^{r}\binom{1/2}{p} \binom{p}{r-p} (2(2\courantNumber^2-1))^{2p-r}&, \\
        \text{and}\quad
        \beta_{2r - 1} = 0&, \qquad r\geq 2.
    \end{align*}
    Let \Cref{ass:stableBulk} hold.
    Then, an explicit expression of $\globalTruncationError_{\indexSpace}^{\indexTime}$, solution of \eqref{eq:zeroInitialError}--\eqref{eq:schemeErrorInitial}--\eqref{eq:schemeErrorEventual}, is given by  
    \begin{multline}\label{eq:nonRecurrentExpressionError}
        \globalTruncationError_{\indexSpace}^{\indexTime} = \sum_{k = \indexSpace}^{\lfloor \frac{\indexTime+\indexSpace-1}{2}\rfloor }\Biggl ( \sum_{\ell = 0}^{\indexSpace} \sum_{s = 0}^{\indexSpace-\ell} \sum_{p = \lfloor\frac{k-s + 1}{2}\rfloor}^{k-s} \frac{(-1)^{\indexSpace-\ell-s}}{(2\courantNumber)^{\indexSpace}}\binom{\indexSpace}{\ell}  \binom{\indexSpace-\ell}{s}  \binom{\ell/2}{p} \binom{p}{k-s-p} (2(2\courantNumber^2-1))^{2p-k+s} \Biggr ) \\
        \times \Biggl ( \sum_{r = 1}^{\indexTime+\indexSpace-2k-1}\sum_{\alpha_1 + \dots+\alpha_{r} = \indexTime+\indexSpace-2k-1} \beta_{\alpha_1}\times \cdots\times \beta_{\alpha_r} \Biggr ).
    \end{multline}
\end{proposition}

\section{Leftover proofs}

\begin{proof}[Proof of \Cref{lemma:saddlePoints}]
We first observe that 
\begin{equation*}
    f'(\timeShiftOperator; \nu) = \frac{1}{\timeShiftOperator} + \nu \frac{\stableRoot'(\timeShiftOperator)}{\stableRoot(\timeShiftOperator)}
    \qquad \text{and}\qquad
    f''(\timeShiftOperator; \nu) = -\frac{1}{\timeShiftOperator^2} + \nu \frac{\stableRoot''(\timeShiftOperator)\stableRoot(\timeShiftOperator)-(\stableRoot'(\timeShiftOperator))^2}{\stableRoot(\timeShiftOperator)^2}.
\end{equation*}
Using the explicit expression of $\stableRoot$, the saddle points satisfy $f'(\timeShiftOperator, \nu) = 0$, which becomes
\begin{equation}\label{eq:spEquationFraction}
    \frac{{\left(\nu + 1\right)} \timeShiftOperator^{4} + 2  {\left(2  {\courantNumber}^{2} - 1\right)} \timeShiftOperator^{2} - \sqrt{\timeShiftOperator^{4} + 2  {\left(2  {\courantNumber}^{2} - 1\right)} \timeShiftOperator^{2} + 1} {\left({\left(\nu + 1\right)} \timeShiftOperator^{2} + \nu - 1\right)} - \nu + 1}{\timeShiftOperator^{5} + 2  {\left(2  {\courantNumber}^{2} - 1\right)} \timeShiftOperator^{3} - \sqrt{\timeShiftOperator^{4} + 2  {\left(2  {\courantNumber}^{2} - 1\right)} \timeShiftOperator^{2} + 1} {\left(\timeShiftOperator^{3} - \timeShiftOperator\right)} + \timeShiftOperator} = 0.
\end{equation}
For $\nu = 0$, the previous equation is $\timeShiftOperator^{-1} = 0$, which has no finite solution.
Let us assume $\nu > 0$ in the remainder of the proof.
The denominator in \eqref{eq:spEquationFraction} vanishes for $\timeShiftOperator = 0$ and for $\timeShiftOperator$ such that $\timeShiftOperator^4 + 2(2\courantNumber^2 - 1)\timeShiftOperator^2 + 1=0$, that is on the branch points by \Cref{lemma:BranchPoints}.
The numerator vanishes for $\timeShiftOperator$ such that $\varphi_4(\timeShiftOperator)\definitionEquality (1-\nu^2)\timeShiftOperator^4 + 2(2\courantNumber^2-\nu^2-1)\timeShiftOperator^2 + 1-\nu^2 = 0$. 
\begin{itemize}
    \item Let $\nu\in (0, |\courantNumber|)$.
    One can use the results in \cite[Chapter 4]{strikwerda2004finite} to show that the zeros of $\varphi_4$ are on $\unitCircle$.
    Straightforward computations provide that these zeros are
    \begin{equation*}
            e^{\pm i \saddlePointAngle}
            \quad \text{and}\quad
            e^{\pm i (\saddlePointAngle\mp \pi)}
            , 
            \qquad \text{where}\qquad
            \saddlePointAngle 
            \definitionEquality
            \tfrac{1}{2}
            \arccos
            \Bigl( 
            \frac{1+\nu^2 - 2\courantNumber^2}{1-\nu^2}
            \Bigr )
            \in (0, \branchPointAngle),
        \end{equation*}
    hence are distinct from the branch points by \Cref{lemma:BranchPoints}, which make the denominator of \eqref{eq:spEquationFraction} vanish.
    This shows that $e^{\pm i \vartheta}$ and $e^{\pm i (\saddlePointAngle\mp \pi)}$ are the sought saddle points points, since they fulfill \eqref{eq:spEquationFraction}.
    
    Equation \eqref{eq:ksSaddlePoints} follows from some algebra, and $f'(e^{\pm i \saddlePointAngle}; \nu), f'(e^{\pm i (\saddlePointAngle\mp \pi)}; \nu) \in i\reals$ follows from the definition of principal determination of the complex logarithm.
        
    We conclude this case on the non-degeneracy of these saddle points.
    Let $\saddlePoint\in\unitCircle$ be one of the four previously identified saddle points and assume degeneracy, that is $f'(\saddlePoint, \nu) = f''(\saddlePoint; \nu) = 0$.
    We obtain $\saddlePoint^{-1}f'(\saddlePoint; \nu) + f''(\saddlePoint; \nu) = 0$, which becomes
    \begin{equation*}
        \frac{\stableRoot'(\saddlePoint)}{\saddlePoint} + \frac{\stableRoot''(\saddlePoint)\stableRoot(\saddlePoint)-(\stableRoot'(\saddlePoint))^2}{\stableRoot(\saddlePoint)} = 0.
    \end{equation*}
    This is a fraction whose denominator cannot vanish as $\saddlePoint, \stableRoot(\saddlePoint)\in\unitCircle$.
    Quite the opposite, the numerator vanishes if $(\courantNumber^2-1)^2\courantNumber^4\saddlePoint^4(\saddlePoint^2-1)^2(\saddlePoint^4+2(2\courantNumber^2-1)\saddlePoint^2 + 1)$ does.
    This only happens if $\saddlePoint = \pm 1$, which is not possible as $\saddlePointAngle>0$.
    The non-degeneracy makes a smooth transition to the following case, where degeneracy takes place.

    \item Let $\nu = |\courantNumber|$. This can be discussed simply by taking the limit of the previous case for $\nu\nearrow-\courantNumber$.
    
    \item Let $\nu \in (|\courantNumber|, 1)$.
    In this case, it is easily shown that the saddle points are real, since they are given by
    \begin{equation}\label{eq:realSaddlePoints}
        \saddlePoint = (-1)^{\alpha} \sqrt{\frac{2\courantNumber^2-\nu^2-1 +(-1)^{\beta} \sqrt{(\courantNumber^2-1)(\courantNumber^2-\nu^2)} }{\nu^2-1}}, 
        \qquad \alpha, \beta \in \{0, 1\}.
    \end{equation}
    Two of them lie in $\unitDisk$ ($\beta = 0$) whereas the other two are in $\neighborhoodInfinity$ ($\beta = 1$).
\end{itemize}
\end{proof}

\section{Study of the diverging series $\sum_{\indexTime\geq 0}|\secondGreenFunction_{0}^{\indexTime}|^2$}\label{app:L2TimeGreen}

Let us investigate $\sum_{\indexTime\geq 0}|\secondGreenFunction_{0}^{\indexTime}|^2$ and its partial sums, to demonstrate that the series diverges and at which speed. 
From \Cref{prop:greenTransition}, we obtain that    
    \begin{equation*}
        \secondGreenFunction_{0}^{\indexTime}
        = 
        \frac{1}{\sqrt{2\pi}} |\courantNumber|^{-1/2}(1-\courantNumber^2)^{-1/4} (1-(-1)^{\indexTime})\cos\Bigl ( \indexTime\arcsin (|\courantNumber|) - \frac{\pi}{4}\textnormal{sgn}(\courantNumber)\Bigr )\indexTime^{-1/2}
        +\bigO{\indexTime^{-3/2}}.
    \end{equation*}
    Thus, we have
    \begin{equation*}
        \sum_{\indexTime\geq 0}|\secondGreenFunction_{0}^{\indexTime}|^2
        = 
        C\sum_{\substack{\indexTime\text{ odd}\\\indexTime\geq 1}}
        \cos^2 ( \indexTime \arcsin (|\courantNumber|) - \tfrac{\pi}{4}\textnormal{sgn}(\courantNumber) ) \,
        \frac{1}{\indexTime}
        +
        \bigO{1}
    \end{equation*}
    Let us focus on the leading order term, which can be rewritten as
    \begin{equation*}
        \sum_{\substack{\indexTime\text{ odd}\\\indexTime\geq 1}}
        \cos^2 ( \indexTime \arcsin (|\courantNumber|) - \tfrac{\pi}{4}\textnormal{sgn}(\courantNumber) ) \,
        \frac{1}{\indexTime} = \frac{1}{2} \sum_{\indexTime 
        \geq 1} (1-(-1)^{\indexTime}) \cos^2 ( \indexTime \arcsin (|\courantNumber|)- \tfrac{\pi}{4}\textnormal{sgn}(\courantNumber) ) \,
        \frac{1}{\indexTime}.
    \end{equation*}
    Let $N \geq 1$ and let us consider
    \begin{multline*}
        \frac{1}{2} \sum_{\indexTime = 1}^N (1-(-1)^{\indexTime}) \cos^2 ( \indexTime \arcsin (|\courantNumber|) - \tfrac{\pi}{4}\textnormal{sgn}(\courantNumber) ) \, \frac{1}{\indexTime} \\
        = \frac{1}{2} \sum_{\indexTime = 1}^N (1-(-1)^{\indexTime}) \Bigl ( 1 + \underbrace{\cos (2\indexTime \arcsin (|\courantNumber|) - \tfrac{\pi}{2}\textnormal{sgn}(\courantNumber) )}_{(-1)^{\textnormal{sgn}(\courantNumber)} \sin(2\indexTime\arcsin (|\courantNumber|)} \Bigr ) \, \frac{1}{\indexTime} \\
        = \frac{1}{2} \Biggl ( \sum_{\substack{\indexTime = 1 \\ \indexTime \textnormal{ odd}} }^N \underbrace{\frac{1}{\indexTime}}_{a_\indexTime} + (-1)^{\textnormal{sgn}(\courantNumber)} \sum_{\indexTime=1}^N \underbrace{\frac{(1-(-1)^{\indexTime})  \sin(2\indexTime\arcsin (|\courantNumber|)}{\indexTime}}_{a_\indexTime b_\indexTime} \Biggr ).
    \end{multline*}
    Since the first sum is a divergent series for $N \to \infty$, the goal is showing that the trigonometric series $\sum_{\indexTime} a_\indexTime b_\indexTime$ strictly converges. By the algebraic properties of the infinite series, the sum of a divergent series and a convergent series must diverge, see \cite[Chapter 3]{rudin1976}. Consequently, establishing the convergence for the second series is sufficient to prove that $\sum_{\indexTime\geq 0}|\secondGreenFunction_{0}^{\indexTime}|^2$ diverges.\\
    To this aim, using the formula in \cite{knapp2009sines} on sines of angles in arithmetic progression, we obtain
    \begin{equation*}
        \sum_{\indexTime=1}^N b_\indexTime = 2 \sum_{k=0}^{\lfloor \tfrac{N-1}{2} \rfloor} 2 \sin((4k+2) \arcsin(|\courantNumber|)) 
        = 2 \frac{\sin^2\left(2(\lfloor \tfrac{N-1}{2} \rfloor + 1)\arcsin(|\courantNumber|)) \right)}{\sin(2\arcsin(|\courantNumber|))} \leq  \frac{2}{\sin(2\arcsin(|\courantNumber|))}.
    \end{equation*}
    The previous formula can be written since $0 < |\courantNumber| < 1$, then $\arcsin(|\courantNumber|) \in (0, \tfrac{\pi}{2})$ and $\sin(2\arcsin(|\courantNumber|)) > 0$. Furthermore, this entails that the partial sums of $b_n$ are bounded and by Dirichlet's test we conclude---see \cite[Theorem 3.42]{rudin1976}.
    
    Moreover, the fact that the partial sums of the harmonic series with odd terms diverge at rate $\tfrac{1}{2}\ln(N)$ entails that 
    \begin{equation*}
        \sum_{\indexTime = 0}^{N}|\secondGreenFunction_{0}^{\indexTime}|^2
        \qquad \text{diverges with }N\text{ at rate}
        \quad 
        \frac{1}{\pi |\courantNumber| \sqrt{1-\courantNumber^2}}
        \ln(N).
    \end{equation*}

\end{document}